\documentclass[11pt,a4paper]{article}

\usepackage[utf8]{inputenc}
\usepackage[T1]{fontenc}
\usepackage{lmodern}
\usepackage[margin=2.3cm]{geometry}
\usepackage{amsmath,amssymb,amsthm,mathtools}
\usepackage{aliascnt}
\usepackage{array}
\usepackage{graphicx}
\usepackage{subcaption}
\usepackage{placeins}
\usepackage{xcolor}
\usepackage{enumitem}
\usepackage{hyperref}
\usepackage{cleveref}

\setlist[enumerate]{leftmargin=*,itemsep=2pt,topsep=4pt}
\newcolumntype{L}[1]{>{\raggedright\arraybackslash}p{#1}}

\theoremstyle{plain}
\newtheorem{theorem}{Theorem}[section]
\newaliascnt{proposition}{theorem}
\newtheorem{proposition}[proposition]{Proposition}
\aliascntresetthe{proposition}
\newaliascnt{lemma}{theorem}
\newtheorem{lemma}[lemma]{Lemma}
\aliascntresetthe{lemma}
\newaliascnt{corollary}{theorem}
\newtheorem{corollary}[corollary]{Corollary}
\aliascntresetthe{corollary}

\theoremstyle{plain}
\newaliascnt{definition}{theorem}
\newtheorem{definition}[definition]{Definition}
\aliascntresetthe{definition}
\theoremstyle{definition}
\newaliascnt{assumption}{theorem}
\newtheorem{assumption}[assumption]{Assumption}
\aliascntresetthe{assumption}

\theoremstyle{definition}
\newaliascnt{remark}{theorem}
\newtheorem{remark}[remark]{Remark}
\aliascntresetthe{remark}

\newcommand{\T}{\mathbb{T}}
\newcommand{\Z}{\mathbb{Z}}
\newcommand{\R}{\mathbb{R}}
\newcommand{\E}{\mathbb{E}}
\newcommand{\Pbb}{\mathbb{P}}
\newcommand{\cN}{\mathcal{N}}
\newcommand{\cL}{\mathcal{L}}
\newcommand{\cF}{\mathcal{F}}
\newcommand{\cE}{\mathcal{E}}
\newcommand{\Tr}{\mathrm{tr}}
\newcommand{\Ent}{\mathrm{Ent}}
\newcommand{\RelEnt}[2]{\Ent\!\left(#1\,\middle|\,#2\right)}

\crefname{equation}{eq.}{eqs.}
\crefname{theorem}{Theorem}{Theorems}
\crefname{proposition}{Proposition}{Propositions}
\crefname{lemma}{Lemma}{Lemmas}
\crefname{corollary}{Corollary}{Corollaries}
\crefname{remark}{Remark}{Remarks}
\crefname{definition}{Definition}{Definitions}
\crefname{assumption}{Assumption}{Assumptions}

\hypersetup{
    colorlinks=true,
    linkcolor=blue,
    citecolor=blue,
    urlcolor=blue,
}

\begin{document}

\title{\bfseries Relaxation and Steady-State Entropy Production\\[3pt] for Langevin SPDEs: A Dirichlet-Form Approach}

\author{Shuyuan Fan \quad Yuanke Chen \quad Lifei Wang \quad Jinqiao Duan}

\date{September 14, 2026}

\maketitle

\begingroup
\renewcommand{\thefootnote}{}
\begin{NoHyper}
\footnotetext{\scriptsize\raggedright E-mail addresses: Shuyuan Fan, \texttt{fanshy9@mail2.sysu.edu.cn}, \texttt{233621002@stu.gbu.edu.cn}; Yuanke Chen, \texttt{chenyk37@mail2.sysu.edu.cn}, \texttt{233621001@stu.gbu.edu.cn}; Lifei Wang, \texttt{wanglf86@hebtu.edu.cn}; Jinqiao Duan, \texttt{duan@gbu.edu.cn}.\par}
\end{NoHyper}
\endgroup

\begin{abstract}
We study relative-entropy dissipation and steady-state entropy production for preconditioned Langevin stochastic partial differential equations through their Dirichlet forms. In infinite dimensions neither a Lebesgue density nor classical integration by parts is available; the Dirichlet-form framework replaces both. For conservative symmetric forms with diffusion calculus, we establish an exact entropy-dissipation relation for regular densities and an integrated inequality for arbitrary initial laws of finite relative entropy. For continuous Markov dynamics generated by these forms, the stationary path law is invariant under time reversal and its forward--reverse relative entropy vanishes. Applied to the one-dimensional $\Phi^4_1$ Gibbs dynamics, this approach gives exponential entropy relaxation through a log-Sobolev estimate uniform over the finite-dimensional approximations. For Gaussian dynamics with bounded skew-adjoint linear forcing, we separate the semigroup into symmetric relaxation and a Gaussian-measure-preserving flow, so that relative entropy follows exactly the same evolution as for the symmetric dynamics started from the same density. Following a mass quench, we obtain an infinite-dimensional Onsager decomposition into entropy dissipation and circulation energy without requiring commutation of the force and covariance operators. At stationarity, the circulation energy equals the relative-entropy rate between the forward and time-reversed laws on the full path space. Numerical illustrations show relaxation in a Fourier approximation of the quartic field and the separation of relaxation from circulation in a two-mode Gaussian example.

\end{abstract}

\setcounter{tocdepth}{2}
\tableofcontents

\section{Introduction}
\label{sec:intro}

Relaxation towards a stationary law and the persistence of probability currents at stationarity are two central themes in stochastic thermodynamics. At the mesoscopic scale, thermodynamic systems exhibiting these effects are often modelled by Markov processes. In finite-dimensional models, relative entropy with respect to the stationary law quantifies relaxation from an initial distribution, while stationary currents and the comparison of forward and time-reversed trajectories reveal irreversibility that remains at stationarity \cite{VdB2010,Seifert2012,LebowitzSpohn1999,Qian2013}. The purpose of this paper is to understand these mechanisms for spatially extended systems whose state is a fluctuating field and whose evolution is governed by a Langevin SPDE.

For a smooth constant-mobility diffusion, let $\mu_t(dx)=p_t(x)\,dx$ and $\Pi(dx)=\pi(x)\,dx$ denote the time-$t$ and stationary laws. Under the usual integrability conditions, the entropy balance takes the form
\begin{equation*}
    \sigma_{\mathrm{tot}} = \sigma_{\mathrm{na}}+\sigma_{\mathrm{hk}},
    \qquad
    \frac{d}{dt}\RelEnt{\mu_t}{\Pi} = -\sigma_{\mathrm{na}}.
\end{equation*}
Here $\sigma_{\mathrm{tot}}$, $\sigma_{\mathrm{na}}$ and $\sigma_{\mathrm{hk}}$ denote the total, nonadiabatic and housekeeping rates, respectively \cite{EspositoVdB2010,Seifert2012}. The nonadiabatic rate measures relative-entropy dissipation, whereas the housekeeping rate can remain positive at stationarity. The calculation uses the evolution equation for $p_t$, integration by parts and the separation of the stationary current from the dissipative current. These are the operations that we seek to justify for a field.

Jiang, Qian and Qian develop a stationary diffusion framework that uses Dirichlet forms and identifies the time-reversed semigroup with the adjoint semigroup \cite[Chapter~3]{JiangQianQian2004}. Qian, Qian and Tang relate reversibility to generator symmetry and vanishing steady entropy production \cite{QianQianTang2002}, while Qian's symmetric--antisymmetric decomposition separates free-energy dissipation from stationary circulation \cite{Qian2013}. For jump diffusions, Fan and Zhang establish the corresponding generator decomposition and nonlocal free-energy balance \cite{FanZhang2026}, and Zhang and Lu obtain the stationary path-space formula and reversibility criteria \cite{ZhangLu2025}. In finite dimensions, the relation between irreversible drift and path-space entropy production is also studied in \cite{DaCostaPavliotis2023}.

We consider fields on the periodic interval $\T=\R/(2\pi\Z)$. On $\mathbb H=L^2(\T;\R)$, let $A=-\partial_x^2+m$ with $m>0$, set $Q=A^{-1}$ and denote the centred Gaussian law with covariance $Q$ by $\mu_Q$. Writing $\nabla_{\mathbb H}V$ for the variational $\mathbb H$-gradient when it exists, the preconditioned Langevin equation is
\begin{equation}\label{eq:main-intro}
    d\phi_t = -\bigl(\phi_t + A^{-1}\nabla_{\mathbb H}V(\phi_t) + A^{-1}f(\phi_t)\bigr)\,dt + \sqrt{2}\,dW^{A^{-1}},
\end{equation}
where $W^{A^{-1}}$ is a $Q$-Wiener process, the mobility is $Q$ and the noise covariance per unit time is $2Q$. For $f=0$, the candidate invariant law is the Gibbs measure $\kappa=Z^{-1}e^{-V}\mu_Q$, with $Z=\int e^{-V}\,d\mu_Q$; the quartic interaction gives the $\Phi^4_1$ model. For $V=0$, we consider linear forces $f(\phi)=B\phi$ with $B$ bounded and skew-adjoint on $\mathbb H$. Such a force preserves $\mu_Q$ but produces a stationary circulation when $B\ne0$.

Gibbs dynamics are of particular interest in stochastic quantisation, where Langevin evolution samples Euclidean field measures \cite{ParisiWu1981,DamgaardHuffel1987}. Constructive field theory and SPDE methods give rigorous constructions of these measures and dynamics \cite{GlimmJaffe1987,DaPratoDebussche2003,Hairer2014,GubinelliHofmanova2021}. The time in \eqref{eq:main-intro} is then an auxiliary sampling time; a thermodynamic interpretation in terms of heat additionally specifies temperature, mobility and local detailed balance \cite{Seifert2012}.

In infinite dimensions, the finite-dimensional Lebesgue-density calculation must be replaced by a calculation relative to a reference measure. For the Gaussian measure $\mu_Q$, the Cameron--Martin space $\mathbb K=Q^{1/2}\mathbb H$ consists precisely of the translation directions that preserve absolute continuity, and its norm determines the gradient energy used below \cite{Bogachev1998}. Gaussian absolute continuity and Sobolev regularity of invariant densities for perturbed Ornstein--Uhlenbeck dynamics were established in \cite{BogachevRoeckner1995,BogachevDaPratoRoeckner1996}. For nonsymmetric dissipative SPDEs, Da~Prato, Debussche and Goldys obtained integration by parts and a logarithmic Sobolev inequality for invariant measures \cite{DaPratoDebusscheGoldys2002}; Da~Prato and Debussche also derived absolute-continuity formulas comparing interacting and free invariant measures \cite{DaPratoDebussche2004}. The corresponding Kolmogorov-operator and invariant-measure frameworks are treated in \cite{DaPrato2004Kolmogorov,DaPratoZabczyk1996}.

For the entropy calculation, the evolving density $h_t=d\mu_t/d\kappa$ need not lie in the generator domain, and $\log h_t$ need not be an admissible test function. We therefore work with the closed quadratic form of the symmetric generator rather than require classical derivatives of $h_t$. The correspondence between a closed symmetric Dirichlet form and its Markov semigroup retains the integration-by-parts relation \cite{MaRoeckner1992}, while diffusion chain rules express entropy dissipation as the energy of $\sqrt{h_t}$ \cite{BakryGentilLedoux2014}. The lower semicontinuity of this energy controls limits of regular densities. Our abstract assumptions require these properties without specifying a Gaussian reference measure; for the Langevin models, the noise covariance $Q$ and the reference law $\mu_Q$ give the Cameron--Martin gradient that realises the form.

For conservative symmetric forms with diffusion calculus, \cref{thm:deBruijn} gives the exact entropy-dissipation relation for regular positive densities and an integrated inequality for arbitrary finite-entropy initial laws. For a continuous Markov process with the form semigroup, self-adjointness gives detailed balance and invariance of its stationary path law under time reversal, as established in \cref{thm:eq-zero}. The forward--reverse path-space relative entropy therefore vanishes on every finite observation interval. The finite-entropy extension of the dissipation relation uses truncation and positivity- and mass-preserving resolvent regularisation, followed by lower semicontinuity of entropy and the square-root energy; see \cref{lem:finite-entropy-extension}. Thus the initial law need not already have a square-root density in the form domain.

The $\Phi^4_1$ application requires an energy argument adapted to the preconditioned equation. Its free evolution is generated by $-I$, so it does not have the spatial smoothing of the heat semigroup, a feature analysed by Hairer, Stuart and Voss \cite{HairerStuartVoss2007}. Their results provide global strong well-posedness and Gibbs-semigroup symmetry for the quartic dynamics \cite[Theorems~2.10 and~3.6]{HairerStuartVoss2007}. Using synchronous coupling in Cameron--Martin directions, cylinder approximation and spectral regularisation of the symmetric $L^2(\kappa)$ semigroup, we identify its full closed-gradient form in \cref{lem:quartic-direct-energy}. A curvature bound uniform in the truncation dimension gives finite-dimensional log-Sobolev inequalities by the Bakry--\'Emery criterion \cite{BakryGentilLedoux2014}. We pass these inequalities to the Gibbs measure and the full form domain, obtaining exponential relative-entropy relaxation for every finite-entropy initial law in \cref{thm:Phi4-model}. A finite-mode calculation illustrates this conclusion for a bounded perturbation of Gibbs equilibrium.

For the driven Gaussian dynamics, we separate relaxation from circulation through the action of the linear flow $\mathsf U_t=e^{-tQB}$. Gaussian semigroup theory relates state-space evolution to its action on the Cameron--Martin space \cite{ChojnowskaMichalikGoldys2002}. Here $\mathsf U_tQ^{1/2}=Q^{1/2}\mathsf O_t$ for an orthogonal operator $\mathsf O_t$ on $\mathbb H$, so the flow preserves both $\mu_Q$ and the Cameron--Martin energy. This gives a factorisation of the transition semigroup and transports densities and their closed-gradient energy along $\mathsf U_t$. In particular, the driven and symmetric dynamics started from the same initial density have the same relative entropy with respect to $\mu_Q$ at every time; see \cref{prop:circulation-semigroup}.

Following a mass quench, the logarithmic-gradient current and the circulation current are square-integrable with respect to the evolving law. A finite-mode calculation, followed by an $L^2$ limit, proves their orthogonality after integration against that law and yields the Onsager decomposition of \cref{thm:lin-hk}. Neither this decomposition nor the semigroup factorisation requires $BQ=QB$. The two-mode example and its numerical illustrations in \S\ref{sec:examples} show both the equality of the relative-entropy curves and the persistent circulation contribution to entropy production.

The interpretation of stationary circulation through path reversal requires a comparison of the full path laws, a separate aspect of infinite-dimensional diffusion theory \cite{FollmerWakolbinger1986}. For the present Gaussian model, reversal replaces $B$ by $-B$. Gaussian exponential integrability and stationarity give short-time change-of-measure estimates, which successive Girsanov transformations extend to every finite interval; the underlying infinite-dimensional Girsanov theorem is treated in \cite[Chapter~10]{DaPratoZabczyk2014}. If $\Pbb_{B,T}$ denotes the stationary law on $C([0,T];\mathbb H)$ and $(r_T\omega)(t)=\omega(T-t)$, then \cref{thm:Galerkin-path-KL} gives
\begin{equation*}
    T^{-1}\RelEnt{\Pbb_{B,T}}{\Pbb_{B,T}\circ r_T^{-1}}
    = \|Q^{1/2}BQ^{1/2}\|_{\mathrm{HS}}^2.
\end{equation*}
Here $\|\cdot\|_{\mathrm{HS}}$ is the Hilbert--Schmidt norm. The right-hand side is the stationary circulation energy, so the formula identifies its contribution to path-space irreversibility. It is also the monotone limit of the Galerkin rates, as proved in \cref{cor:galerkin-path-rate}.

Recent studies also address stationary entropy production for diffusions on infinite-dimensional spaces. Using a Gaussian reference measure, Duan, Zhang and Zimmer establish a stationary path-space formula and relate its vanishing to the stationary current, generator symmetry, detailed balance and reversibility \cite{DuanZhangZimmer2026}. Working directly with the invariant measure, Li, Zhao and Zhu prove a Cameron--Martin lower bound for stationary entropy production, allowing state-dependent and degenerate trace-class noise covariance, and give conditions for equality \cite{LiZhaoZhu2026}. Our emphasis is on transient entropy dissipation for interacting Gibbs fields and on the separation of relaxation from circulation in the driven Gaussian class.

\S\ref{sec:setting} derives the field problem from finite-system calculations and introduces the closed Dirichlet-form assumptions. \S\ref{sec:equilibrium} proves the symmetric results and states their consequences for Langevin dynamics. \S\ref{sec:linear} treats Gaussian circulation and path-space entropy production, and \S\ref{sec:examples} presents the quartic Gibbs field and two-mode circulation. \S\ref{sec:outlook} discusses further directions. The numerical methods are described in Appendix~\ref{app:numerics}.
\section*{Use of AI tools}

DeepSeek and GLM were used in preparing the initial draft. OpenAI Codex assisted with subsequent revisions, including the review of mathematical arguments and references, as well as work on calculations and proofs for selected examples. All AI-assisted material was reviewed and verified by the authors, who take full responsibility for the content.

\section{From Finite Systems to Fields}
\label{sec:setting}

Finite-system entropy balances use an invariant reference law and integration or summation by parts \cite{LebowitzSpohn1999,Qian2013,MaRoeckner1992}. In passing from finite-dimensional diffusions to fields, Lebesgue measure must be replaced by a suitable reference measure, and the integration-by-parts formula must be justified relative to that measure.

\subsection{Finite-System and Field Calculations}

We first fix the entropy notation used in both calculations. For probability measures $\nu$ and $\rho$ on the same measurable space, their relative entropy is
\begin{equation}\label{eq:relative-entropy-definition}
    \RelEnt{\nu}{\rho}
    =\begin{cases}
        \displaystyle\int \log\!\left(\frac{d\nu}{d\rho}\right)d\nu, & \nu\ll\rho,\\[5pt]
        +\infty, &\text{otherwise}.
    \end{cases}
\end{equation}
Here $\nu\ll\rho$ means absolute continuity and $\nu\sim\rho$ means mutual absolute continuity. For $h\ge0$ with $0<\int h\,d\rho<\infty$, its homogeneous density entropy relative to $\rho$ is
\begin{equation}\label{eq:entropy-definition}
    \Ent_\rho(h) = \int h\log\!\left(\frac{h}{\int h\,d\rho}\right)d\rho.
\end{equation}
We use the conventions $0\log0=0$ and $\Ent_\rho(0)=0$. If $h$ is a probability density with respect to $\rho$, then $\int h\,d\rho=1$ and $d(h\rho)/d\rho=h$; hence $\Ent_\rho(h)=\RelEnt{h\rho}{\rho}$. For an integrable observable $F$, we write $\E_\rho[F]=\int F\,d\rho$.

The diffusion calculation gives the most direct form of the finite-system argument; see also the Fokker--Planck entropy calculation in \cite[Section~3.3.1]{JiangQianQian2004}. Let $\mathsf M$ be a symmetric positive-definite mobility matrix on $\mathbb R^N$, and let $U\in C^2(\mathbb R^N)$ satisfy
\begin{equation*}
    Z:=\int_{\mathbb R^N}e^{-U(x)}\,dx\in(0,\infty).
\end{equation*}
Set
\begin{equation}\label{eq:finite-gibbs-density}
    \pi(dx)=Z^{-1}e^{-U(x)}\,dx.
\end{equation}
We also write $\pi(x)=Z^{-1}e^{-U(x)}$ for the Lebesgue density of $\pi$. Throughout this finite-dimensional calculation, we consider conservative diffusions with invariant law $\pi$, under regularity and integrability conditions that justify the displayed operations. For smooth $F,G$ with compact support, or more generally with sufficient decay to justify integration by parts, the reversible generator and its energy are
\begin{equation}\label{eq:finite-diffusion-ibp}
    L_sG = \Tr(\mathsf M\nabla^2G) - \langle\mathsf M\nabla U,\nabla G\rangle,
    \qquad
    -\int F L_sG\,d\pi = \int(\nabla F)^T\mathsf M\nabla G\,d\pi.
\end{equation}
Here $\Tr$ is the matrix trace and $\langle a,b\rangle=a^Tb$ is the Euclidean inner product. Thus, using the symmetry of $\mathsf M$,
\begin{equation*}
    \Tr(\mathsf M\nabla^2G) = \sum_{i,j=1}^N\mathsf M_{ij}\partial_{ij}G,
    \qquad
    \langle\mathsf M\nabla U,\nabla G\rangle = \sum_{i,j=1}^N\mathsf M_{ij}\partial_iU\,\partial_jG,
\end{equation*}
and $(\nabla F)^T\mathsf M\nabla G=\sum_{i,j=1}^N\mathsf M_{ij}\partial_iF\,\partial_jG$. In \eqref{eq:finite-diffusion-ibp}, the first formula defines the backward generator, and the second gives the integration-by-parts relation used in the Fokker--Planck calculation. In particular, under the reversible evolution, if $h_t$ is a smooth positive density of the time-$t$ law with respect to $\pi$, then differentiation under the integral and the equation $\partial_t h_t=L_sh_t$ give

\begin{equation}\label{eq:finite-entropy-dissipation}
    \begin{aligned}
        \frac d{dt}\Ent_\pi(h_t)
        &= \frac d{dt}\int_{\mathbb R^N}h_t\log h_t\,d\pi
        = \int_{\mathbb R^N}(1+\log h_t)\partial_t h_t\,d\pi\\
        &= \int_{\mathbb R^N}(1+\log h_t)L_sh_t\,d\pi
        = \int_{\mathbb R^N}\log h_t\,L_sh_t\,d\pi\\
        &= -\int_{\mathbb R^N}(\nabla h_t)^T\mathsf M\nabla\log h_t\,d\pi
        = -\int_{\mathbb R^N}h_t(\nabla\log h_t)^T\mathsf M\nabla\log h_t\,d\pi\\
        &=-4\int_{\mathbb R^N}(\nabla\sqrt{h_t})^T
        \mathsf M\nabla\sqrt{h_t}\,d\pi.
    \end{aligned}
\end{equation}

The constant term disappears because $\int L_sh_t\,d\pi=0$, which follows from invariance and conservativity, and the integration-by-parts step is \eqref{eq:finite-diffusion-ibp} with $F=\log h_t$ and $G=h_t$. The last two lines use
\begin{equation*}
    \nabla h_t=h_t\nabla\log h_t = 2\sqrt{h_t}\,\nabla\sqrt{h_t}.
\end{equation*}
A stationary circulation is described by a velocity field $v_{\rm ss}$ satisfying
\begin{equation}\label{eq:finite-stationary-divergence}
    \nabla\cdot(\pi v_{\rm ss})=0.
\end{equation}
This condition makes $L_aG=v_{\rm ss}\cdot\nabla G$ antisymmetric on smooth compactly supported functions in $L^2(\pi)$. Writing $L=L_s+L_a$, the antisymmetric part is a derivation and satisfies $L_a\Phi(h)=\Phi'(h)L_ah$ for smooth $\Phi$. The density relative to $\pi$ evolves under the $L^2(\pi)$-adjoint generator $L^\dagger$,
\begin{equation*}
    \partial_t h_t=L^\dagger h_t=(L_s-L_a)h_t.
\end{equation*}
Consequently, the contribution of the antisymmetric part to the entropy derivative is

\begin{equation}\label{eq:finite-skew-entropy-neutrality}
    -\int(1+\log h_t)L_ah_t\,d\pi = \int h_tL_a\log h_t\,d\pi = \int L_ah_t\,d\pi=0.
\end{equation}

At stationarity the reversed generator is $L^\dagger=L_s-L_a$. For stationary diffusions satisfying the hypotheses of \cite[Theorem~4.1]{DaCostaPavliotis2023}, the forward--reverse relative-entropy rate is the quadratic current energy \cite{Qian2013}
\begin{equation}\label{eq:finite-current-energy}
    \int_{\mathbb R^N}v_{\rm ss}^T\mathsf M^{-1}v_{\rm ss}\,d\pi.
\end{equation}
The equilibrium and driven calculations therefore use the same invariant law and adjoint operation, while the symmetric gradient energy describes dissipation and the antisymmetric velocity describes circulation.

The same scheme can be stated on different state spaces. Let $\mathsf E$ be a locally compact Polish space, let $\pi$ be invariant, and let $L$ generate a conservative Markov semigroup $(P_t)_{t\ge0}$ on $L^2(\pi)$. Here the Markov property means that $P_t$ is positivity preserving and $0\le F\le1$ implies $0\le P_tF\le1$; conservativity means $P_t1=1$.

For a finite-state chain, $\mathsf E$ is a finite discrete space and $L$ is a rate matrix acting on $\mathbb R^{\mathsf E}$; for a diffusion, $\mathsf E=\mathbb R^N$ or a finite-dimensional manifold and $L$ is a second-order differential operator on a smooth core; for a jump diffusion, $L$ is the corresponding local--nonlocal integro-differential operator. On a common core, set
\begin{equation*}
    L_s=\frac12(L+L^\dagger),\qquad
    L_a=\frac12(L-L^\dagger),\qquad
    \cE_s(F,G)=-\langle F,L_sG\rangle_\pi.
\end{equation*}
Here $\langle F,G\rangle_\pi:=\int_{\mathsf E}FG\,d\pi$. For the finite-state chains, diffusions and jump diffusions used as benchmarks, the symmetric energy has a continuous part and a jump part of the form
\begin{equation*}
    \cE_s(F,G) = \cE_s^{(c)}(F,G) + \frac12\iint_{\mathsf E\times\mathsf E\setminus\{x=y\}}[F(y)-F(x)][G(y)-G(x)]\,J_s(dx,dy).
\end{equation*}
Here $\cE_s^{(c)}$ is the continuous part and $J_s$ is the symmetric jump measure on $\mathsf E\times\mathsf E\setminus\{x=y\}$. On a finite discrete space only the jump increments remain; for a diffusion on $\mathbb R^N$ only the continuous gradient part remains; and for a jump diffusion both occur. For the reversible dynamics generated by $L_s$, the master or Fokker--Planck equation gives
\begin{equation*}
    \frac d{dt}\Ent_\pi(h_t)
    = -\cE_s(h_t,\log h_t)
    \le -4\cE_s(\sqrt{h_t},\sqrt{h_t}).
\end{equation*}
For the continuous part, the diffusion chain rule gives
\begin{equation*}
    \cE_s^{(c)}(h_t,\log h_t) = 4\cE_s^{(c)}(\sqrt{h_t},\sqrt{h_t}).
\end{equation*}
For the jump integral, the pointwise inequality
\begin{equation*}
    (a-b)(\log a-\log b)
    \ge4(\sqrt a-\sqrt b)^2
\end{equation*}
gives the corresponding inequality after integration against $J_s$. Thus the displayed entropy relation is an equality for a pure diffusion, whereas it may be strict when a symmetric jump part is present.

Oriented edge currents and divergence-free velocities provide the corresponding finite-system descriptions of irreversibility \cite{LebowitzSpohn1999,FanZhang2026,ZhangLu2025}. These calculations use an invariant reference law, integration or summation by parts, the resulting symmetric energy and the adjoint operation that determines the reversed dynamics.

For the field calculation, let $\mathbb H$ be a real separable Hilbert space with inner product $\langle\cdot,\cdot\rangle_{\mathbb H}$ and norm $\|\cdot\|_{\mathbb H}$. Let $Q$ be a positive, injective, trace-class operator, and write $A=Q^{-1}$ for its generally unbounded inverse. Denote the Gaussian measure with covariance $Q$ by $\mu_Q=\cN(0,Q)$. Choose a topological field state space $\mathbb X\subset\mathbb H$ such that $\mu_Q(\mathbb X)=1$. The Cameron--Martin space of $\mu_Q$ is $\mathbb K=Q^{1/2}\mathbb H$, equipped with
\begin{equation*}
    \langle h,k\rangle_{\mathbb K} = \langle Q^{-1/2}h,Q^{-1/2}k\rangle_{\mathbb H},
    \qquad
    \|h\|_{\mathbb K} = \|Q^{-1/2}h\|_{\mathbb H}.
\end{equation*}
Here $Q^{-1/2}$ denotes the inverse of $Q^{1/2}$ on its range $\mathbb K$. Equivalently, if $(e_j)$ is an orthonormal eigenbasis with $Qe_j=q_je_j$, $q_j>0$, then
\begin{equation*}
    \mathbb K = \left\{h=\sum_jh_je_j:\sum_j\frac{|h_j|^2}{q_j}<\infty\right\},
    \qquad
    \langle h,k\rangle_{\mathbb K} = \sum_j\frac{h_jk_j}{q_j}.
\end{equation*}
Thus $\langle\cdot,\cdot\rangle_{\mathbb H}$ is the ambient pairing, whereas $\langle\cdot,\cdot\rangle_{\mathbb K}$ measures Cameron--Martin directions with the stronger $Q^{-1/2}$ weight. For a scalar cylinder functional $F$, $\mathrm DF$ denotes its Fr\'echet differential and $\nabla_{\mathbb H}F$ its $\mathbb H$-Riesz representative, determined by
\begin{equation*}
    \mathrm DF(\phi)[h] = \langle\nabla_{\mathbb H}F(\phi),h\rangle_{\mathbb H};
\end{equation*}
$\nabla_{\mathbb H}^2F = \mathrm D(\nabla_{\mathbb H}F)$ is its Hessian and $\mathrm DG$ denotes the derivative of a cylinder vector field. Let $W^Q$ be an $\mathbb H$-valued $Q$-Wiener process, so that
\begin{equation*}
    \mathbb E\!\left[ \langle W_t^Q-W_s^Q,h\rangle_{\mathbb H} \langle W_t^Q-W_s^Q,k\rangle_{\mathbb H}
    \right] = (t-s)\langle Qh,k\rangle_{\mathbb H}
\end{equation*}
for $0\le s\le t$ and $h,k\in\mathbb H$. If $V$ is the interaction potential and $f$ the driving field, the formal field equation corresponding to the finite-dimensional calculation is
\begin{equation}\label{eq:main}
    d\phi_t = -\bigl(\phi_t + Q\,\nabla_{\mathbb H}V(\phi_t) + Q\,f(\phi_t)\bigr)\,dt + \sqrt{2}\,dW^Q,
\end{equation}
where $Q$ is both the Gaussian covariance and the mobility. The finite-dimensional construction begins with the Gibbs density \eqref{eq:finite-gibbs-density}. For the unforced detailed-balance model $f=0$, there is no infinite-dimensional Lebesgue measure against which the analogous density can be written. The quadratic part of the formal energy $U(\phi)=\frac12\langle\phi,A\phi\rangle_{\mathbb H}+V(\phi)$ is instead represented by $\mu_Q$, and the interaction suggests
\begin{equation}\label{eq:gibbs-candidate}
    \kappa(d\phi) = Z^{-1}e^{-V(\phi)}\mu_Q(d\phi),
    \qquad
    Z = \int_{\mathbb X}e^{-V}\,d\mu_Q\in(0,\infty).
\end{equation}
The noise covariance determines the gradient energy through the quadratic variation of cylinder observables: for the noise in \eqref{eq:main-intro}, this energy is $\langle Q^{1/2}\nabla_{\mathbb H}F,Q^{1/2}\nabla_{\mathbb H}G\rangle_{\mathbb H}$. It is represented by the Cameron--Martin gradient introduced below. Gaussian integration by parts then provides the counterpart of \eqref{eq:finite-diffusion-ibp}. On cylinder functions the Hessian has finite rank, and the corresponding formal generator is
\begin{equation}\label{eq:formal-generator}
    \cL F = \Tr(Q\nabla_{\mathbb H}^2F) - \langle\phi+Q\nabla_{\mathbb H}V,\nabla_{\mathbb H}F\rangle_{\mathbb H}
\end{equation}
where $\Tr$ is the operator trace. The Gaussian integration-by-parts formula reads
\begin{equation}\label{eq:G-IBP}
    \int_{\mathbb H}\langle\phi,G(\phi)\rangle_{\mathbb H}\,d\mu_Q = \int_{\mathbb H}\Tr(Q\,\mathrm DG(\phi))\,d\mu_Q.
\end{equation}
Formula~\eqref{eq:G-IBP} holds for suitable cylinder vector fields, but the substitution $G=e^{-V}\nabla_{\mathbb H}F$ is no longer automatic: the Gibbs weight must have the required Cameron--Martin derivatives, and the resulting terms must be integrable and obtainable by approximation. When this substitution is valid, with the standard convention $ (u\otimes v)h=\langle v,h\rangle_{\mathbb H}u $, the product rule gives
\begin{equation*}
    \mathrm DG=e^{-V}\nabla_{\mathbb H}^2F - e^{-V}\nabla_{\mathbb H}F\otimes\nabla_{\mathbb H}V.
\end{equation*}
Substituting this derivative into \eqref{eq:G-IBP} gives
\begin{align*}
    \int_{\mathbb H}e^{-V}\cL F\,d\mu_Q
    &=\int_{\mathbb H}e^{-V}\Tr(Q\nabla_{\mathbb H}^2F)\,d\mu_Q - \int_{\mathbb H}e^{-V}
    \langle Q\nabla_{\mathbb H}V,\nabla_{\mathbb H}F\rangle_{\mathbb H} \,d\mu_Q\\
    &\quad-\int_{\mathbb H}e^{-V}\langle\phi,\nabla_{\mathbb H}F\rangle_{\mathbb H}\,d\mu_Q = 0.
\end{align*}
Since $d\kappa=Z^{-1}e^{-V}d\mu_Q$, this is precisely
\begin{equation}\label{eq:infinitesimal-invariance}
    \int_{\mathbb X}\cL F\,d\kappa=0.
\end{equation}
Repeating this admissible integration-by-parts step with $F e^{-V}\nabla_{\mathbb H}G$ gives the field counterpart of the second relation in \eqref{eq:finite-diffusion-ibp}:
\begin{equation*}
    -\int_{\mathbb X}F\,\cL G\,d\kappa = \int_{\mathbb X}\langle Q^{1/2}\nabla_{\mathbb H}F,Q^{1/2}\nabla_{\mathbb H}G\rangle_{\mathbb H}\,d\kappa =: \cE_0(F,G).
\end{equation*}
In finite dimensions the right-hand side is a classical Sobolev energy. On the field state space, the bilinear energy $\cE_0$ is initially defined only on the cylinder core. If it is closable, its closure is a closed form on $L^2(\kappa)$; the precise terminology is fixed in the next subsection.

The calculation in \eqref{eq:finite-entropy-dissipation} identifies this energy as the quantity governing entropy decay. Passing to its closure allows the dissipation to be formulated for square-root densities in the form domain, without requiring the evolving density to admit the classical derivatives used in the finite-dimensional calculation. Conservativity, strong locality and the diffusion chain rule of the closed form then make the passage from \eqref{eq:finite-diffusion-ibp} to the entropy identity \eqref{eq:finite-entropy-dissipation} valid in the field setting. Under these properties, a smooth positive relative density satisfying $\partial_t h_t=\cL h_t$ obeys
\begin{equation*}
    \begin{aligned}
        \frac d{dt}\Ent_\kappa(h_t)
        &= \int_{\mathbb X}(1+\log h_t)\cL h_t\,d\kappa
        = \int_{\mathbb X}\log h_t\,\cL h_t\,d\kappa\\
        &= -\cE(h_t,\log h_t)=-4\cE(\sqrt{h_t},\sqrt{h_t}).
    \end{aligned}
\end{equation*}

For the driven calculation, write the full formal generator as $\cL+\cL_f$, where
\begin{equation*}
    \cL_fF=-\langle Qf,\nabla_{\mathbb H}F\rangle_{\mathbb H}.
\end{equation*}
The finite-dimensional divergence condition \eqref{eq:finite-stationary-divergence} makes the stationary velocity antisymmetric and leads to \eqref{eq:finite-skew-entropy-neutrality}. Since there is no ambient volume divergence on the field state space, its counterpart must be verified directly against $\kappa$ as the infinitesimal invariance condition
\begin{equation*}
    \int_{\mathbb X}\cL_fF\,d\kappa=0
\end{equation*}
on an algebra stable under products. If this condition holds and $\cL_f$ acts there as a derivation, then
\begin{equation*}
    \int_{\mathbb X}F\cL_fG\,d\kappa = -\int_{\mathbb X}G\cL_fF\,d\kappa,
    \qquad
    \int_{\mathbb X}h\cL_f\log h\,d\kappa = \int_{\mathbb X}\cL_fh\,d\kappa=0.
\end{equation*}
This is the field counterpart of \eqref{eq:finite-skew-entropy-neutrality}: the antisymmetric term contributes zero to the entropy derivative at the current density, and the stationary adjoint changes $\cL+\cL_f$ into $\cL-\cL_f$. To obtain the analogue of the finite-dimensional path-space rate \eqref{eq:finite-current-energy}, these generator and current calculations must hold for the prescribed SPDE. We also need to identify the reversed process and verify square integrability of the Cameron--Martin current. The $\mathbb H$-coordinate of this current is $-Q^{1/2}f$, and its quadratic energy is
\begin{equation*}
    \int_{\mathbb X}\|Q^{1/2}f(\phi)\|_{\mathbb H}^2\,d\kappa(\phi).
\end{equation*}

To justify the entropy calculation, the energy of the prescribed dynamics must extend to a closed form with the required chain rule. We formulate these properties below and verify them for the nonlinear SPDEs in the examples. \Cref{tab:finite-infinite-map} summarises the corresponding finite-system and field objects; the antisymmetric part is treated in \S\ref{sec:linear}.

\begin{table}[htbp]
\centering
\small
\renewcommand{\arraystretch}{1.15}
\begin{tabular}{@{}L{0.24\textwidth}L{0.31\textwidth}L{0.37\textwidth}@{}}
    \hline
    \textbf{Finite-system language} & \textbf{Function-space object} &
    \textbf{Role in this paper}\\
    \hline
    Reversible generator and Fisher information &
    Symmetric Dirichlet form and carr\'e du champ &
    Relative-entropy dissipation without an ambient Lebesgue density.\\
    Steady velocity or edge current &
    Antisymmetric action on cylinder observables and, in the Gaussian class, a Cameron--Martin field &
    Stationary circulation and forward--reverse path-space relative entropy.\\
    Drift and noise coefficients of a stochastic differential equation (SDE) &
    Closed gradient form and the SPDE transition semigroup &
    Verification of the closed-form hypotheses for the prescribed dynamics.\\
    \hline
\end{tabular}
\caption{The finite-to-infinite reformulation used throughout the paper.}
\label{tab:finite-infinite-map}
\end{table}

\subsection{Closed Dirichlet-Form Framework}

For a symmetric generator, the integration-by-parts relation is expressed through the bilinear energy $\cE(F,G)=-\langle F,LG\rangle_{L^2(\kappa)}$. Closedness makes this energy available beyond smooth test functions, and the local calculus provides the chain rule for the entropy calculation. We state these requirements as properties of the form; their verification for a concrete dynamics is part of the application.

\begin{definition}[Forms, closure and local calculus]
Let $\kappa$ be a probability measure on a topological state space $\mathbb X$, let $\mathcal D_0$ be dense in $L^2(\mathbb X,\kappa)$, and let $\cE_0:\mathcal D_0\times\mathcal D_0\to\mathbb R$ be a nonnegative symmetric bilinear form. We call $(\cE_0,\mathcal D_0)$ a \emph{pre-form}. For a form $(\cE,\operatorname{Dom}(\cE))$, set
\begin{equation*}
    \|u\|_{\cE_1}^2 = \|u\|_{L^2(\kappa)}^2+\cE(u,u).
\end{equation*}
For the pre-form, we similarly write
\begin{equation*}
    \|u\|_{\cE_{0,1}}^2 := \|u\|_{L^2(\kappa)}^2+\cE_0(u,u),\qquad u\in\mathcal D_0.
\end{equation*}
The form $(\cE,\operatorname{Dom}(\cE))$ is \emph{closed} if $\operatorname{Dom}(\cE)$ is complete for $\|\cdot\|_{\cE_1}$. The pre-form $(\cE_0,\mathcal D_0)$ is \emph{closable} if every sequence $(u_n)\subset\mathcal D_0$ satisfying
\begin{equation*}
    u_n\longrightarrow0\quad\text{in }L^2(\kappa),
    \qquad
    \cE_0(u_n-u_m,u_n-u_m)\longrightarrow0
\end{equation*}
as $n,m\to\infty$ also satisfies $\cE_0(u_n,u_n)\to0$. In that case its closure $(\cE,\operatorname{Dom}(\cE))$ is defined by
\begin{equation*}
    \operatorname{Dom}(\cE) = \left\{u\in L^2(\kappa):
    \begin{array}{l}
        \text{there is a $\|\cdot\|_{\cE_{0,1}}$-Cauchy sequence $(u_n)\subset\mathcal D_0$}\\
        \text{such that $u_n\to u$ in $L^2(\kappa)$}
    \end{array}\right\},
\end{equation*}
with $\cE(u,v)=\lim_n\cE_0(u_n,v_n)$ for approximating sequences. This limit is independent of the approximations. A subspace $\mathcal C\subset\operatorname{Dom}(\cE)$ is a \emph{form core} if it is dense in $\operatorname{Dom}(\cE)$ for $\|\cdot\|_{\cE_1}$; in particular, $\mathcal D_0$ is a core of the closure of a closable pre-form.

A \emph{symmetric Dirichlet form} $(\cE,\operatorname{Dom}(\cE))$ on $L^2(\mathbb X,\kappa)$ is a densely defined, closed, nonnegative symmetric bilinear form with the Markov property
\begin{equation*}
    u\in \operatorname{Dom}(\cE)
    \quad\Longrightarrow\quad
    (0\vee u)\wedge1\in \operatorname{Dom}(\cE),\qquad
    \cE((0\vee u)\wedge1,(0\vee u)\wedge1)\le \cE(u,u).
\end{equation*}
A \emph{normal contraction} is a map $\theta:\mathbb R\to\mathbb R$ satisfying $\theta(0)=0$ and $|\theta(r)-\theta(s)|\le|r-s|$ for all $r,s\in\mathbb R$. The form $\cE$ is \emph{conservative} when $1\in \operatorname{Dom}(\cE)$ and $\cE(1,1)=0$, and \emph{strongly local} when $\cE(u,v)=0$ whenever $u$ is constant on a neighbourhood of the support of $v$. Let $L$ be the nonpositive self-adjoint generator associated with $\cE$. For $F\in\operatorname{Dom}(\cE)$ and $G\in\operatorname{Dom}(L)$,
\begin{equation}\label{eq:generator-form-convention}
    \cE(F,G)=-\int_{\mathbb X}F\,LG\,d\kappa.
\end{equation}
On an algebra on which products and the generator are defined, the \emph{carr\'e du champ} is
\begin{equation}\label{eq:cdc-definition}
    \Gamma(F,G)=\frac12\bigl(L(FG)-F\,LG-G\,LF\bigr), \qquad \Gamma(F)=\Gamma(F,F).
\end{equation}
For the gradient forms constructed below,
\begin{equation*}
    \Gamma(F,G) = \langle Q^{1/2}\nabla_{\mathbb H}F,Q^{1/2}\nabla_{\mathbb H}G\rangle_{\mathbb H},
    \qquad
    \cE(F,G) = \int_{\mathbb X}\Gamma(F,G)\,d\kappa.
\end{equation*}
These conventions agree with the standard theories of symmetric Dirichlet forms and diffusion semigroups \cite{MaRoeckner1992,BakryGentilLedoux2014}.
\end{definition}

Fix a probability measure $\kappa$ on a topological state space $\mathbb X$ and a bilinear form $(\cE,\operatorname{Dom}(\cE))$ on $L^2(\mathbb X,\kappa)$. We impose the following two standing assumptions.

\begin{assumption}[Symmetric form hypotheses]\label{ass:equilibrium}
\leavevmode
\begin{enumerate}[label=(\roman*)]
    \item $(\cE,\operatorname{Dom}(\cE))$ is a densely defined, closed, symmetric Dirichlet form on $L^2(\kappa)$.
    \item The form is conservative: $1\in \operatorname{Dom}(\cE)$ and $\cE(1,1)=0$.
\end{enumerate}
\end{assumption}

\begin{assumption}[Diffusion-form structure]\label{ass:diffusion-calculus}
The form $(\cE,\operatorname{Dom}(\cE))$ is strongly local and admits a carr\'e du champ $\Gamma:\operatorname{Dom}(\cE)\times\operatorname{Dom}(\cE)\to L^1(\kappa)$ such that
\begin{equation*}
    \cE(F,G)=\int_{\mathbb X}\Gamma(F,G)\,d\kappa.
\end{equation*}
For bounded $F,G\in\operatorname{Dom}(\cE)$ and every $C^1$ function $\Phi$ whose derivative is bounded on the range of $F$, we have $\Phi(F)\in\operatorname{Dom}(\cE)$ and
\begin{equation*}
    \Gamma(\Phi(F),G) = \Phi'(F)\Gamma(F,G) \qquad \kappa\text{-a.e.}.
\end{equation*}
We refer to this carr\'e du champ formula as the diffusion chain rule.
\end{assumption}

Under the first standing assumption, the closed symmetric form determines a self-adjoint Markov semigroup and the reference measure is invariant. The second assumption supplies the chain rule used to express entropy dissipation as the square-root energy. The entropy-dissipation results depend on these form properties rather than on a Gaussian representation of the reference measure. For a given SPDE, we verify that its transition semigroup corresponds to a form with these properties.

\begin{proposition}[Closed-form dynamics and invariance]
\label{prop:form-invariance}
Under \cref{ass:equilibrium}, the form has a non-positive self-adjoint generator $(L,\operatorname{Dom}(L))$ and a self-adjoint Markov semigroup $(P_t)_{t\ge0}$. The semigroup is conservative and $\kappa$ is invariant:
\begin{equation}\label{eq:form-semigroup-invariance}
    \int_{\mathbb X} P_tF\,d\kappa=\int_{\mathbb X}F\,d\kappa,\qquad F\in L^1(\kappa),\ t\ge0.
\end{equation}
\end{proposition}

\begin{proof}
The representation theorem for closed symmetric forms gives the non-positive self-adjoint generator and the self-adjoint Markov semigroup \cite[Chapter~I]{MaRoeckner1992}. Conservativity of the form is equivalent to $P_t1=1$ $\kappa$-a.e. Hence, for $F\in L^2(\kappa)$,
\begin{equation*}
    \int_{\mathbb X}P_tF\,d\kappa
    = \langle P_tF,1\rangle_{L^2(\kappa)}
    = \langle F,P_t1\rangle_{L^2(\kappa)}
    = \int_{\mathbb X}F\,d\kappa.
\end{equation*}
For general $F\in L^1(\kappa)$, apply this identity to bounded truncations and use $L^1$-contractivity of the Markov extension.
\end{proof}

\begin{lemma}[Diffusion-calculus identity]
\label{lem:diffusion-calculus-identity}
Under \cref{ass:equilibrium,ass:diffusion-calculus}, if $h\in\operatorname{Dom}(\cE)$ and $0<c_-\le h\le c_+<\infty$, then $\log h,\sqrt h\in\operatorname{Dom}(\cE)$ and
\begin{equation}\label{eq:diffusion-calculus-identity}
    \cE(h,\log h)
    = \int_{\mathbb X}\frac{\Gamma(h)}{h}\,d\kappa
    = 4\cE(\sqrt h,\sqrt h).
\end{equation}
\end{lemma}

\begin{proof}
On the interval $[c_-,c_+]$, the functions $r\mapsto\log r$ and $r\mapsto\sqrt r$ have bounded derivatives. The form-domain functional calculus and the diffusion chain rule therefore give
\begin{equation*}
    \Gamma(h,\log h)=\frac{\Gamma(h)}{h},
    \qquad
    \Gamma(\sqrt h)=\frac{\Gamma(h)}{4h}.
\end{equation*}
Integrating the first identity and four times the second proves \eqref{eq:diffusion-calculus-identity}.
\end{proof}

\section{Systems with an Equilibrium Steady State}
\label{sec:equilibrium}

In this section an equilibrium steady state means an invariant reference law $\kappa$ with respect to which the transition semigroup $P_t$ satisfies detailed balance. The symmetric form $\cE$ describes entropy dissipation of the evolving law $h_t\kappa$ relative to $\kappa$, while self-adjointness of $P_t$ gives time-reversal invariance of a continuous stationary process with this semigroup.

\subsection{Entropy Dissipation from Symmetric Dirichlet Forms}

With the invariant reference law $\kappa$ and Markov semigroup $P_t$ established in \cref{prop:form-invariance}, we prove entropy dissipation for the densities $h_t=P_th_0$. The entropy identity follows from the symmetric form semigroup, while its Fisher-information representation uses the diffusion chain rule \cite{BakryGentilLedoux2014}. We use the homogeneous convention \eqref{eq:entropy-definition}. For a probability density $h$, it gives
\begin{equation*}
    \Ent_\kappa(h)=\int h\log h\,d\kappa.
\end{equation*}
In the log-Sobolev inequality (LSI), the substitution $h=g^2$ gives
\begin{equation}\label{eq:homogeneous-entropy}
    \Ent_\kappa(g^2) := \int g^2\log\!\left(\frac{g^2}{\int g^2\,d\kappa}\right)d\kappa, \qquad
    0<\int g^2\,d\kappa<\infty.
\end{equation}
For a positive density $h$ to which the diffusion chain rule applies, the quantity
\begin{equation*}
    4\cE(\sqrt h,\sqrt h) = \int h\,\Gamma(\log h)\,d\kappa
\end{equation*}
is the intrinsic Fisher information, or Dirichlet-form dissipation, relative to $\kappa$.
\begin{theorem}[Dirichlet-form entropy dissipation]\label{thm:deBruijn}
Assume \cref{ass:equilibrium,ass:diffusion-calculus}, let $h_0$ be a probability density with respect to $\kappa$, set $h_t=P_th_0$, and write $\mu_t=h_t\kappa$.
\begin{enumerate}[label=(\roman*)]
    \item If $h_0 \in \operatorname{Dom}(L)$ and $0 < c_- \le h_0 \le c_+ < \infty$, then, for every $t\ge0$,
    \begin{equation}\label{eq:deBruijn-general}
        \Ent_\kappa(h_t) + \int_0^t \cE(h_s,\log h_s)\,ds = \Ent_\kappa(h_0).
    \end{equation}
    For every $s\ge0$,
    \begin{equation}\label{eq:sqrt-energy-identity}
        \cE(h_s,\log h_s)=4\cE(\sqrt{h_s},\sqrt{h_s}),
    \end{equation}
    and consequently
    \begin{equation}\label{eq:deBruijn-eq}
        \Ent_\kappa(h_t) + 4\int_0^t \cE(\sqrt{h_s},\, \sqrt{h_s})\,ds = \Ent_\kappa(h_0).
    \end{equation}
    \item If $h_0 \ge 0$, $\int h_0\,d\kappa = 1$, and $\Ent_\kappa(h_0) < \infty$, then
    \begin{equation}\label{eq:deBruijn-ineq}
        \Ent_\kappa(h_t) + 4\int_0^t \cE(\sqrt{h_s},\, \sqrt{h_s})\,ds \le \Ent_\kappa(h_0).
    \end{equation}
\end{enumerate}
In part~(ii), the energy is understood to be $+\infty$ outside $\operatorname{Dom}(\cE)$. In particular, $\sqrt{h_s}\in\operatorname{Dom}(\cE)$ for almost every $s$ on each finite time interval. At steady state $h\equiv1$, $\cE(1,1)=0$, so the equilibrium relative entropy dissipation vanishes.
\end{theorem}

\begin{proof}
We first prove part~(i). Positivity preservation and the Markov property give $c_-\le h_t\le c_+$ for all $t\ge0$. Since $h_0\in \operatorname{Dom}(L)$, one has $h_t\in \operatorname{Dom}(L)$ and $\partial_t h_t=Lh_t$ in $L^2(\kappa)$. On the compact interval $[c_-,c_+]$, the functions $r\mapsto\log r$ and $r\mapsto1+\log r$ are Lipschitz. The normal-contraction calculus for Dirichlet forms therefore yields $\log h_t\in \operatorname{Dom}(\cE)$.

The map $t\mapsto\int h_t\log h_t\,d\kappa$ is absolutely continuous. Indeed, $t\mapsto h_t$ is differentiable in $L^2(\kappa)$, and the derivative of $r\mapsto r\log r$ is Lipschitz on $[c_-,c_+]$. The chain rule for this $L^2(\kappa)$-valued curve therefore gives, for almost every $t$,
\begin{equation*}
    \frac d{dt}\Ent_\kappa(h_t)
    =\int_{\mathbb X}(1+\log h_t)Lh_t\,d\kappa
    =\int_{\mathbb X}\log h_t\,Lh_t\,d\kappa
    =-\cE(h_t,\log h_t).
\end{equation*}
In the second equality we used $\int Lh_t\,d\kappa=-\cE(h_t,1)=0$. Integrating in time proves~\eqref{eq:deBruijn-general}. Applying \cref{lem:diffusion-calculus-identity} to $h_t$ proves \eqref{eq:sqrt-energy-identity} and \eqref{eq:deBruijn-eq}.

We now prove part~(ii). A finite-entropy density need not belong to $\operatorname{Dom}(L)$ and need not be bounded or bounded away from zero, so part~(i) cannot be applied directly. By part~(i), every probability density $g\in\operatorname{Dom}(L)$ satisfying $0<c_-\le g\le c_+<\infty$ obeys \eqref{eq:regular-dissipation-hypothesis}. The remaining hypothesis of \cref{lem:finite-entropy-extension} is \cref{ass:equilibrium}, which is already assumed in the theorem. The lemma therefore applies to $h_0$ and gives \eqref{eq:deBruijn-ineq}, proving part~(ii).
\end{proof}

\subsection{Detailed Balance and Path Reversal}
\label{sec:detailed-balance-path-reversal}

We now pass from entropy dissipation of the evolving one-time laws $h_t\kappa$ to time reversal of stationary trajectories. For a continuous Markov process with semigroup $P_t$, self-adjointness of $P_t$ gives the two-time detailed-balance relation, which the Markov property extends to the stationary path law. For a Polish state space $\mathbb S$ and $T>0$, define the time-reversal map
\begin{equation*}
    r_T:C([0,T];\mathbb S)\longrightarrow C([0,T];\mathbb S),
    \qquad
    (r_T\omega)(t)=\omega(T-t).
\end{equation*}
Thus $\Pbb\circ r_T^{-1}$ is the reversed law of a path measure $\Pbb$ on $C([0,T];\mathbb S)$.

\begin{theorem}[Equilibrium detailed balance and path reversal]
\label{thm:eq-zero}
Assume \cref{ass:equilibrium}, let $\mathbb X$ be a Polish state space, and let $X$ be a conservative time-homogeneous Markov process with continuous paths whose transition operators represent the form semigroup $P_t$ on $L^2(\kappa)$. Start $X$ from $\kappa$. Then $X$ is $\kappa$-stationary, and the following properties are equivalent:
\begin{enumerate}[label=(\roman*)]
    \item \textbf{Semigroup detailed balance.} $\int F\,P_t G\,d\kappa = \int G\,P_t F\,d\kappa$ for all $F,G\in L^2(\kappa)$ and $t\ge 0$.

    \item \textbf{Pathwise reversibility.} For every $T>0$, with $\Pbb_{\kappa,T}$ denoting the path law of $X$ on $[0,T]$,
    \begin{equation}\label{eq:path-rev}
        \Pbb_{\kappa,T} = \Pbb_{\kappa,T} \circ r_T^{-1},\qquad
        \frac{1}{T}\RelEnt{\Pbb_{\kappa,T}}{\Pbb_{\kappa,T}\circ r_T^{-1}} = 0,
    \end{equation}
    where $r_T$ is the time-reversal map defined above.
\end{enumerate}
\end{theorem}

\begin{proof}
\par\noindent\textbf{Step 1. Detailed balance for the given process.}\par
By \cref{prop:form-invariance}, $(P_t)$ is self-adjoint and $\kappa$ is invariant. The transition operators of $X$ agree with $P_t$ in $L^2(\kappa)$, so $X$ started from $\kappa$ is stationary and
\begin{equation*}
    \int F\,P_t G\,d\kappa
    =\langle F,P_tG\rangle_{L^2(\kappa)}
    =\langle P_tF,G\rangle_{L^2(\kappa)}
    =\int G\,P_t F\,d\kappa.
\end{equation*}

\par\noindent\textbf{Step 2. From detailed balance to path reversal.}\par
Write $\E_\kappa$ for expectation for the process $X$ started from $\kappa$. Take $0=t_0<t_1<\cdots<t_n=T$, put $\Delta_i=t_i-t_{i-1}$, and let $\varphi_0,\ldots,\varphi_n$ be bounded measurable functions. The Markov property gives
\begin{equation}\label{eq:forward-fdd}
    \E_\kappa\prod_{i=0}^n\varphi_i(X_{t_i})
    =\int\varphi_0P_{\Delta_1}
    \bigl(\varphi_1P_{\Delta_2}(\cdots P_{\Delta_n}\varphi_n)\bigr) \,d\kappa.
\end{equation}
To display the iteration, put
\begin{equation*}
    H_1 = \varphi_1P_{\Delta_2}\bigl(\varphi_2\cdots P_{\Delta_n}\varphi_n\bigr).
\end{equation*}
The first use of self-adjointness gives
\begin{equation*}
    \int\varphi_0P_{\Delta_1}H_1\,d\kappa
    =\int H_1P_{\Delta_1}\varphi_0\,d\kappa
    =\int \varphi_1P_{\Delta_2} \bigl( \varphi_2\cdots P_{\Delta_n}\varphi_n \bigr) P_{\Delta_1}\varphi_0\,d\kappa.
\end{equation*}
We then apply self-adjointness to $P_{\Delta_2}$ and successively to $P_{\Delta_3},\ldots,P_{\Delta_n}$, moving one semigroup at a time to the right. After $n$ iterations the expression becomes
\begin{equation}\label{eq:reversed-fdd}
    \int\varphi_nP_{\Delta_n} \bigl( \varphi_{n-1}P_{\Delta_{n-1}}(\cdots P_{\Delta_1}\varphi_0) \bigr) \,d\kappa.
\end{equation}
This is exactly $\E_\kappa\prod_{i=0}^n\varphi_i(X_{T-t_i})$. Thus all finite-dimensional distributions are invariant under $r_T$. An arbitrary finite set of observation times need not contain $0$ or $T$; inserting the endpoint test functions $1$ reduces it to the calculation above. Hence the conclusion covers all cylinder events. The path-space $\sigma$-algebra is generated by the coordinate maps, so the path laws agree and the relative entropy in~\eqref{eq:path-rev} is zero.

\par\noindent\textbf{Step 3. From path reversal to detailed balance.}\par
For any $\kappa$-stationary Markov process satisfying (ii), equality of the path laws implies equality of all finite-dimensional distributions. Given $t\ge0$, apply (ii) on any interval $[0,T]$ with $T\ge t$ and specialise to the two observation times $0$ and $t$. Stationarity then gives
\begin{equation*}
    \E_\kappa[F(X_0)G(X_t)]
    =\E_\kappa[G(X_0)F(X_t)],
\end{equation*}
which is precisely the detailed-balance relation in (i).
\end{proof}

\subsection{Entropy Dissipation for Langevin Dynamics}
\label{sec:gibbs-verification}
\label{sec:conventions}

On $\mathbb H=L^2(\T;\mathbb R)$, with $\T=\mathbb R/(2\pi\mathbb Z)$, let $A=-\partial_x^2+m$, $m>0$, and let $Q=A^{-1}$ be its periodic Green operator. The unforced Langevin equation is
\begin{equation}\label{eq:equilibrium-langevin}
    d\phi_t=-\bigl(\phi_t+Q\nabla_{\mathbb H}V(\phi_t)\bigr)\,dt+\sqrt2\,dW_t^Q.
\end{equation}
Here $\nabla_{\mathbb H}V$ denotes the $\mathbb H$-valued representative of the first variation: for a smooth periodic direction $h$,
\begin{equation*}
    \partial_hV(\phi)
    :=\left.\frac{d}{d\varepsilon}V(\phi+\varepsilon h)\right|_{\varepsilon=0}
    =\langle\nabla_{\mathbb H}V(\phi),h\rangle_{\mathbb H}.
\end{equation*}
For the nonlinear models below, the state and Cameron--Martin spaces are
\begin{equation}\label{eq:field-spaces}
    \mathbb X=C(\T;\mathbb R),\qquad
    \mathbb K=Q^{1/2}\mathbb H=H^1(\T),
\end{equation}
with the Cameron--Martin inner product fixed in \S\ref{sec:setting}. The Gaussian law $\mu_Q=\cN(0,Q)$ is carried by $\mathbb X$; its regularity is established in \cref{app:nonlinear-gibbs-verification}. Let $\kappa=Z^{-1}e^{-V}\mu_Q$, where $0<Z<\infty$ and $V$ is finite $\mu_Q$-almost everywhere.

Solutions are continuous adapted processes satisfying the integral equation in the Banach or Hilbert state space specified for each model. A probabilistically strong solution is defined for a prescribed initial variable and Wiener process and is adapted to their completed filtration. Pathwise uniqueness means that solutions driven by the same initial variable and noise are indistinguishable.

Write $\cF C_b^\infty$ for functions of finitely many real Fourier coefficients whose defining function belongs to $C_b^\infty(\mathbb R^n)$, so it and all its derivatives are bounded. On these functions, It\^o's formula gives the differential expression \eqref{eq:formal-generator}. The weighted integration-by-parts calculation following \eqref{eq:G-IBP} identifies its energy with the gradient form below. In terms of the Cameron--Martin gradient,
\begin{equation}\label{eq:K-gradient-convention}
    \nabla_{\mathbb K}F=Q\nabla_{\mathbb H}F,\qquad
    \partial_hF=\langle\nabla_{\mathbb K}F,h\rangle_{\mathbb K},
    \quad h\in\mathbb K,
\end{equation}
\begin{equation}\label{eq:pre-form}
    \cE_0(F,G)
    :=\int_{\mathbb X}\langle\nabla_{\mathbb K}F,\nabla_{\mathbb K}G\rangle_{\mathbb K}\,d\kappa
    =\int_{\mathbb X}\langle Q\nabla_{\mathbb H}F,\nabla_{\mathbb H}G\rangle_{\mathbb H}\,d\kappa,
\end{equation}
with $-\int F\cL G\,d\kappa=\cE_0(F,G)$ under \cref{ass:gibbs-ibp}. The same hypotheses give closability by \cref{prop:closability-criterion}. Closing $\cE_0$ extends the energy from smooth cylinder observables to its $L^2(\kappa)$ form domain. Once this closed form has the transition semigroup of the prescribed equation, the abstract results give the following consequences.

\begin{corollary}[Entropy dissipation and equilibrium for the Langevin equation]
\label{cor:langevin-entropy-equilibrium}
With $\mathbb X,Q,V,\kappa$ as above, assume that the pre-form \eqref{eq:pre-form} on $\cF C_b^\infty$ is closable in $L^2(\kappa)$. Let $\cE$ be its closure, with semigroup $P_t$ and generator $L$, and let $D_{\mathbb K}$ denote the closed cylinder gradient. Suppose \eqref{eq:equilibrium-langevin} admits a conservative time-homogeneous Markov family of continuous $\mathbb X$-valued solutions whose transition operators agree with $P_t$ in $L^2(\kappa)$. Then $\kappa$ is invariant, and the following statements hold.
\begin{enumerate}[label=(\roman*)]
    \item For an initial law $\mu_0=h_0\kappa$ with $\Ent_\kappa(h_0)<\infty$, the time-$t$ law is $\mu_t=h_t\kappa$, where $h_t=P_th_0$, and
    \begin{equation}\label{eq:spde-integrated-dissipation}
        \Ent_\kappa(h_t)
        +4\int_0^t\!\int_{\mathbb X}\|D_{\mathbb K}\sqrt{h_s}\|_{\mathbb K}^2\,d\kappa\,ds
        \le\Ent_\kappa(h_0).
    \end{equation}
    The gradient exists for almost every $s>0$. If additionally $h_0\in\operatorname{Dom}(L)$ and $0<c_-\le h_0\le c_+<\infty$, the integrated relation is an equality and the nonadiabatic rate in information units is
    \begin{equation}\label{eq:spde-nonadiabatic-rate}
        \sigma_{\rm na}(t)
        :=4\int_{\mathbb X}\|D_{\mathbb K}\sqrt{h_t}\|_{\mathbb K}^2\,d\kappa
        =-\frac{d}{dt}\Ent_\kappa(h_t)
        \qquad\text{for a.e. }t>0,
    \end{equation}
    \item If $\Ent_\kappa(u^2)\le(2/\rho)\cE(u,u)$ for every $u\in\operatorname{Dom}(\cE)$ and some $\rho>0$, then every finite-entropy initial density satisfies
    \begin{equation*}
        \Ent_\kappa(h_t)\le e^{-2\rho t}\Ent_\kappa(h_0),\qquad t\ge0.
    \end{equation*}
    \item The stationary process satisfies detailed balance. For every $T>0$, its path law $\Pbb_{\kappa,T}$ on $C([0,T];\mathbb X)$ obeys
    \begin{equation}\label{eq:langevin-equilibrium-path}
        \Pbb_{\kappa,T}=\Pbb_{\kappa,T}\circ r_T^{-1},
        \qquad
        \frac1T\RelEnt{\Pbb_{\kappa,T}}{\Pbb_{\kappa,T}\circ r_T^{-1}}=0.
    \end{equation}
    In particular, its steady-state path-space entropy-production rate is zero.
\end{enumerate}
\end{corollary}

\begin{proof}
The pre-form is closable on the stated Fourier cylinders, so \cref{prop:closed-gradient-properties} gives \cref{ass:equilibrium,ass:diffusion-calculus} and
\begin{equation*}
    \cE(u,u)=\int_{\mathbb X}\|D_{\mathbb K}u\|_{\mathbb K}^2\,d\kappa,
    \qquad u\in\operatorname{Dom}(\cE).
\end{equation*}
By \cref{prop:form-invariance}, $P_t$ is self-adjoint and preserves $\kappa$. Its agreement with the SPDE transition operators gives invariance for the equation. For bounded $h_0$ and bounded Borel $F$, self-adjointness gives
\begin{equation}\label{eq:spde-density-duality}
    \int F\,d\mu_t=\int h_0P_tF\,d\kappa
    =\int F P_th_0\,d\kappa.
\end{equation}
Truncation and $L^1$-contractivity extend the last equality to every probability density $h_0$. Thus $\mu_t=(P_th_0)\kappa$. \Cref{thm:deBruijn}(ii) now gives \eqref{eq:spde-integrated-dissipation}; part~(i), under the additional domain and positive-bound conditions, gives equality and \eqref{eq:spde-nonadiabatic-rate}.

For part~(ii), restart \eqref{eq:spde-integrated-dissipation} from $h_s$ and apply the log-Sobolev inequality to $\sqrt{h_r}$ for almost every $r$. For $0\le s\le t$ this yields
\begin{equation*}
    \Ent_\kappa(h_t)+2\rho\int_s^t\Ent_\kappa(h_r)\,dr
    \le\Ent_\kappa(h_s).
\end{equation*}
The integral form of Gr\"onwall's inequality gives the asserted exponential bound. Finally, $\mathbb X=C(\T)$ is Polish and the given Markov solutions are conservative and continuous. Their transition semigroup is the symmetric form semigroup, so \cref{thm:eq-zero} applies and proves part~(iii).
\end{proof}

For these overdamped fields, the reversal map $r_T$ reverses the order of observations while leaving the field values unchanged. Equality of the stationary path law $\Pbb_{\kappa,T}$ and its reversal in \eqref{eq:langevin-equilibrium-path} means that the fluctuations have no preferred direction of time, so their forward--reverse relative entropy vanishes on every observation interval. This dynamical equilibrium property is stronger than invariance of the one-time Gibbs law $\kappa$.

An initial law $h_0\kappa\ne\kappa$ can nevertheless relax irreversibly under these equilibrium dynamics. For the regular densities in part~(i), \eqref{eq:spde-nonadiabatic-rate} measures this relaxation; under a thermodynamic interpretation with local detailed balance, $k_BT_{\mathrm{phys}}\Ent_\kappa(h_t)$ is the excess free energy and $k_B\sigma_{\rm na}(t)$ its associated entropy-production rate. Here $k_B$ is Boltzmann's constant and $T_{\mathrm{phys}}$ is the physical temperature. At equilibrium, $h=1$ and this relaxation term vanishes together with the stationary path-space rate.

\begin{remark}[Verifying the hypotheses for a given equation]
\label{rmk:form-construction}
For the quartic potentials in \cref{thm:Phi4-model,rmk:quartic-perturbation}, Hairer, Stuart and Voss provide strong well-posedness and symmetry of the Gibbs semigroup \cite[Theorems~2.10 and~3.6]{HairerStuartVoss2007}. We then verify that the energy form of this semigroup is the closure of \eqref{eq:pre-form}, including equality of the full form domains; see \cref{lem:quartic-direct-energy,prop:quartic-model-verification}. The model results in \S\ref{sec:examples} apply the entropy results and establish their log-Sobolev bounds; admissible initial laws are verified in \cref{prop:quartic-model-verification}(iv). Alternatively, one may start from the Gibbs pre-form, use the closability criterion in \cref{prop:closability-criterion}, and construct an associated diffusion under quasi-regularity. The coordinate martingales of this form-associated diffusion, together with uniqueness in law for the target SPDE in the corresponding class of solutions, can then identify the diffusion's path law with that of the SPDE for starting points outside a $\cE$-exceptional set \cite{MaRoeckner1992,AlbeverioMaRockner2015,RocknerZhuZhu2017}.
\end{remark}

\section{Systems without Detailed Balance}
\label{sec:linear}

A stationary circulation can preserve the reference measure while breaking time-reversal symmetry. For the Gaussian dynamics considered here, we separate this measure-preserving transport from the symmetric evolution at the semigroup level. The separation identifies the relative-entropy dissipation during relaxation and the current energy that persists at stationarity. Set $V=0$ and let
\begin{equation}\label{eq:B-assumptions}
    B\in\mathcal L(\mathbb H),\qquad B^*=-B.
\end{equation}
Here $\mathcal L(\mathbb S)$ denotes the bounded linear operators on a Banach space $\mathbb S$, and $B^*$ is the Hilbert-space adjoint. For the external force
\begin{equation}\label{eq:linear-force}
    f(\phi)=B\phi,
\end{equation}
the SPDE~\eqref{eq:main} becomes
\begin{equation}\label{eq:linear}
    d\phi_t=-\mathsf K_B\phi_t\,dt+\sqrt{2}\,dW^Q_t,
    \qquad
    \mathsf K_{\pm B}=I\pm QB,
\end{equation}
with $Q=A^{-1}$ and $I$ the identity on $\mathbb H$. The case $B=0$ is the reversible Ornstein--Uhlenbeck field. Its housekeeping term and stationary path-space entropy-production rate vanish, while the transient formulas below describe relaxation after a mass quench. In finite dimensions the skew drift is the conservative part of the symmetric--antisymmetric decomposition of an irreversible diffusion \cite{Qian2013}; on a Hilbert space \eqref{eq:linear} is a non-symmetric Ornstein--Uhlenbeck equation \cite{Fuhrman1995,ChojnowskaMichalikGoldys1996,ChojnowskaMichalikGoldys2002}.

The semigroup and stationary arguments below apply to every positive, injective trace-class covariance on a real separable Hilbert space.

\subsection{Gaussian Dynamics and Circulation}
\label{sec:linear-cov}

The linear equation admits an explicit solution, whose covariance determines the invariant Gaussian law. The skew drift then defines a Cameron--Martin current and a measure-preserving flow.

For the spectral calculations, we fix the real orthonormal Fourier basis
\begin{equation}\label{eq:basis}
    e_0 = (2\pi)^{-1/2},\qquad
    e_k^c = \pi^{-1/2}\cos(kx),\qquad
    e_k^s = \pi^{-1/2}\sin(kx),\quad k\ge 1.
\end{equation}
This basis diagonalises $A$ and $Q$. The eigenvalues are $\lambda_0 = m$ (multiplicity $d_0=1$) and $\lambda_k = k^2+m$ (multiplicity $d_k = 2$ for $k\ge 1$). In one spatial dimension, whenever the real basis is enumerated by a single index $j$, we write $Ae_j=\lambda_je_j$ and $Qe_j=q_je_j$, so $q_j=\lambda_j^{-1}$; thus $q_0=m^{-1}$ and $q_{k,c}=q_{k,s}=(k^2+m)^{-1}$ for $k\ge1$. With $\Tr$ denoting the operator trace,
\begin{equation}\label{eq:trQ}
    \Tr Q = \frac{1}{m} + 2\sum_{k\ge 1}\frac{1}{k^2+m} < \infty.
\end{equation}

For $N\ge0$, let $\mathbb H_N=\operatorname{span}\{e_0,e_k^c,e_k^s:1\le k\le N\}$ and let $P_N$ be the $\mathbb H$-orthogonal projection onto $\mathbb H_N$. From this point onward, $N$ in $\mathbb H_N,P_N,A_N,Q_N$ is a frequency cutoff, not an enumeration cutoff. We write $A_N=P_NA|_{\mathbb H_N}$ and $Q_N=P_NQ|_{\mathbb H_N}=A_N^{-1}$, and $I_N$ denotes the identity on $\mathbb H_N$.

\subsubsection{Well-Posedness and the Invariant Gaussian Law}

\begin{proposition}[Global well-posedness and invariant Gaussian law]
\label{prop:skew-invariance}
For every $x\in \mathbb H$ and every prescribed $\mathbb H$-valued $Q$-Wiener process on a filtered probability space, \eqref{eq:linear} admits a unique global probabilistically strong $\mathbb H$-valued solution with continuous paths. It is pathwise unique, depends continuously on $x$, and is given by
\begin{equation*}
    \phi_t^x=e^{-t\mathsf K_B}x + \sqrt2\int_0^t e^{-(t-s)\mathsf K_B}\,dW_s^Q.
\end{equation*}
In particular, these solutions define a Borel Markov family. Moreover,
\begin{equation}\label{eq:covariance-transport}
    e^{-t\mathsf K_B}Qe^{-t\mathsf K_B^*}=e^{-2t}Q,
    \qquad
    2\int_0^t e^{-s\mathsf K_B}Qe^{-s\mathsf K_B^*}\,ds=(1-e^{-2t})Q.
\end{equation}
Consequently, $\mu_Q=\cN(0,Q)$ is invariant.
\end{proposition}

\begin{proof}
Since $\mathsf K_B=I+QB$ is bounded and $Q$ is trace class, the displayed stochastic convolution is $\mathbb H$-valued, adapted to the prescribed noise, and has a continuous version. The variation-of-constants formula and stochastic Fubini show that it satisfies, in $\mathbb H$,
\begin{equation*}
    \phi_t^x=x-\int_0^t\mathsf K_B\phi_s^x\,ds+\sqrt2\,W_t^Q.
\end{equation*}
Hence it is a global strong solution in the convention of \cref{sec:conventions}. If $\phi$ and $\widetilde\phi$ are two solutions with the same initial state and noise, then
\begin{equation*}
    \|\phi_t-\widetilde\phi_t\|_{\mathbb H}
    \le \|\mathsf K_B\|_{\mathcal L(\mathbb H)} \int_0^t \|\phi_s-\widetilde\phi_s\|_{\mathbb H}\,ds.
\end{equation*}
By Gronwall's lemma, the two solutions are indistinguishable, which proves pathwise uniqueness. The identity $\phi_t^x-\phi_t^y=e^{-t\mathsf K_B}(x-y)$ gives continuous dependence on the initial state; this continuity and the independent increments of $W^Q$ also establish the Borel Markov property.

Skew-adjointness of $B$ gives the Lyapunov identity
\begin{equation*}
    \mathsf K_BQ+Q\mathsf K_B^*
    = (I+QB)Q+Q(I+B^*Q) = 2Q.
\end{equation*}
For $\mathsf R_t^{\rm cov}=e^{-t\mathsf K_B}Qe^{-t\mathsf K_B^*}$, differentiation in trace norm therefore yields
\begin{equation*}
    (\mathsf R_t^{\rm cov})'
    = -e^{-t\mathsf K_B}(\mathsf K_BQ+Q\mathsf K_B^*)e^{-t\mathsf K_B^*}
    = -2\mathsf R_t^{\rm cov}.
\end{equation*}
Since $\mathsf R_0^{\rm cov}=Q$, this proves the first identity in \eqref{eq:covariance-transport}; integration gives the second. If $\phi_0\sim\mu_Q$ independently of the noise, the two independent terms in the mild solution have covariances $e^{-2t}Q$ and $(1-e^{-2t})Q$, so their sum again has law $\mu_Q$.
\end{proof}

\subsubsection{Generator Decomposition and Cameron--Martin Current}

When $V=0$, the symmetric part is the Gaussian Dirichlet-form generator from \S\ref{sec:equilibrium}, while the skew drift gives the antisymmetric part. The next two propositions represent the antisymmetric part by a Cameron--Martin current and identify its mean squared Cameron--Martin norm under the invariant Gaussian law with the squared Hilbert--Schmidt norm used in the path-space comparison below.

On cylinder functions $F\in\cF C_b^\infty$, the formal generator of~\eqref{eq:linear} is
\begin{equation}\label{eq:L-linear}
    L_BF(\phi)
    = \Tr(Q\nabla_{\mathbb H}^2F(\phi)) -\langle(I+QB)\phi,\nabla_{\mathbb H}F(\phi)\rangle_{\mathbb H}.
\end{equation}
Then we have
\begin{proposition}[Symmetric and antisymmetric parts of the Gaussian generator]\label{prop:L-decomp} For every $B$ satisfying~\eqref{eq:B-assumptions}, set
\begin{equation*}
    L_sF
    = \Tr(Q\nabla_{\mathbb H}^2F)-\langle\phi,\nabla_{\mathbb H}F\rangle_{\mathbb H},
    \qquad
    L_a^BF
    = -\langle QB\phi,\nabla_{\mathbb H}F\rangle_{\mathbb H}.
\end{equation*}
On the cylinder test algebra $\cF C_b^\infty$, $L_s$ is symmetric with $\Gamma(F)=\|Q^{1/2}\nabla_{\mathbb H}F\|_{\mathbb H}^2$, $L_a^B$ is antisymmetric, and $L_B=L_s+L_a^B$.
\end{proposition}

\begin{proof}
Gaussian integration by parts gives the symmetry of $L_s$ and the stated carr\'e du champ. For the antisymmetric part, set $\mathsf U_t=e^{-tQB}=e^te^{-t\mathsf K_B}$. Equation~\eqref{eq:covariance-transport} gives $\mathsf U_tQ\mathsf U_t^*=Q$, so $\mathsf U_t$ preserves $\mu_Q$. Differentiating $\int_{\mathbb H}F(\mathsf U_t\phi)G(\mathsf U_t\phi)\,d\mu_Q(\phi)$ at $t=0$ gives
\begin{equation*}
    \int_{\mathbb H}(L_a^BF)G\,d\mu_Q=-\int_{\mathbb H}F(L_a^BG)\,d\mu_Q.
\end{equation*}
The decomposition follows directly from~\eqref{eq:L-linear}.
\end{proof}

\begin{proposition}[Cameron--Martin current and its stationary energy]
\label{prop:La-representation}
For $F\in\cF C_b^\infty$,
\begin{equation}\label{eq:La-CM}
    L_a^BF(\phi)=\langle J_B^{\mathbb K}(\phi),\nabla_{\mathbb K}F(\phi)\rangle_{\mathbb K},
    \qquad J_B^{\mathbb K}(\phi)=-QB\phi,
\end{equation}
where $\nabla_{\mathbb K}=Q\nabla_{\mathbb H}$ on cylinder functions and $\mathbb K=Q^{1/2}\mathbb H$. We also write $J_B(\phi):=Q^{-1/2}J_B^{\mathbb K}(\phi)=-Q^{1/2}B\phi$ for the $\mathbb H$-coordinate of this current. Moreover, $J_B^{\mathbb K}(\phi)\in \mathbb K$ for every $\phi\in\mathbb H$, and
\begin{equation}\label{eq:J-norm}
    \E_{\mu_Q}\|Q^{-1/2}J_B^{\mathbb K}\|_{\mathbb H}^2
    = \Tr(B^*QBQ)
    = \|Q^{1/2}BQ^{1/2}\|_{\mathrm{HS}}^2.
\end{equation}
This value vanishes if and only if $B=0$.
\end{proposition}

\begin{proof}
Since $QB\phi=Q^{1/2}(Q^{1/2}B\phi)$, the vector $J_B^{\mathbb K}(\phi)$ belongs to $\mathbb K$ for every $\phi\in\mathbb H$, with $\|J_B^{\mathbb K}(\phi)\|_{\mathbb K}=\|Q^{1/2}B\phi\|_{\mathbb H}\le\|Q^{1/2}B\|\,\|\phi\|_{\mathbb H}$. Also,
\begin{align*}
    \langle J_B^{\mathbb K},\nabla_{\mathbb K}F\rangle_{\mathbb K}
    &= \langle-Q^{1/2}B\phi,Q^{1/2}\nabla_{\mathbb H}F\rangle_{\mathbb H}
    = -\langle QB\phi,\nabla_{\mathbb H}F\rangle_{\mathbb H}=L_a^BF,\\
    \E_{\mu_Q}\|Q^{-1/2}J_B^{\mathbb K}\|_{\mathbb H}^2
    &= \E_{\mu_Q}\|Q^{1/2}B\phi\|_{\mathbb H}^2
    = \Tr(B^*QBQ)
    = \|Q^{1/2}BQ^{1/2}\|_{\mathrm{HS}}^2.
\end{align*}
The trace is finite because $Q^{1/2}$ is Hilbert--Schmidt and $B$ is bounded.

If this trace vanishes, then $Q^{1/2}BQ^{1/2}=0$. Injectivity of $Q^{1/2}$ gives $BQ^{1/2}=0$. Since the range of $Q^{1/2}$ is dense and $B$ is bounded, it follows that $B=0$.
\end{proof}

The preceding antisymmetric operator generates a transport that preserves the Gaussian reference measure. Its action on the full semigroup gives a direct relation between circulation and entropy relaxation.

\begin{proposition}[Circulation and entropy relaxation]
\label{prop:circulation-semigroup}
Assume \eqref{eq:B-assumptions}. Let $P_t^B$ be the transition semigroup of \eqref{eq:linear}, and write $P_t^0$ for the case $B=0$. Define, for $t\in\mathbb R$,
\begin{equation}\label{eq:circulation-flow}
    \mathsf A_B=Q^{1/2}BQ^{1/2},\qquad
    \mathsf U_t=e^{-tQB},\qquad
    \mathsf O_t=e^{-t\mathsf A_B}.
\end{equation}
Then $\mathsf O_t$ is orthogonal on $\mathbb H$, $\mathsf U_t$ preserves $\mu_Q$, and
\begin{equation}\label{eq:circulation-intertwining}
    \mathsf U_tQ^{1/2}=Q^{1/2}\mathsf O_t,\qquad
    \|\mathsf U_th\|_{\mathbb K}=\|h\|_{\mathbb K}
    \quad(h\in\mathbb K).
\end{equation}
For every bounded Borel $F$ and $t\ge0$,
\begin{equation}\label{eq:semigroup-circulation}
    P_t^BF=(P_t^0F)\circ\mathsf U_t
    =P_t^0(F\circ\mathsf U_t).
\end{equation}
If the initial law is $h_0\mu_Q$, its time-$t$ density and relative entropy satisfy
\begin{equation}\label{eq:density-circulation}
    h_t^B=(P_t^0h_0)\circ\mathsf U_{-t},
    \qquad
    \Ent_{\mu_Q}(h_t^B)=\Ent_{\mu_Q}(P_t^0h_0).
\end{equation}
The entropy equality holds in $[0,\infty]$; both sides are finite whenever $\Ent_{\mu_Q}(h_0)<\infty$.

Let $\cE_Q$ be the closure of \eqref{eq:pre-form} with $\kappa=\mu_Q$, and let $D_{\mathbb K}$ be the corresponding closed cylinder gradient in $L^2(\mu_Q)$. Set $D_Q=Q^{-1/2}D_{\mathbb K}$, the $\mathbb H$-valued gradient. Composition with $\mathsf U_t$ maps $\operatorname{Dom}(\cE_Q)$ onto itself, and
\begin{align}
    D_Q(u\circ\mathsf U_t)
    &=\mathsf O_t^*(D_Qu)\circ\mathsf U_t,
    \label{eq:gradient-circulation}\\
    \cE_Q(u\circ\mathsf U_t,u\circ\mathsf U_t)
    &=\cE_Q(u,u).
    \label{eq:energy-circulation}
\end{align}
\end{proposition}

\begin{proof}\leavevmode
\par\noindent\textbf{Step 1. The flow and the transition semigroup.}\par\nopagebreak
Since $\mathsf A_B^*=-\mathsf A_B$, the operators $\mathsf O_t$ are orthogonal. The bounded-operator relation $(QB)Q^{1/2}=Q^{1/2}\mathsf A_B$ gives the first equality in \eqref{eq:circulation-intertwining} by the exponential series. Writing $h=Q^{1/2}z$ gives its norm statement. Equation~\eqref{eq:covariance-transport}, with $\mathsf U_t=e^te^{-t\mathsf K_B}$, gives preservation of $\mu_Q$ for $t\ge0$; since $\mathsf U_t^{-1}=\mathsf U_{-t}$, it holds for every $t\in\mathbb R$. By \cref{prop:skew-invariance}, the transition from $x$ is Gaussian with mean $e^{-t}\mathsf U_tx$ and covariance $(1-e^{-2t})Q$. Thus
\begin{equation*}
    P_t^BF(x)
    =\int_{\mathbb H}F(e^{-t}\mathsf U_tx+y)
    \,\mathcal N(0,(1-e^{-2t})Q)(dy).
\end{equation*}
This is $(P_t^0F)(\mathsf U_tx)$. The change of variable $y=\mathsf U_tz$ preserves the Gaussian measure in this integral and gives $P_t^0(F\circ\mathsf U_t)(x)$.

\par\noindent\textbf{Step 2. Density evolution and entropy.}\par\nopagebreak
Self-adjointness of $P_t^0$ and invariance under $\mathsf U_t$ give, for bounded Borel $F$,
\begin{equation*}
    \int h_0P_t^BF\,d\mu_Q
    =\int h_0P_t^0(F\circ\mathsf U_t)\,d\mu_Q
    =\int (P_t^0h_0)(F\circ\mathsf U_t)\,d\mu_Q
    =\int \bigl[(P_t^0h_0)\circ\mathsf U_{-t}\bigr]F\,d\mu_Q.
\end{equation*}
For an $L^1$ density, the middle equality follows from bounded truncations and $L^1$-contractivity. This proves the density formula. A further change of variable gives the entropy equality, including the extended-valued case because $r\log r\ge-e^{-1}$. Jensen's inequality for the invariant Markov semigroup $P_t^0$ gives $\Ent_{\mu_Q}(P_t^0h_0)\le\Ent_{\mu_Q}(h_0)$.

\par\noindent\textbf{Step 3. The closed gradient.}\par\nopagebreak
We first allow bounded smooth cylinders with arbitrary directions $a_1,\ldots,a_n\in\mathbb H$. Replacing each $a_i$ by its Fourier projection $P_Na_i$ gives convergence in $L^2(\mu_Q)$ and of the gradients in $L^2(\mu_Q;\mathbb H)$. Indeed, the coordinate variances of the differences are $\|Q^{1/2}(P_Na_i-a_i)\|_{\mathbb H}^2\to0$; the first derivatives of the defining smooth function are bounded and continuous. The cylinder chain rule and dominated convergence therefore identify the closed gradient with $Q^{1/2}\nabla_{\mathbb H}$ for these cylinders. For such a cylinder $F$, the adjoint of \eqref{eq:circulation-intertwining} gives
\begin{equation*}
    D_Q(F\circ\mathsf U_t)
    =Q^{1/2}\mathsf U_t^*(\nabla_{\mathbb H}F)\circ\mathsf U_t
    =\mathsf O_t^*(D_QF)\circ\mathsf U_t.
\end{equation*}
Both composition with $\mathsf U_t$ on $L^2(\mu_Q)$ and the resulting transformation on $L^2(\mu_Q;\mathbb H)$ are isometries. Approximate $u\in\operatorname{Dom}(\cE_Q)$ in the form norm by Fourier cylinders. Their compositions and transformed gradients are Cauchy in the respective $L^2$ spaces, so closedness of $D_Q$ proves \eqref{eq:gradient-circulation}. Taking squared norms proves \eqref{eq:energy-circulation}. Applying the same argument with $-t$ proves that composition maps the domain onto itself.
\end{proof}

\subsection{Circulation and Entropy Dissipation}
\label{sec:circulation-dissipation}

For a mass-quench initial law, the covariance formula and \cref{prop:circulation-semigroup} give an explicit spectral formula for the relative-entropy dissipation and a trace formula for the energy of the circulating current throughout relaxation.

\subsubsection{Mass-Quench Covariance}

Take the initial law $\mu_0=\cN(0,A_0^{-1})$, where $A_0=-\partial_x^2+m_0$ and $m_0>0$. Here ``mass quench'' means that the initial covariance uses $m_0$, while the subsequent dynamics and its invariant covariance use $m$. Write
\begin{equation*}
    \overline C_t=e^{-2t}A_0^{-1}+(1-e^{-2t})Q,
    \qquad \overline\mu_t=\cN(0,\overline C_t).
\end{equation*}
Thus $\overline\mu_t$ is the time-$t$ law of the symmetric dynamics started from the same initial law $\mu_0$.

\begin{lemma}[Mass-quench covariance]\label{lem:cov}
Under \eqref{eq:B-assumptions}, the law of \eqref{eq:linear} is $\mu_t=\cN(0,C_t^B)$, where
\begin{equation}\label{eq:Ct-B}
    C_t^B=\mathsf U_t\overline C_t\mathsf U_t^*,
    \qquad \mu_t=(\mathsf U_t)_\#\overline\mu_t.
\end{equation}
Here $(\mathsf U_t)_\#\overline\mu_t$ denotes the pushforward of $\overline\mu_t$ by $\mathsf U_t$. Moreover, $\mu_Q$ is the unique invariant probability measure, and the law converges weakly to $\mu_Q$ from every initial probability distribution on $\mathbb H$.
\end{lemma}

\begin{proof}\leavevmode
The covariance formula in \cref{prop:skew-invariance}, together with $e^{-t\mathsf K_B}=e^{-t}\mathsf U_t$, gives
\begin{equation*}
    C_t^B
    =e^{-2t}\mathsf U_t A_0^{-1}\mathsf U_t^*+(1-e^{-2t})Q
    =\mathsf U_t\bigl[e^{-2t}A_0^{-1}+(1-e^{-2t})Q\bigr]\mathsf U_t^*,
\end{equation*}
because $\mathsf U_tQ\mathsf U_t^*=Q$ by \cref{prop:circulation-semigroup}. This proves \eqref{eq:Ct-B}.

For convergence from a general initial law, the intertwining relation in \cref{prop:circulation-semigroup} yields
\begin{equation*}
    \mathsf U_sQB=Q^{1/2}\mathsf O_sQ^{1/2}B,
    \qquad \|\mathsf U_sQB\|\le\|Q\|\|B\|.
\end{equation*}
Integrating $\mathsf U_t'=-\mathsf U_tQB$ therefore gives
\begin{equation*}
    \|e^{-t\mathsf K_B}\|
    \le e^{-t}\bigl(1+t\|Q\|\|B\|\bigr)\longrightarrow0.
\end{equation*}
Consequently, for any $\mathbb H$-valued initial random variable independent of the driving noise, the initial contribution tends to zero almost surely. The stochastic convolution has law $\cN(0,(1-e^{-2t})Q)$, which converges weakly to $\mu_Q$. The mild solution thus converges in law to $\mu_Q$. Since \cref{prop:skew-invariance} already gives invariance, this also proves uniqueness.
\end{proof}

\subsubsection{Onsager Decomposition}

Set $\lambda_k^0=k^2+m_0$ and
\begin{equation}\label{eq:Tk}
    T_k(t)=e^{-2t}\Bigl(\frac{\lambda_k}{\lambda_k^0}-1\Bigr).
\end{equation}
Write $\mathbb V_0=\operatorname{span}\{e_0\}$ and $\mathbb V_k=\operatorname{span}\{e_k^c,e_k^s\}$ for $k\ge1$, and denote the corresponding $\mathbb H$-orthogonal projections by $P_{\mathbb V_k}$.

For the Galerkin expressions, let $\mu_t^N$ be the law on $\mathbb H_N$ of
\begin{equation*}
    d\phi_t^N=-(I_N+Q_NB_N)\phi_t^N\,dt+\sqrt2\,dW_t^{Q_N},
    \qquad \phi_0^N\sim\cN(0,P_NA_0^{-1}P_N),
\end{equation*}
where $Q_N=P_NQP_N$ and $B_N=P_NBP_N$. Set $\mu_Q^N=\cN(0,Q_N)$, $h_t^N=d\mu_t^N/d\mu_Q^N$, and $v_{\rm ss}^N(\phi_N)=-Q_NB_N\phi_N$. The quantities below are Onsager current quadratic forms: the squared $L^2(\mu_t;\mathbb H)$ norms of the Cameron--Martin coordinates of the relaxation and circulation fields.

\begin{theorem}[Gaussian entropy dissipation and infinite-dimensional Onsager decomposition]\label{thm:lin-hk}
Assume \eqref{eq:B-assumptions} and take the mass-quench initial law $\mu_0=\cN(0,A_0^{-1})$ specified above. For every $t\ge0$, $\mu_t\sim\mu_Q$, and
\begin{equation}\label{eq:rel-entropy}
    \RelEnt{\mu_t}{\mu_Q}
    =\frac12\sum_{k\ge0}d_k\bigl[T_k-\log(1+T_k)\bigr]<\infty.
\end{equation}
Define
\begin{equation}\label{eq:sigma-na-hk}
    \begin{aligned}
        \sigma_{\rm na}(t)&=\sum_{k\ge0}d_k\frac{T_k^2}{1+T_k},\\
        \sigma_{\rm hk}(t)&=\Tr(B^*QB C_t^B)
        =\Tr(B^*QB\,\overline C_t).
    \end{aligned}
\end{equation}
Then
\begin{equation}\label{eq:entropy-decay}
    \frac{d}{dt}\RelEnt{\mu_t}{\mu_Q}=-\sigma_{\rm na}(t).
\end{equation}
The fields
\begin{equation}\label{eq:Onsager-fields}
    \begin{aligned}
        \overline G_t(\psi)
        &=\sum_{k\ge0}\frac{\sqrt{\lambda_k}T_k}{1+T_k}P_{\mathbb V_k}\psi,\\
        G_t^B(\phi)&=\mathsf O_t\overline G_t(\mathsf U_{-t}\phi),
        \qquad J_B(\phi)=-Q^{1/2}B\phi
    \end{aligned}
\end{equation}
are well defined, with $\overline G_t\in L^2(\overline\mu_t;\mathbb H)$ and $G_t^B,J_B\in L^2(\mu_t;\mathbb H)$. If $h_t=d\mu_t/d\mu_Q$ and $D_Q$ denotes the $\mathbb H$-coordinate of the closed Cameron--Martin gradient, then
\begin{equation*}
    \sqrt{h_t}\in\operatorname{Dom}(\cE_Q),
    \qquad D_Q\sqrt{h_t}=\frac12\sqrt{h_t}\,G_t^B,
    \qquad 4\cE_Q(\sqrt{h_t},\sqrt{h_t})=\sigma_{\rm na}(t).
\end{equation*}
Moreover,
\begin{equation}\label{eq:Onsager-orthogonality}
    \E_{\mu_t}\langle J_B,G_t^B\rangle_{\mathbb H}=0.
\end{equation}
The total current quadratic form satisfies
\begin{equation}\label{eq:sigma-tot-limit}
    \sigma_{\rm tot}(t)
    :=\E_{\mu_t}\|J_B-G_t^B\|_{\mathbb H}^2
    =\lim_{N\to\infty}\int_{\mathbb H_N}
    \bigl\|Q_N^{-1/2}(v_{\rm ss}^N-Q_N\nabla_{\mathbb H_N}\log h_t^N)\bigr\|_{\mathbb H_N}^2\,d\mu_t^N,
\end{equation}
where the finite-dimensional integrals increase with $N$, and
\begin{equation}\label{eq:onsager-decomp}
    \sigma_{\rm tot}(t)=\sigma_{\rm na}(t)+\sigma_{\rm hk}(t).
\end{equation}
The stationary value of the circulation term, denoted by the superscript ${\rm ss}$, is
\begin{equation}\label{eq:hk-ss}
    \sigma_{\rm hk}^{\rm ss}
    =\Tr(B^*QBQ)
    =\|Q^{1/2}BQ^{1/2}\|_{\mathrm{HS}}^2.
\end{equation}
It is positive exactly when $B\ne0$ and equals the stationary forward--reverse path-space relative-entropy rate in \cref{thm:Galerkin-path-KL}.
\end{theorem}

\begin{proof}\leavevmode
\par\noindent\textbf{Step 1. Equivalence and entropy dissipation.}\par\nopagebreak
The eigenvalues of $\overline C_t$ on $\mathbb V_k$ are $(1+T_k)/\lambda_k$. The numbers $1+T_k$ are uniformly bounded above and away from zero, and $T_k=O(k^{-2})$. Thus the Cameron--Martin ranges of $\overline C_t$ and $Q$ agree, and their relative covariance perturbation is Hilbert--Schmidt:
\begin{equation*}
    \operatorname{Ran}\overline C_t^{1/2}=\operatorname{Ran}Q^{1/2},
    \qquad \sum_{k\ge0}d_kT_k^2<\infty.
\end{equation*}
The centred Feldman--H\'ajek theorem \cite{Bogachev1998} gives $\overline\mu_t\sim\mu_Q$. Since $\mathsf U_t$ is invertible and preserves $\mu_Q$, \eqref{eq:Ct-B} gives $\mu_t\sim\mu_Q$ and
\begin{equation*}
    h_t=\overline h_t\circ\mathsf U_{-t},
    \qquad \overline h_t=\frac{d\overline\mu_t}{d\mu_Q},
    \qquad \RelEnt{\mu_t}{\mu_Q}=\RelEnt{\overline\mu_t}{\mu_Q}.
\end{equation*}
The Fourier coordinates generate the Borel sigma-field. Applying \cref{lem:entropy-projections} to $\overline\mu_t\ll\mu_Q$, and using the finite-dimensional Gaussian formula, gives \eqref{eq:rel-entropy}; the series converges because $T_k-\log(1+T_k)=O(T_k^2)$. Since $T_k'=-2T_k$,
\begin{equation*}
    -\frac{d}{dt}\frac12\bigl[T_k-\log(1+T_k)\bigr]
    =\frac{T_k^2}{1+T_k}.
\end{equation*}
The uniform lower bound on $1+T_k$ and the square-summable bound on $T_k$ justify termwise differentiation locally uniformly for $t\ge0$, proving \eqref{eq:entropy-decay}.

\par\noindent\textbf{Step 2. Closed gradients and current energies.}\par\nopagebreak
The diagonal Gaussian measure $\overline\mu_t$ is equivalent to $\mu_Q$, and $\sum_kd_kT_k^2/(1+T_k)<\infty$. These are the hypotheses of \cref{lem:gaussian-density-energy}, which gives
\begin{equation*}
    D_Q\sqrt{\overline h_t}
    =\frac12\sqrt{\overline h_t}\,\overline G_t,
    \qquad
    \E_{\overline\mu_t}\|\overline G_t\|_{\mathbb H}^2
    =\sum_{k\ge0}d_k\frac{T_k^2}{1+T_k}.
\end{equation*}
In particular, the series defining $\overline G_t$ converges in $L^2(\overline\mu_t;\mathbb H)$. The closed-gradient transformation in \cref{prop:circulation-semigroup}, applied to $\sqrt{\overline h_t}\circ\mathsf U_{-t}$, yields
\begin{align*}
    D_Q\sqrt{h_t}
    &=\mathsf O_t(D_Q\sqrt{\overline h_t})\circ\mathsf U_{-t}
    =\frac12\sqrt{h_t}\,G_t^B,\\
    \E_{\mu_t}\|G_t^B\|_{\mathbb H}^2
    &=\E_{\overline\mu_t}\|\overline G_t\|_{\mathbb H}^2
    =\sigma_{\rm na}(t).
\end{align*}
By the definition of the closed gradient form $\cE_Q$, this gives $4\cE_Q(\sqrt{h_t},\sqrt{h_t})=\sigma_{\rm na}(t)$.

The bounded-operator relation $(Q^{1/2}B)(QB)=\mathsf A_B(Q^{1/2}B)$ gives, by the exponential series,
\begin{equation*}
    Q^{1/2}B\mathsf U_t=\mathsf O_tQ^{1/2}B,
    \qquad J_B(\mathsf U_t\psi)=\mathsf O_tJ_B(\psi).
\end{equation*}
Therefore
\begin{equation*}
    \E_{\mu_t}\|J_B\|_{\mathbb H}^2
    =\Tr(B^*QB C_t^B)
    =\E_{\overline\mu_t}\|\mathsf O_tJ_B\|_{\mathbb H}^2
    =\Tr(B^*QB\,\overline C_t)<\infty.
\end{equation*}
The finiteness follows from boundedness of $Q^{1/2}B$ and the trace-class property of $\overline C_t$. This proves both current-energy formulas in \eqref{eq:sigma-na-hk}.

\par\noindent\textbf{Step 3. Orthogonality and stationary circulation.}\par\nopagebreak
Use a single index $\alpha$ for the Fourier eigenvectors, and write
\begin{equation*}
    Qe_\alpha=q_\alpha e_\alpha,
    \qquad \overline C_te_\alpha=\overline c_\alpha(t)e_\alpha,
    \qquad \overline G_t(\psi)=\sum_\alpha g_\alpha(t)\langle\psi,e_\alpha\rangle e_\alpha.
\end{equation*}
The coefficients in the last series are those in \eqref{eq:Onsager-fields}. For $\mathcal I_N=\{0,(k,c),(k,s):1\le k\le N\}$, diagonal Gaussian covariance gives
\begin{align*}
    \E_{\overline\mu_t}\langle J_B,P_N\overline G_t\rangle_{\mathbb H}
    &=-\sum_{\alpha\in\mathcal I_N}\sqrt{q_\alpha}\,g_\alpha(t)
    \E_{\overline\mu_t}\bigl[\langle B\psi,e_\alpha\rangle\langle\psi,e_\alpha\rangle\bigr]\\
    &=-\sum_{\alpha\in\mathcal I_N}\sqrt{q_\alpha}\,g_\alpha(t)
    \overline c_\alpha(t)\langle Be_\alpha,e_\alpha\rangle=0.
\end{align*}
Here every diagonal matrix element of the real skew-adjoint operator $B$ vanishes. Since $P_N\overline G_t\to\overline G_t$ in $L^2(\overline\mu_t;\mathbb H)$ and $J_B$ belongs to the same space, Cauchy--Schwarz allows $N\to\infty$. Using the two field transformations from Step~2 and the unitarity of $\mathsf O_t$ now gives
\begin{equation*}
    \E_{\mu_t}\langle J_B,G_t^B\rangle_{\mathbb H}
    =\E_{\overline\mu_t}\langle\mathsf O_tJ_B,\mathsf O_t\overline G_t\rangle_{\mathbb H}
    =0.
\end{equation*}
Expanding $\E_{\mu_t}\|J_B-G_t^B\|^2$ proves \eqref{eq:onsager-decomp}. At stationarity the density is $1$ and its gradient vanishes. Substituting $\overline C_t=Q$ in the current energy gives \eqref{eq:hk-ss}; positivity for $B\ne0$ follows from \cref{prop:La-representation}, and \cref{thm:Galerkin-path-KL} identifies this energy with the stationary path-space rate.

\par\noindent\textbf{Step 4. Galerkin convergence.}\par\nopagebreak
The compressed operators satisfy $B_N^*=-B_N$ and $Q_N>0$ on $\mathbb H_N$. Put
\begin{equation*}
    \mathsf U_t^N=e^{-tQ_NB_N},
    \qquad \overline C_t^N=P_N\overline C_tP_N,
    \qquad \overline\mu_t^N=\cN(0,\overline C_t^N).
\end{equation*}
The covariance calculation in \cref{lem:cov} applies in $\mathbb H_N$ and gives
\begin{equation*}
    \mu_t^N=(\mathsf U_t^N)_\#\overline\mu_t^N,
    \qquad h_t^N=\overline h_t^N\circ\mathsf U_{-t}^N,
    \qquad \overline h_t^N=\frac{d\overline\mu_t^N}{d\mu_Q^N}.
\end{equation*}
Finite-dimensional differentiation gives
\begin{equation*}
    Q_N^{1/2}\nabla\log\overline h_t^N(\psi_N)
    =\sum_{k=0}^N\frac{\sqrt{\lambda_k}T_k}{1+T_k}P_{\mathbb V_k}\psi_N.
\end{equation*}
Transport by $\mathsf U_t^N$ and the finite-dimensional orthogonality calculation of Step~3 therefore give
\begin{equation*}
    \int_{\mathbb H_N}
    \bigl\|Q_N^{-1/2}(v_{\rm ss}^N-Q_N\nabla\log h_t^N)\bigr\|_{\mathbb H_N}^2\,d\mu_t^N
    =\sum_{k=0}^Nd_k\frac{T_k^2}{1+T_k}
    +\Tr(B_N^*Q_NB_N\overline C_t^N).
\end{equation*}
Writing $B_{\alpha\beta}=\langle Be_\beta,e_\alpha\rangle$, the second term is
\begin{equation*}
    \Tr(B_N^*Q_NB_N\overline C_t^N)
    =\sum_{\alpha,\beta\in\mathcal I_N}
    q_\alpha\overline c_\beta(t)|B_{\alpha\beta}|^2
    \uparrow\Tr(B^*QB\,\overline C_t).
\end{equation*}
Both sums increase over nonnegative terms. Their limits are $\sigma_{\rm na}(t)$ and $\sigma_{\rm hk}(t)$, respectively, proving \eqref{eq:sigma-tot-limit} and its monotonicity.
\end{proof}

\begin{remark}[Relation to path-space entropy production]
The rates $\sigma_{\rm na}(t)$, $\sigma_{\rm hk}(t)$ and $\sigma_{\rm tot}(t)$ in \cref{thm:lin-hk} are instantaneous Onsager current quadratic forms. At stationarity, \cref{thm:Galerkin-path-KL} identifies $\sigma_{\rm hk}^{\rm ss}=\sigma_{\rm tot}^{\rm ss}$ with the forward--reverse path-space relative-entropy rate.
\end{remark}

\subsubsection{Explicit Consequences}

\begin{corollary}[Large-time asymptotics]\label{cor:asymptotic}
Under the hypotheses of \cref{thm:lin-hk}, set
\begin{equation*}
    S=\sum_{k\ge0}d_k(\lambda_k/\lambda_k^0-1)^2<\infty.
\end{equation*}
Then
\begin{equation*}
    e^{4t}\sigma_{\rm na}(t)\longrightarrow S,
    \qquad e^{4t}\RelEnt{\mu_t}{\mu_Q}\longrightarrow\frac S4,
\end{equation*}
while $\sigma_{\rm hk}(t)\to\sigma_{\rm hk}^{\rm ss}$ and $\sigma_{\rm tot}(t)\to\sigma_{\rm hk}^{\rm ss}$. Moreover,
\begin{equation}\label{eq:housekeeping-quench-relaxation}
    \sigma_{\rm hk}(t)-\sigma_{\rm hk}^{\rm ss}
    =e^{-2t}(m-m_0)\Tr(B^*QBQ A_0^{-1}).
\end{equation}
For $B\ne0$, the housekeeping term decreases when $m>m_0$ and increases when $m<m_0$. It is constant when $m=m_0$ or $B=0$.
\end{corollary}

\begin{proof}\leavevmode
Put $a_k=\lambda_k/\lambda_k^0-1$, so $T_k=e^{-2t}a_k$. The uniform positivity of $1+e^{-2t}a_k$ and summability of $\sum_kd_ka_k^2$ allow dominated convergence in
\begin{equation*}
    e^{4t}\sigma_{\rm na}(t)
    =\sum_{k\ge0}d_k\frac{a_k^2}{1+e^{-2t}a_k},
\end{equation*}
giving the limit $S$. For sufficiently large $t$, the Taylor remainder in $T_k-\log(1+T_k)=T_k^2/2+O(|T_k|^3)$ is uniform in $k$. Since $\sum_kd_k|a_k|^3<\infty$,
\begin{equation*}
    e^{4t}\sum_{k\ge0}d_kO(|T_k|^3)
    =O(e^{-2t})\sum_{k\ge0}d_k|a_k|^3\longrightarrow0,
\end{equation*}
which proves the entropy asymptotic. The resolvent identity
\begin{equation*}
    A_0^{-1}-Q=(m-m_0)Q A_0^{-1}
\end{equation*}
and \eqref{eq:sigma-na-hk} give \eqref{eq:housekeeping-quench-relaxation}. Its trace coefficient is nonnegative and is positive for $B\ne0$: in the common eigenbasis of $Q$ and $A_0^{-1}$, it is a sum of strictly positive weights times $\|Q^{1/2}Be_\alpha\|^2$. This proves the monotonicity statements and, with \eqref{eq:onsager-decomp}, the remaining limits.
\end{proof}

\begin{remark}[Gaussian LSI and exact entropy decay]
\label{rmk:LSI-transfer}
In the linear case, set
\begin{equation*}
    U_N^{\rm lin}(x)=\frac12\langle x,A_Nx\rangle_{\mathbb H_N},
    \qquad \mu_Q^N=\cN(0,Q_N),
\end{equation*}
and let $\cE_N$ be the $Q_N$-gradient form of $\mu_Q^N$. The preconditioned Hessian of $U_N^{\rm lin}$ satisfies
\begin{equation*}
    A_N^{-1/2}\nabla^2U_N^{\rm lin}A_N^{-1/2}=I_N,
\end{equation*}
and therefore has curvature lower bound $\rho_N=1$. Hence $\Ent_{\mu_Q^N}(g^2)\le2\cE_N(g,g)$. The corresponding infinite-dimensional Gaussian LSI, together with \cref{prop:circulation-semigroup}, gives the entropy bound
\begin{equation*}
    \Ent_{\mu_Q}(h_t^B)\le e^{-2t}\Ent_{\mu_Q}(h_0)
\end{equation*}
for every finite-entropy initial density $h_0$, independently of $B$, with $h_t^B$ as in \eqref{eq:density-circulation}. For a non-trivial mass quench $m_0\ne m$, \cref{cor:asymptotic} instead gives the exact asymptotic $\RelEnt{\mu_t}{\mu_Q}\sim(S/4)e^{-4t}$. The covariance ratios of $\overline C_t$ relative to $Q$ differ from one by $T_k$, which decay as $e^{-2t}$. Gaussian relative entropy is quadratic in these deviations near equilibrium and therefore decays as $e^{-4t}$.
\end{remark}

Commutation makes the circulation act within equal-variance Fourier eigenspaces. The general integral decomposition then has the following modewise form.

\begin{corollary}[Commuting spectral formulas]\label{cor:commuting-Galerkin-rate}\label{cor:commuting-transient-formulas}
If, in addition to the hypotheses of \cref{thm:lin-hk},
\begin{equation}\label{eq:B-commuting}
    BQ=QB,
\end{equation}
then
\begin{equation}\label{eq:B-blocks}
    Be_0=0,\qquad B|_{\mathbb V_k}=b_kR,\qquad
    R=\begin{pmatrix}0&1\\-1&0\end{pmatrix},\qquad
    \sup_{k\ge1}|b_k|\le\|B\|.
\end{equation}
In this case $C_t^B=\overline C_t$ and $G_t^B=\overline G_t$. On each nonzero Fourier eigenspace,
\begin{equation*}
    P_{\mathbb V_k}J_B(\phi)
    =-b_k\lambda_k^{-1/2}RP_{\mathbb V_k}\phi,
\end{equation*}
so the two fields are pointwise orthogonal, $\mu_t$-almost everywhere. Moreover,
\begin{align*}
    \sigma_{\rm hk}(t)
    &=2\sum_{k\ge1}b_k^2\frac{1+T_k}{\lambda_k^2},
    &\sigma_{\rm hk}^{\rm ss}
    &=2\sum_{k\ge1}\frac{b_k^2}{(k^2+m)^2},\\
    \sigma_{\rm hk}(t)-\sigma_{\rm hk}^{\rm ss}
    &=2e^{-2t}(m-m_0)\sum_{k\ge1}
    \frac{b_k^2}{(k^2+m)^2(k^2+m_0)}.
\end{align*}
For uniform block coefficients $b_k=c$, the stationary value is $2c^2\sum_{k\ge1}(k^2+m)^{-2}$. Let $\Pbb_{B_N,N,T}$ denote the path law of the Galerkin equation started from $\mu_Q^N$ on $[0,T]$, as in \cref{cor:galerkin-path-rate}. Its stationary path-space rate satisfies
\begin{equation*}
    \frac1T\RelEnt{\Pbb_{B_N,N,T}}{\Pbb_{B_N,N,T}\circ r_T^{-1}}
    =2\sum_{k=1}^N\frac{b_k^2}{(k^2+m)^2},
\end{equation*}
and
\begin{equation*}
    0\le\Tr(B^*QBQ)
    -\frac1T\RelEnt{\Pbb_{B_N,N,T}}{\Pbb_{B_N,N,T}\circ r_T^{-1}}
    \le2\|B\|^2\sum_{k>N}(k^2+m)^{-2}=O(N^{-3}).
\end{equation*}
\end{corollary}

\begin{proof}\leavevmode
The commutation relation makes each eigenspace of $Q$ invariant under $B$. Its real skew-adjoint restriction vanishes on $\mathbb V_0$ and has the displayed $2\times2$ form on $\mathbb V_k$. Both $A_0^{-1}$ and $\overline C_t$ are scalar on these eigenspaces. The operators $\mathsf U_t=\mathsf O_t$ are therefore unitary on $\mathbb H$ and commute with $\overline C_t$ and the coefficients of $\overline G_t$. This proves $C_t^B=\overline C_t$ and $G_t^B=\overline G_t$. The vector $P_{\mathbb V_k}\overline G_t$ is parallel to $P_{\mathbb V_k}\phi$, whereas its circulation partner is a multiple of $RP_{\mathbb V_k}\phi$. Orthogonality follows first for finite sums and then almost everywhere from convergence of the two fields. Taking the trace block by block proves the three housekeeping formulas. Finally, \cref{cor:galerkin-path-rate} gives the finite-mode path-space rate, and the bound $|b_k|\le\|B\|$ gives the displayed tail estimate.
\end{proof}

The coefficient of $R$ in the actual drift $-QB$ on $\mathbb V_k$ is $-b_k/(k^2+m)$; relative to the counterclockwise orientation generated by $-R$, the angular velocity is $b_k/(k^2+m)$. In this commuting case, the deterministic amplitude decays as $e^{-t}$ while the field rotates within each eigenspace.

\subsection{Stationary Path-Space Entropy Production}
\label{sec:stationary-path-production}

At stationarity, the Onsager decomposition leaves the circulation energy $\Tr(B^*QBQ)$, which we now identify with the forward--reverse relative-entropy rate. The stationary adjoint identifies the reversed dynamics as \eqref{eq:linear} with $B$ replaced by $-B$. Gaussian exponential integrability of the current coordinate $J_B(\phi)=-Q^{1/2}B\phi$ under $\mu_Q$, together with stationarity, gives the short-time Novikov bounds used in \cref{lem:stationary-gaussian-girsanov}. Successive Girsanov changes of measure then give equivalence of the forward and reversed laws on $C([0,T];\mathbb H)$, in the setting of infinite-dimensional diffusion time reversal \cite{FollmerWakolbinger1986}.

\begin{theorem}[Forward--reverse equivalence and exact KL rate]
\label{thm:Galerkin-path-KL}
Let $\Pbb_{B,T}$ denote the law on $C([0,T];\mathbb H)$ of~\eqref{eq:linear} started from $\mu_Q$, and let $r_T:C([0,T];\mathbb H)\to C([0,T];\mathbb H)$ be the time-reversal map defined in \cref{sec:detailed-balance-path-reversal}. Write $\E_{\pm B}$ for expectation under $\Pbb_{\pm B,T}$. Then, for every $T>0$,
\begin{equation}\label{eq:path-reversal-B}
    \Pbb_{B,T}\circ r_T^{-1}=\Pbb_{-B,T},
    \qquad \Pbb_{B,T}\sim\Pbb_{-B,T}.
\end{equation}
Under $\Pbb_{-B,T}$, write the driving noise as $dW_t^Q=Q^{1/2}dW_t^{-B}$, where $W^{-B}$ is cylindrical Wiener noise. Then
\begin{equation}\label{eq:path-RN-B}
    \frac{d\Pbb_{B,T}}{d\Pbb_{-B,T}}
    = \exp\!\left\{ -\sqrt{2}\int_0^T\!\langle Q^{1/2}B\phi_t,dW_t^{-B}\rangle_{\mathbb H}
    -\int_0^T\!\|Q^{1/2}B\phi_t\|_{\mathbb H}^2\,dt \right\}.
\end{equation}
Consequently,
\begin{equation}\label{eq:galerkin-KL}
    \frac1T\RelEnt{\Pbb_{B,T}}{\Pbb_{B,T}\circ r_T^{-1}}
    = \Tr(B^*QBQ)
    = \|Q^{1/2}BQ^{1/2}\|_{\mathrm{HS}}^2
    = \E_{\mu_Q}\|Q^{-1/2}J_B^{\mathbb K}\|_{\mathbb H}^2.
\end{equation}
\end{theorem}

\begin{proof}
By \cref{prop:skew-invariance}, the stationary solution has continuous $\mathbb H$-valued paths. For bounded Borel functions $F,G$, \cref{prop:circulation-semigroup} gives
\begin{align*}
    \int_{\mathbb H}F P_t^BG\,d\mu_Q
    &=\int_{\mathbb H}F\,[(P_t^0G)\circ\mathsf U_t]\,d\mu_Q
    =\int_{\mathbb H}(F\circ\mathsf U_{-t})P_t^0G\,d\mu_Q\\
    &=\int_{\mathbb H}G P_t^0(F\circ\mathsf U_{-t})\,d\mu_Q
    =\int_{\mathbb H}G P_t^{-B}F\,d\mu_Q.
\end{align*}
The second equality uses the invertibility of $\mathsf U_t$ and its preservation of $\mu_Q$, the third uses self-adjointness of $P_t^0$, and the last applies the semigroup formula with $-B$. By $L^2(\mu_Q)$-contractivity and density, the identity extends to all $F,G\in L^2(\mu_Q)$, so $(P_t^B)^*=P_t^{-B}$. Applying this adjoint relation successively, as in \eqref{eq:forward-fdd}--\eqref{eq:reversed-fdd}, identifies the stationary finite-dimensional distributions under time reversal. Continuity of the paths then gives the first identity in \eqref{eq:path-reversal-B}.

Under \eqref{eq:B-assumptions}, $Q$ is positive, injective and trace class, while $B$ is bounded and skew-adjoint; moreover, \cref{prop:skew-invariance} gives the common stationary marginal $\mu_Q$. Thus all hypotheses of \cref{lem:stationary-gaussian-girsanov} are met. By \cref{lem:stationary-gaussian-girsanov}, the two path laws are equivalent and their Radon--Nikodym derivative is exactly~\eqref{eq:path-RN-B}.

Set $u(\phi)=-\sqrt2\,Q^{1/2}B\phi$. Under $\Pbb_{B,T}$, the cylindrical process $W^B=W^{-B}-\int_0^\cdot u(\phi_s)\,ds$ is Wiener. Rewriting the logarithm in \eqref{eq:path-RN-B} under this measure gives
\begin{equation*}
    \log\frac{d\Pbb_{B,T}}{d\Pbb_{-B,T}}
    = \int_0^T\langle u(\phi_t),dW_t^B\rangle_{\mathbb H} + \frac12\int_0^T\|u(\phi_t)\|_{\mathbb H}^2\,dt.
\end{equation*}
By stationarity and \eqref{eq:Girsanov-energy}, $\E_B\int_0^T\|u(\phi_t)\|_{\mathbb H}^2dt = 2T\Tr(B^*QBQ)<\infty$. Thus the stochastic integral is square-integrable and has mean zero. Taking expectations yields
\begin{equation*}
    \RelEnt{\Pbb_{B,T}}{\Pbb_{-B,T}}
    = \frac12\E_B\int_0^T\|u(\phi_t)\|_{\mathbb H}^2dt
    = \int_0^T\E_B\|Q^{1/2}B\phi_t\|_{\mathbb H}^2dt = T\Tr(B^*QBQ),
\end{equation*}
where the last equality uses stationarity and~\eqref{eq:Girsanov-energy}. Combining this with \eqref{eq:path-reversal-B} and~\eqref{eq:J-norm} proves \eqref{eq:galerkin-KL}.

\end{proof}

\begin{corollary}[Finite-dimensional approximation of the stationary path-space rate]
\label{cor:galerkin-path-rate}
Let
\begin{equation*}
    Q_N=P_NQP_N,\qquad B_N=P_NBP_N,
\end{equation*}
and let $\Pbb_{B_N,N,T}$ be the stationary path law on $C([0,T];\mathbb H_N)$ of
\begin{equation*}
    d\phi_t^N=-(I_N+Q_NB_N)\phi_t^N\,dt + \sqrt2\,dW_t^{Q_N}.
\end{equation*}
Then
\begin{equation}\label{eq:galerkin-limit}
    \frac1T\RelEnt{\Pbb_{B_N,N,T}}{\Pbb_{B_N,N,T}\circ r_T^{-1}}
    = \|P_NQ^{1/2}BQ^{1/2}P_N\|_{\mathrm{HS}}^2 \uparrow \Tr(B^*QBQ).
\end{equation}
\end{corollary}

\begin{proof}
Since $P_N$ is a spectral projection of $Q$,
\begin{align*}
    Q_N^{1/2}B_NQ_N^{1/2}
    &= P_NQ^{1/2}P_NBP_NQ^{1/2}P_N\\
    &= P_NQ^{1/2}BQ^{1/2}P_N.
\end{align*}
Moreover, $B_N^*=-B_N$. The finite-dimensional instance of \cref{thm:Galerkin-path-KL} therefore gives
\begin{equation*}
    \frac1T\RelEnt{\Pbb_{B_N,N,T}}{\Pbb_{B_N,N,T}\circ r_T^{-1}} = \|P_NQ^{1/2}BQ^{1/2}P_N\|_{\mathrm{HS}}^2.
\end{equation*}
Put $\mathsf A_B=Q^{1/2}BQ^{1/2}$ and let $\mathcal I_N=\{0,(k,c),(k,s):1\le k\le N\}$ index the Fourier basis of $\mathbb H_N$. Then
\begin{align*}
    \|P_N\mathsf A_BP_N\|_{\mathrm{HS}}^2
    &= \sum_{\alpha,\beta\in\mathcal I_N} |\langle e_\alpha,\mathsf A_Be_\beta\rangle_{\mathbb H}|^2\\
    &\uparrow \sum_{\alpha,\beta} |\langle e_\alpha,\mathsf A_Be_\beta\rangle_{\mathbb H}|^2
    = \|\mathsf A_B\|_{\mathrm{HS}}^2
    = \Tr(B^*QBQ)\qquad\text{as }N\to\infty.
\end{align*}
The monotonicity follows because the index sets $\mathcal I_N\times\mathcal I_N$ increase. The approximating law is defined by the Galerkin-truncated equation and need not equal the law of the projected process $P_N\phi_t$.
\end{proof}

\begin{remark}[Time-reversal convention]
\label{rmk:time-reversal}
For the coordinate-wise reversal used in \cref{thm:Galerkin-path-KL}, field values are left unchanged, corresponding to $\Theta=I$. The theorem identifies the stationary path-space relative-entropy rate with the resulting forward--reverse asymmetry. Hence $B\ne0$ breaks this time-reversal symmetry even though the invariant one-time law remains $\mu_Q$.

If an orthogonal involution $\Theta$ satisfies $\Theta Q=Q\Theta$ and $\Theta B\Theta=-B$, then the model instead obeys generalised detailed balance relative to $\Theta$: its stationary path law is invariant under the transformed reversal $\omega(t)\mapsto\Theta\omega(T-t)$. For the two-mode force in \eqref{eq:two-mode-force}, the reflection $\Theta\phi=\phi-2\langle\phi,e_k\rangle_{\mathbb H}e_k$ has these properties. The rate in \eqref{eq:galerkin-KL} is computed with the coordinate-wise reversal $\Theta=I$.
\end{remark}

\section{Examples}
\label{sec:examples}

The quartic Gibbs field gives a nonlinear application of the symmetric entropy results. A two-mode Gaussian circulation then illustrates the distinction between relaxation of the one-time law and irreversibility of stationary trajectories.

\subsection{The \texorpdfstring{$\Phi^4_1$}{Phi4(1)} Field}

The one-dimensional $\Phi^4$ theory is the Euclidean scalar field with quartic self-interaction:
\begin{equation*}
    V(\phi) = V_4(\phi):=\frac{\lambda}{4}\int_{\T}\phi^4\,dx,
    \qquad \lambda>0,\quad f=0.
\end{equation*}
The Gaussian reference $\mu_Q$, with $Q=(-\partial_x^2+m)^{-1}$, is the free massive Euclidean field; $m$ is the quadratic mass parameter and fixes a free correlation length of order $m^{-1/2}$. The Gibbs weight $e^{-V_4}$ introduces the quartic self-interaction. Here the unsubscripted $\lambda$ is its coupling constant and is unrelated to the spectral eigenvalues $\lambda_k=k^2+m$ fixed after \eqref{eq:basis}. Since $m,\lambda>0$ and $V=V_4$, \cref{prop:quartic-model-verification}(i) gives $Z_4=\int_{\mathbb X}e^{-V_4}\,d\mu_Q\in(0,\infty)$, and we set $\kappa_4=Z_4^{-1}e^{-V_4}\mu_Q$. This is the standard one-dimensional quartic Euclidean Gibbs measure of constructive quantum field theory \cite{GlimmJaffe1987}.

In stochastic quantisation \cite{ParisiWu1981,DamgaardHuffel1987}, the Langevin time is an auxiliary sampling time, and the entropy estimate below quantifies convergence of that sampling dynamics to the interacting field law. The model theorem below combines this relaxation estimate with verification that the SPDE has the closed gradient form and diffusion calculus required by the abstract results.

\begin{theorem}[The one-dimensional $\Phi^4_1$ Langevin dynamics]
\label{thm:Phi4-model}
Let
\begin{equation*}
    \mathbb X=C(\T),\qquad
    A=-\partial_x^2+m,\qquad Q=A^{-1},\qquad
    \mu_Q=\cN(0,Q),
\end{equation*}
and let
\begin{equation*}
    V_4(\phi)=\frac\lambda4\int_\T\phi^4\,dx,
    \qquad
    \kappa_4=Z_4^{-1}e^{-V_4}\mu_Q,
    \qquad m,\lambda>0.
\end{equation*}
Let $(\cE,\operatorname{Dom}(\cE))$ be the closure of \eqref{eq:pre-form} on Fourier cylinders, with semigroup $(P_t)$ and generator $L$ on $L^2(\kappa_4)$. Then the following statements hold.
\begin{enumerate}[label=(\roman*)]
    \item For every $x\in\mathbb X$, the equation
    \begin{equation}\label{eq:Phi4-spde}
        d\phi_t=-[\phi_t+\lambda Q(\phi_t^3)]\,dt+\sqrt2\,dW_t^Q,
        \qquad \phi_0=x,
    \end{equation}
    has a unique global probabilistically strong $\mathbb X$-valued solution with continuous paths, satisfying the stochastic integral equation in $\mathbb X$. The closure $\cE$ satisfies \cref{ass:equilibrium,ass:diffusion-calculus}, and its $L^2(\kappa_4)$ semigroup is the transition semigroup of \eqref{eq:Phi4-spde}.
    \item The measure $\kappa_4$ is invariant, the transition semigroup is reversible with respect to $\kappa_4$, and the stationary process is invariant under path reversal. If $h_0$ is a probability density with respect to $\kappa_4$, $h_t=P_th_0$ and $\Ent_{\kappa_4}(h_0)<\infty$, then
    \begin{equation*}
        \Ent_{\kappa_4}(h_t)
        +4\int_0^t\cE(\sqrt{h_s},\sqrt{h_s})\,ds
        \le \Ent_{\kappa_4}(h_0).
    \end{equation*}
    If in addition $h_0\in\operatorname{Dom}(L)$ and $0<c_-\le h_0\le c_+<\infty$, the inequality is an equality.
    \item For every $g\in\operatorname{Dom}(\cE)$ with $g\ge0$ and $\int g^2\,d\kappa_4=1$,
    \begin{equation}\label{eq:LSI-Phi4}
        \Ent_{\kappa_4}(g^2)\le2\cE(g,g).
    \end{equation}
    Consequently,
    \begin{equation*}
        \Ent_{\kappa_4}(P_th)
        \le e^{-2t}\Ent_{\kappa_4}(h)
    \end{equation*}
    for every probability density $h$ with finite entropy.
\end{enumerate}
\end{theorem}

\begin{proof}
\par\noindent\textbf{Step 1. Verification of the form hypotheses.}\par
Here $\mathbb X=C(\T)$, $Q=(-\partial_x^2+m)^{-1}$, $m,\lambda>0$ and $V=V_4$. Thus \cref{prop:quartic-model-verification}(i)--(iii) gives the Gibbs probability measure $\kappa_4$, the unique global continuous strong solution of \eqref{eq:Phi4-spde}, and closability of its Fourier-gradient pre-form. The closure satisfies \cref{ass:equilibrium,ass:diffusion-calculus} and has the same semigroup as the SPDE, proving (i). The solutions form a conservative continuous Markov family on $\mathbb X$, so the hypotheses of \cref{cor:langevin-entropy-equilibrium} hold. Its parts~(i) and~(iii) give (ii), with equality under the additional domain and positive-bound conditions on the initial density.

\par\noindent\textbf{Step 2. LSI for $\kappa_{4,N}$.}\par
Define the Galerkin-approximated measure
\begin{equation*}
    \kappa_{4,N}
    = Z_{4,N}^{-1}\,e^{-V_4(P_N\phi)}\,\mu_Q,
\end{equation*}
where
\begin{equation*}
    Z_{4,N}=\int_{\mathbb X}e^{-V_4(P_N\phi)}\,d\mu_Q.
\end{equation*}
For $g\in\cF C_b^\infty$, define the gradient energy
\begin{equation*}
    \cE_{4,N}(g,g)=\int\|Q^{1/2}\nabla_{\mathbb H}g\|_{\mathbb H}^2\,d\kappa_{4,N}.
\end{equation*}
Thus $\kappa_{4,N}$ is the product of a low-mode $\Phi^4_1$ interaction with a free Gaussian on the high modes.

Fix $N$, and let $g$ be a smooth cylinder depending on the Fourier coordinates in $\mathbb H_M$ for some $M\ge N$. The $\mathbb H_M$-marginal of $\kappa_{4,N}$ has the interacting potential on $\mathbb H_N$ and the free Gaussian potential on $\mathbb H_M\ominus\mathbb H_N$. On the interacting block $\mathbb H_N$, set
\begin{equation*}
    U_{4,N}(\phi)
    =\frac12\langle\phi,A_N\phi\rangle_{\mathbb H_N}+V_4(P_N\phi).
\end{equation*}
Its Hessian is
\begin{equation*}
    \nabla^2 U_{4,N}=A_N
    +3\lambda P_N\operatorname{Mult}_{(P_N\phi)^2}P_N.
\end{equation*}
Here $\operatorname{Mult}_a u=au$. For every $u\in\mathbb H_N$,
\begin{equation*}
    \left\langle u,
    P_N\operatorname{Mult}_{(P_N\phi)^2}P_Nu
    \right\rangle_{\mathbb H_N}
    =\int_{\T}(P_N\phi)^2u^2\,dx\ge0.
\end{equation*}
The projections are needed because multiplication by $(P_N\phi)^2$ need not preserve $\mathbb H_N$. The preconditioned Hessian satisfies
\begin{equation*}
    Q_N^{1/2}\nabla^2 U_{4,N}\,Q_N^{1/2}
    =Q_N^{1/2}\bigl(A_N
    +3\lambda P_N\operatorname{Mult}_{(P_N\phi)^2}P_N\bigr)Q_N^{1/2}
    \ge I_N.
\end{equation*}
Indeed, the quadratic form of the additional term is
\begin{equation*}
    3\lambda\int_{\T}(P_N\phi)^2
    |Q_N^{1/2}u|^2\,dx\ge0.
\end{equation*}
The free Gaussian block $\mathbb H_M\ominus\mathbb H_N$ has preconditioned Hessian equal to the identity. Hence the $\mathbb H_M$-marginal has a smooth, strictly positive density and its preconditioned Hessian is bounded below by $I_M$. The finite-dimensional Bakry--\'Emery criterion \cite{BakryGentilLedoux2014} states that a lower bound $\rho I_M$ on this Hessian yields an LSI with constant $2/\rho$. All its hypotheses are therefore satisfied here with $\rho=1$, and it gives
\begin{equation*}
    \Ent_{\kappa_{4,N}}(g^2)
    \le2\cE_{4,N}(g,g).
\end{equation*}
The constant is independent of both $M$ and $N$. Since every $g\in\cF C_b^\infty$ depends on Fourier coordinates in some $\mathbb H_M$, the displayed inequality holds for every such $g$.

\par\noindent\textbf{Step 3. The limit $N\to\infty$.}\par
Fix $s\in(1/4,1/2)$. Since $m>0$, \cref{lem:Gaussian-regularity} gives $\phi\in H^s(\T)$ for $\mu_Q$-almost every $\phi$; hence the Fourier projections satisfy $P_N\phi\to\phi$ in $H^s(\T)$ and, by $H^s(\T)\hookrightarrow L^4(\T)$, also in $L^4(\T)$. Hence $V_4(P_N\phi) \to V_4(\phi)$ $\mu_Q$-almost surely. Since $0 \le e^{-V_4(P_N\phi)} \le 1$, dominated convergence gives
\begin{equation*}
    Z_{4,N}\longrightarrow Z_4,
    \qquad
    \|\kappa_{4,N}-\kappa_4\|_{\rm TV}\longrightarrow0.
\end{equation*}
Fix $g\in\cF C_b^\infty$. Both $g$ and $\|Q^{1/2}\nabla_{\mathbb H}g\|_{\mathbb H}$ are bounded, and the continuous function $r\mapsto r^2\log r^2$, with value $0$ at $r=0$, is bounded on the range of $g$. Total-variation convergence therefore gives
\begin{align*}
    \int g^2\,d\kappa_{4,N}&\longrightarrow\int g^2\,d\kappa_4,\\
    \int g^2\log g^2\,d\kappa_{4,N}
    &\longrightarrow\int g^2\log g^2\,d\kappa_4,\\
    \cE_{4,N}(g,g)&\longrightarrow\cE(g,g).
\end{align*}
Using the definition \eqref{eq:homogeneous-entropy}, the first two limits imply
\begin{equation*}
    \Ent_{\kappa_{4,N}}(g^2)
    \longrightarrow\Ent_{\kappa_4}(g^2).
\end{equation*}
Passing to the limit in the cylinder inequality from Step~2 proves the LSI for every smooth cylinder $g$.

\par\noindent\textbf{Step 4. Extension and entropy decay.}\par\nopagebreak
For any $u \in \operatorname{Dom}(\cE)$, choose $u_n\in\cF C_b^\infty$ such that
\begin{equation*}
    \|u_n-u\|_{L^2(\kappa_4)}^2
    +\cE(u_n-u,u_n-u)\longrightarrow0.
\end{equation*}
Then $u_n^2\to u^2$ in $L^1(\kappa_4)$, because
\begin{equation*}
    \|u_n^2-u^2\|_{L^1}
    \le\|u_n-u\|_{L^2}\bigl(\|u_n\|_{L^2}+\|u\|_{L^2}\bigr)
    \longrightarrow0.
\end{equation*}
The variational representation of homogeneous entropy makes $u\mapsto\Ent_{\kappa_4}(u^2)$ lower semicontinuous under this convergence. The form Cauchy--Schwarz inequality gives
\begin{equation*}
    |\cE(u_n,u_n)-\cE(u,u)| \le \cE(u_n-u,u_n-u) + 2\cE(u,u)^{1/2}\cE(u_n-u,u_n-u)^{1/2} \longrightarrow 0.
\end{equation*}
Taking the lower limit in the cylinder LSI proves \eqref{eq:LSI-Phi4} on the whole form domain. The form and SPDE semigroup satisfy the hypotheses of \cref{cor:langevin-entropy-equilibrium}, as verified at the beginning of the proof. With $\rho=1$ in its log-Sobolev condition, part~(ii) therefore gives $\Ent_{\kappa_4}(h_t)\le e^{-2t}\Ent_{\kappa_4}(h_0)$ for every finite-entropy initial density.
\end{proof}

The entropy estimate in \cref{thm:Phi4-model} describes relaxation from any initial law of finite relative entropy with respect to $\kappa_4$. Its guaranteed exponential rate is uniform in $\lambda$: the $Q$-preconditioning normalises the free quadratic curvature, while $\nabla_{\mathbb H}^2V_4\succeq0$ prevents the interaction from lowering the Bakry--\'Emery bound. Free Gaussian initial laws with mass $m_0>0$ and quartic Gibbs initial laws with coupling $\lambda_{\mathrm{init}}\ge0$ are included by \cref{prop:quartic-model-verification}(iv).

For the numerical illustration, we take $m=\lambda=1$ and $h_0(\phi)=1+0.2\tanh z_0(\phi)$, where $z_0(\phi)=q_0^{-1/2}\langle\phi,e_0\rangle_{\mathbb H}$. The symmetry of $\kappa_4$ and the bounds $0.8\le h_0\le1.2$ make $h_0$ a normalised finite-entropy density, so \cref{thm:Phi4-model}(ii)--(iii) applies to the initial density $h_0$, which favours positive spatial averages. Let $H_N(t)$ be the relative entropy of the entire Fourier-truncated field law with respect to its Gibbs equilibrium. Figure~\ref{fig:phi4-entropy-relaxation} shows the estimated ratio $H_N(t)/H_N(0)$ decaying well below $e^{-2t}$ as the truncated law relaxes towards that equilibrium. The correlation representation in Appendix~\ref{app:numerics-phi4} controls the entropy approximation separately from sampling and time-discretisation effects.

\begin{figure}[!htbp]
    \centering
    \includegraphics[width=0.68\textwidth]{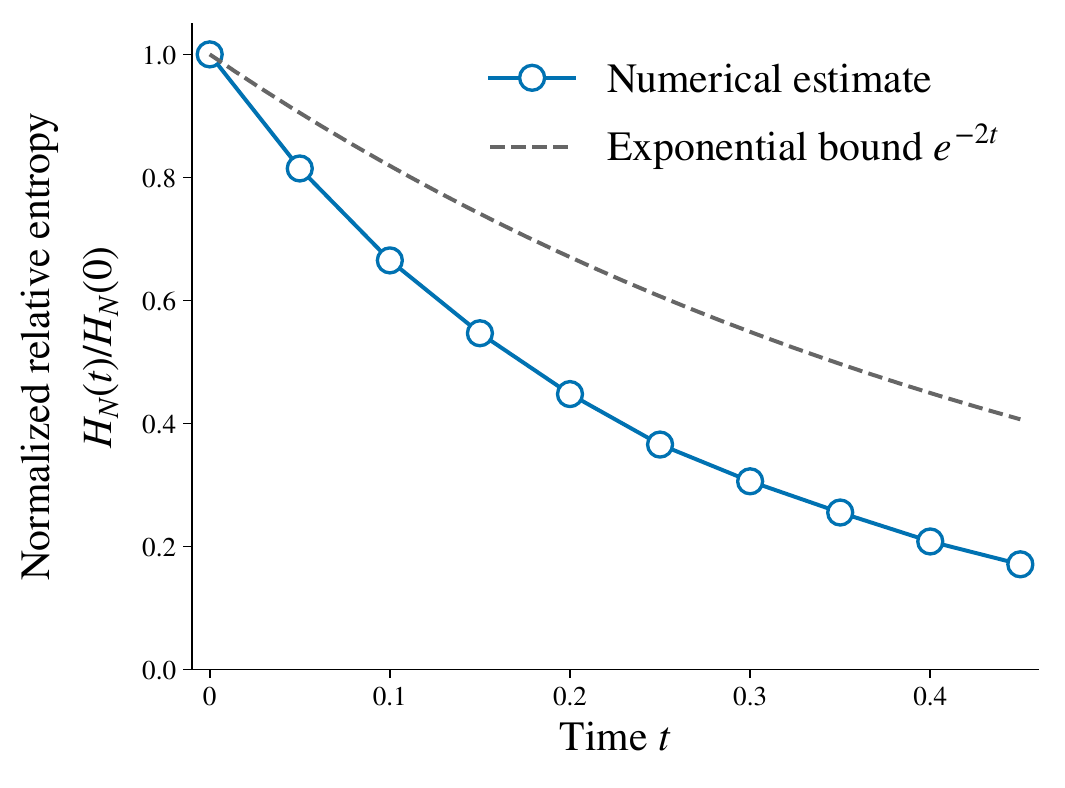}
    \caption{Relative-entropy relaxation for the quartic field with $m=\lambda=1$ and Fourier cutoff $N=32$ ($65$ real modes). Circles estimate $H_N(t)/H_N(0)$ from $40{,}000$ stationary trajectories with time step $0.0025$; the connecting line guides the eye. The correlation-based approximation has a deterministic multiplicative error factor of $1.00678$, separate from sampling and time-discretisation errors. The dashed line is the continuous-time bound $e^{-2t}$. See Appendix~\ref{app:numerics-phi4}.}
    \label{fig:phi4-entropy-relaxation}
\end{figure}
\FloatBarrier

\begin{remark}[Integrated entropy relaxation]
\label{rmk:phi4-inverse-energy}
Let $f\ne0$ be a real-valued element of $L^\infty(\kappa_4)$ with $\int f\,d\kappa_4=0$. For $\varepsilon\in\R$ with $r=|\varepsilon|\|f\|_\infty<1$, set $h_0^\varepsilon=1+\varepsilon f$ and $H_\varepsilon(t)=\Ent_{\kappa_4}(P_th_0^\varepsilon)$. Write $\langle\cdot,\cdot\rangle_2$ and $\|\cdot\|_2$ for the inner product and norm in $L^2(\kappa_4)$. If $L$ denotes the self-adjoint generator, the Poincar\'e consequence $\operatorname{Var}_{\kappa_4}(u)\le\cE(u,u)$ of \eqref{eq:LSI-Phi4} implies that $(-L)^{-1}$ is bounded on the mean-zero subspace. Its variational formula defines the inverse energy
\begin{equation}\label{eq:phi4-inverse-energy}
    \mathcal R(f)
    :=\sup_{u\in\operatorname{Dom}(\cE)}
    \bigl\{2\langle f,u\rangle_2-\cE(u,u)\bigr\}
    =\langle f,(-L)^{-1}f\rangle_2.
\end{equation}
Then
\begin{equation}\label{eq:phi4-integrated-entropy}
    \frac{\varepsilon^2}{4(1+r)}\mathcal R(f)
    \le\int_0^\infty H_\varepsilon(t)\,dt
    \le\frac{\varepsilon^2}{4(1-r)}\mathcal R(f).
\end{equation}
To verify \eqref{eq:phi4-integrated-entropy}, set $\Phi(x)=(1+x)\log(1+x)-x$. Since $\Phi(0)=\Phi'(0)=0$ and $\Phi''(x)=(1+x)^{-1}$, Taylor's formula gives $x^2/[2(1+r)]\le\Phi(x)\le x^2/[2(1-r)]$ for $|x|\le r$ \cite{SasonVerdu2016}. The Markov property and invariance give $|\varepsilon P_tf|\le r$ and $\int P_tf\,d\kappa_4=0$, hence $H_\varepsilon(t)=\int\Phi(\varepsilon P_tf)\,d\kappa_4$. Symmetry and the spectral theorem yield
\begin{equation*}
    2\int_0^\infty\|P_tf\|_2^2\,dt
    =\int_0^\infty\langle f,P_sf\rangle_2\,ds
    =\mathcal R(f),
\end{equation*}
which proves \eqref{eq:phi4-integrated-entropy}. Since $\langle f,P_sf\rangle_2=\E_{\kappa_4}[f(\phi_0)f(\phi_s)]$ under the stationary law, $\mathcal R(f)$ is the integrated equilibrium covariance \cite{KipnisVaradhan1986}. The spectral gap gives $\mathcal R(f)\le\|f\|_2^2$. For fixed $f$, Taylor's formula gives $H_\varepsilon(0)=\varepsilon^2\|f\|_2^2/2+o(\varepsilon^2)$ as $\varepsilon\to0$, and \eqref{eq:phi4-integrated-entropy} therefore implies
\begin{equation}\label{eq:phi4-normalized-entropy-integral}
    \lim_{\varepsilon\to0}\int_0^\infty
    \frac{H_\varepsilon(t)}{H_\varepsilon(0)}\,dt
    =\frac{\mathcal R(f)}{2\|f\|_2^2}
    =\frac12\int_0^\infty
    \frac{\E_{\kappa_4}[f(\phi_0)f(\phi_s)]}{\|f\|_2^2}\,ds
    \le\frac12.
\end{equation}
Thus, in the small-perturbation limit, the integrated normalised relative entropy equals one half of the integrated normalised equilibrium autocorrelation of $f$.
\end{remark}

\begin{remark}[A negative quadratic perturbation]
\label{rmk:quartic-perturbation}
The same verification applies to the Allen--Cahn-type interaction \cite{AllenCahn1979}
\begin{equation*}
    V_{\mathrm{AC}}(\phi)=\lambda\int_\T\left(\frac{\phi^4}{4}-\frac{\phi^2}{2}\right)dx,
    \qquad \kappa_{\mathrm{AC}}=Z_{\mathrm{AC}}^{-1}e^{-V_{\mathrm{AC}}}\mu_Q,
    \qquad m,\lambda>0.
\end{equation*}
For every $x\in\mathbb X=C(\T)$, \cref{prop:quartic-model-verification}(i)--(iii) gives a unique global continuous probabilistically strong solution of
\begin{equation}\label{eq:AC-spde}
    d\phi_t=-[\phi_t+\lambda Q(\phi_t^3-\phi_t)]\,dt+\sqrt2\,dW_t^Q,
    \qquad \phi_0=x,
\end{equation}
and identifies its transition semigroup with that of a closed gradient form satisfying \cref{ass:equilibrium,ass:diffusion-calculus}. In this remark, the form on $L^2(\kappa_{\mathrm{AC}})$ and its semigroup are again denoted by $\cE$ and $(P_t)$. Thus \cref{cor:langevin-entropy-equilibrium} gives the entropy-dissipation inequality, its equality for regular positive initial densities, and stationary path reversal with respect to $\kappa_{\mathrm{AC}}$.

If $m>\lambda$, put $\rho=1-\lambda/m>0$. On $\mathbb H_N$, the potential $U_{\mathrm{AC},N}(\phi)=\frac12\langle\phi,A_N\phi\rangle_{\mathbb H_N}+V_{\mathrm{AC}}(\phi)$ satisfies
\begin{equation*}
    Q_N^{1/2}\nabla^2U_{\mathrm{AC},N}Q_N^{1/2}
    =I_N+\lambda Q_N^{1/2}P_N\operatorname{Mult}_{3\phi^2-1}P_NQ_N^{1/2}
    \ge I_N-\lambda Q_N\ge\rho I_N.
\end{equation*}
The complementary free modes have preconditioned curvature $1$, so the finite-dimensional Bakry--\'Emery criterion gives an LSI with constant $2/\rho$. By \cref{lem:Gaussian-regularity}, $P_N\phi\to\phi$ in $L^4(\T)$ for $\mu_Q$-almost every $\phi$, while $e^{-V_{\mathrm{AC}}(P_N\phi)}\le e^{\lambda\pi/2}$. Dominated convergence therefore gives convergence of the normalised Gibbs measures in total variation. The cylinder-limit and form-core arguments in the proof of \cref{thm:Phi4-model} then yield
\begin{equation}\label{eq:LSI-AC-full}
    \Ent_{\kappa_{\mathrm{AC}}}(g^2)\le\frac2\rho\cE(g,g),
    \qquad g\in\operatorname{Dom}(\cE).
\end{equation}
Applying \cref{cor:langevin-entropy-equilibrium}(ii) gives
\begin{equation}\label{eq:LSI-AC-decay}
    \Ent_{\kappa_{\mathrm{AC}}}(P_th)\le e^{-2\rho t}\Ent_{\kappa_{\mathrm{AC}}}(h)
\end{equation}
for every probability density $h$ of finite entropy. The rate reflects the reduction of the free quadratic curvature by the negative quadratic interaction. Free Gaussian initial laws with any mass $m_0>0$ have finite relative entropy by \cref{prop:quartic-model-verification}(iv).
\end{remark}

\subsection{A Gaussian Field with Two-Mode Circulation}
\label{sec:two-mode-circulation}

On $\mathbb H=L^2(\T;\R)$, consider the Gaussian field equation
\begin{equation*}
    d\phi_t=-(\phi_t+QB\phi_t)\,dt+\sqrt2\,dW_t^Q,
    \qquad Q=(-\partial_x^2+m)^{-1},\quad m>0,
\end{equation*}
with $W^Q$ the $\mathbb H$-valued $Q$-Wiener process of \eqref{eq:linear}. Choose two orthonormal eigenvectors $e_j,e_k$ of $Q$, with $Qe_j=q_je_j$ and $Qe_k=q_ke_k$, and set
\begin{equation}\label{eq:two-mode-force}
    B\phi=c\bigl(\langle\phi,e_j\rangle_{\mathbb H}e_k-\langle\phi,e_k\rangle_{\mathbb H}e_j\bigr),
    \qquad c\in\R.
\end{equation}
The parameter $c$ controls the coupling of the two selected modes, while their orthogonal complement retains all the other field modes, evolving under free Ornstein--Uhlenbeck dynamics. The operator $B$ is bounded and skew-adjoint, so \cref{prop:skew-invariance} applies. For $c\ne0$, the relations $BQe_j=cq_je_k$ and $QBe_j=cq_ke_k$ show that $BQ=QB$ precisely when $q_j=q_k$. The equal-variance choice $e_j=e_1^s$, $e_k=e_1^c$ recovers the rotational case of \cref{cor:commuting-transient-formulas} \cite{LandiTomeOliveira2013,GodrecheLuck2019}.

For every $c$, the full field has the same invariant law $\mu_Q$, but \cref{thm:Galerkin-path-KL} gives the stationary path-space rate
\begin{equation*}
    \frac1T\RelEnt{\Pbb_{B,T}}{\Pbb_{B,T}\circ r_T^{-1}}
    =\|Q^{1/2}BQ^{1/2}\|_{\mathrm{HS}}^2
    =2c^2q_jq_k.
\end{equation*}
This rate is positive exactly when $c\ne0$, whereas $c=0$ gives the reversible Gaussian field. Under $\mu_Q$, the driven pair and its orthogonal complement are independent, and the latter has a reversible stationary path law. Thus the two-mode stationary rate equals the full-field rate and agrees with the Ornstein--Uhlenbeck current formula \cite{GodrecheLuck2019,DaCostaPavliotis2023}.

For the figures, we use the modal coordinates $X_t=(\langle\phi_t,e_j\rangle_{\mathbb H},\langle\phi_t,e_k\rangle_{\mathbb H})^T$. We take $e_j=e_0$, $e_k=e_1^c$ and $m=1$, so the stationary marginal has covariance $\mathsf Q^{(2)}=\operatorname{diag}(1,1/2)$ and density $\pi_2$. Figure~\ref{fig:gaussian-stationary-current} shows the nonzero current for $c=1$ running along the density contours, whereas setting $c=0$ removes the current without changing $\pi_2$.

\begin{figure}[!htbp]
    \centering
    \includegraphics[width=0.68\textwidth]{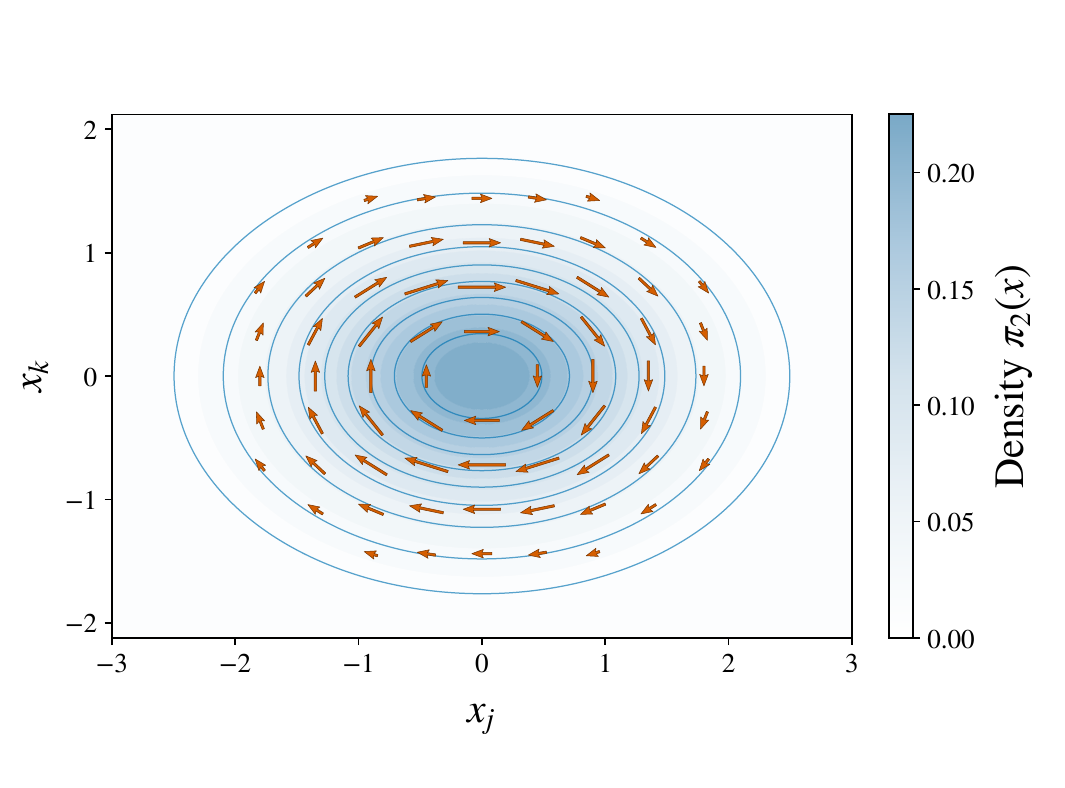}
    \caption{Stationary two-mode marginal density $\pi_2$ (shading and contours) and probability current $J_{\rm ss}(x)=-\mathsf Q^{(2)}B_1x\,\pi_2(x)$ (arrows), where $B_1=\left(\begin{smallmatrix}0&-1\\1&0\end{smallmatrix}\right)$. Both are evaluated analytically. Arrow lengths use one common scale and retain the relative current magnitudes. The corresponding full-field stationary entropy-production rate is $2q_jq_k=1$.}
    \label{fig:gaussian-stationary-current}
\end{figure}

To study relaxation, take the full-field initial law $\mu_0=\cN(0,(-\partial_x^2+m_0)^{-1})$, where $m_0>0$. This mass-quench law also factors over the driven pair and the remaining modes. Write $\overline c_\ell(t)=\langle\overline C_te_\ell,e_\ell\rangle_{\mathbb H}$ for the variances of the symmetric evolution from \cref{lem:cov}. When $c\ne0$, $q_j\ne q_k$ and $m_0\ne m$, the driven covariance $C_t^B=\mathsf U_t\overline C_t\mathsf U_t^*$ differs from $\overline C_t$ at small positive times, since $\mathsf U'_0=-QB$ and $\overline c_\ell(0)=q_\ell/[1+(m_0-m)q_\ell]$ give
\begin{equation*}
    \left. \frac{d}{dt}\langle ( C_t^B - \overline C_t )e_k,e_j \rangle_{\mathbb H} \right|_{t=0}
    = c\bigl[q_j\overline c_k(0)-q_k\overline c_j(0)\bigr]
    = \frac{c(m_0-m)q_jq_k(q_j-q_k)}{[1+(m_0-m)q_j][1+(m_0-m)q_k]}\ne0.
\end{equation*}
Nevertheless, \cref{prop:circulation-semigroup} gives the same relative entropy for the driven and symmetric field laws. Their complementary modes have identical free evolution, so additivity of relative entropy under this product decomposition gives the same conclusion for the two-mode marginals. For each value of $c$, write $\mathsf C_t^{(2)}=\operatorname{Cov}(X_t)$ for the marginal covariance. With $m_0=2$, the initial marginal covariance is $\mathsf C_0^{(2)}=\operatorname{diag}(1/2,1/3)$; Figure~\ref{fig:gaussian-entropy-relaxation} compares the resulting marginal entropy $H_2(t)$ for $c=0$ and $c=1$.

To distinguish the transient contributions, denote the Onsager rates of this marginal by $\sigma_{\rm na}^{(2)}$, $\sigma_{\rm hk}^{(2)}$ and $\sigma_{\rm tot}^{(2)}$. Since
\begin{equation*}
    B^*QB=c^2\bigl(q_k e_j\otimes e_j+q_j e_k\otimes e_k\bigr),
\end{equation*}
the trace formula and decomposition in \cref{thm:lin-hk}, applied to the mass-quench law, give
\begin{align*}
    \sigma_{\rm hk}^{(2)}(t)&=\sigma_{\rm hk}(t)
    =c^2\bigl[q_k\overline c_j(t)+q_j\overline c_k(t)\bigr],\\
    \sigma_{\rm tot}^{(2)}(t)&=\sigma_{\rm na}^{(2)}(t)+\sigma_{\rm hk}^{(2)}(t).
\end{align*}
The remaining modes contribute free relaxation to the full-field entropy and its dissipation, while the selected pair carries all the circulation energy. For $c=1$ and $m_0=2$, the initial modal variances are smaller than their stationary values, so $\sigma_{\rm hk}^{(2)}(t)$ increases towards $1$ as $\sigma_{\rm na}^{(2)}(t)$ decreases to zero. Figure~\ref{fig:gaussian-entropy-production} shows how these opposing trends give a nonmonotone total rate, which first decreases and then approaches $1$ from below.

\begin{figure}[!htbp]
    \centering
    \begin{subfigure}[t]{0.49\textwidth}
        \centering
        \includegraphics[width=\linewidth]{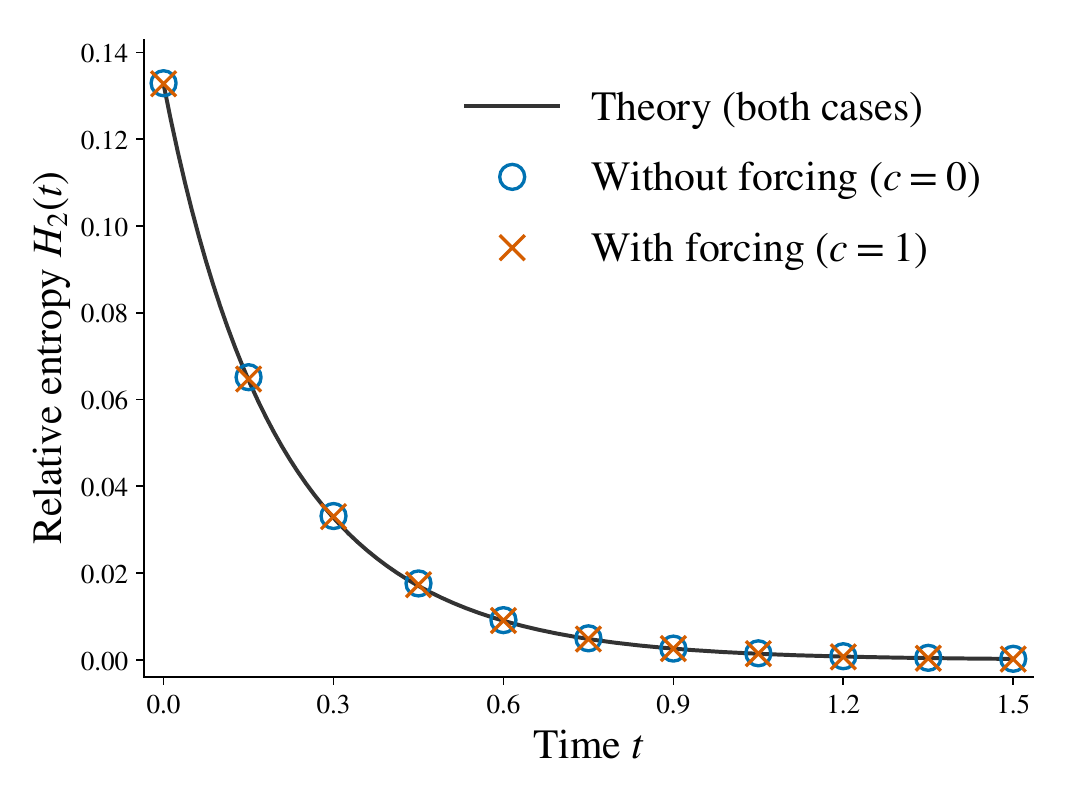}
        \caption{Relative-entropy relaxation.}
        \label{fig:gaussian-entropy-relaxation}
    \end{subfigure}\hfill
    \begin{subfigure}[t]{0.49\textwidth}
        \centering
        \includegraphics[width=\linewidth]{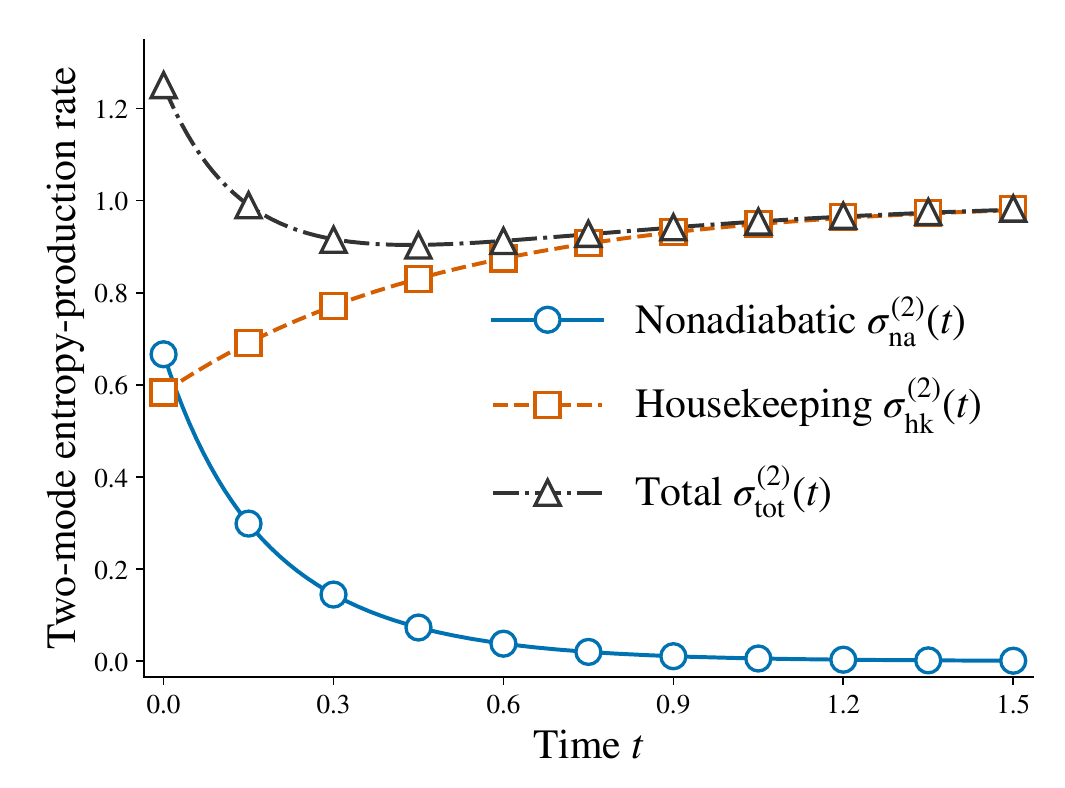}
        \caption{Entropy-production decomposition.}
        \label{fig:gaussian-entropy-production}
    \end{subfigure}
    \caption{Relaxation and entropy production of the two-mode marginal after the full-field mass quench from $m_0=2$ to $m=1$. (\subref{fig:gaussian-entropy-relaxation}) Relative entropy $H_2(t)=\RelEnt{\cN(0,\mathsf C_t^{(2)})}{\cN(0,\mathsf Q^{(2)})}$. Circles ($c=0$) and crosses ($c=1$) use empirical second moments from $400{,}000$ trajectories per case, sampled by the exact Ornstein--Uhlenbeck transition of the projected dynamics. The solid curve is the common analytic entropy. (\subref{fig:gaussian-entropy-production}) Entropy-production decomposition with $c=1$. The solid, dashed and dash-dotted curves are $\sigma_{\rm na}^{(2)}$, $\sigma_{\rm hk}^{(2)}$ and $\sigma_{\rm tot}^{(2)}$, respectively. Markers estimate the three quadratic forms using the same $c=1$ trajectories as in panel~(\subref{fig:gaussian-entropy-relaxation}), with the total evaluated directly rather than obtained by adding the other estimates.}
    \label{fig:gaussian-transient}
\end{figure}
\FloatBarrier

For this mass-quench initial law, the field has the same Gaussian stationary law and the same relative-entropy relaxation curve with and without the force, although its transient distributions differ. The distinction remains at stationarity, where a nonzero force sustains a current and a positive full-field forward--reverse path entropy rate after the relaxation contribution has vanished.

\section{Discussion and Outlook}
\label{sec:outlook}

Relative entropy $\Ent_\kappa(h)$ measures departure from the invariant law, whereas the Dirichlet form $\cE$ describes the energy governing its dissipation. For regular densities, the instantaneous loss is $4\cE(\sqrt{h_t},\sqrt{h_t})$; the closed-form argument also gives an integrated dissipation inequality for every finite-entropy initial law. For the gradient dynamics considered here, this energy depends on the mobility and noise covariance as well as on $\kappa$. Identifying the SPDE semigroup with the closed-form semigroup makes these relations applicable to the field evolution. The results on relaxation, circulation and stationary time reversal are summarised below.

\begin{center}
\small
\renewcommand{\arraystretch}{1.15}
\begin{tabular}{@{}L{0.28\textwidth}L{0.33\textwidth}L{0.31\textwidth}@{}}
    \hline
    \textbf{Question} & \textbf{Answer in this paper} &
    \textbf{Mathematical formulation}\\
    \hline
    How is relative-entropy relaxation formulated? &
    By an entropy equality for regular densities and an inequality for finite-entropy data. &
    Symmetric Dirichlet form and carr\'e du champ.\\
    How does circulation affect relaxation? &
    Driven and symmetric Gaussian evolutions from the same initial density have equal relative entropy with respect to $\mu_Q$ at each time. &
    Symmetric Gaussian semigroup and a measure-preserving flow.\\
    When is stationary time reversal preserved or broken? &
    Symmetric equilibrium dynamics is reversible, whereas Gaussian circulation with $B\ne0$ breaks this symmetry. &
    Stationary forward--reverse path relative entropy: zero at equilibrium and positive for $B\ne0$.\\
    How are the SPDE applications justified? &
    The SPDE transition semigroup is identified with the closed-form semigroup. &
    Verification of the form hypotheses and equality of the full form domains.\\
    \hline
\end{tabular}
\end{center}

For initial densities $h_0^\varepsilon=1+\varepsilon f$ relative to $\kappa_4$, with $f$ bounded, nonzero and of mean zero, \cref{rmk:phi4-inverse-energy} identifies the small-perturbation limit of $\int_0^\infty H_\varepsilon(t)/H_\varepsilon(0)\,dt$ as $\mathcal R(f)/(2\|f\|_{L^2(\kappa_4)}^2)$. This is one half of the integrated normalised equilibrium autocorrelation of $f$, and the variational formula for $\mathcal R(f)$ gives lower bounds in terms of the Dirichlet form $\cE$. The relation supplements the uniform exponential bound by quantifying how the integrated relaxation depends on the initial perturbation $f$.

The $\Phi^4_1$ application combines established well-posedness and Gibbs symmetry with verification of the closed-gradient form and a log-Sobolev bound. For higher-dimensional fields, symmetric Dirichlet-form constructions for the two-dimensional renormalised polynomial Gibbs measures $P(\Phi)_2$ are classical \cite{BorkarChariMitter1988,AlbeverioRockner1991}, while R.~Zhu and X.~Zhu constructed a quasi-regular form associated with the dynamical $\Phi^4_3$ model and identified its energy on cylinder functions as the classical gradient form \cite{ZhuZhu2018Dirichlet}. Quantitative results include a spectral gap for the two-dimensional stochastic quantisation equation \cite{TsatsoulisWeber2016} and log-Sobolev inequalities for the continuum $\varphi^4_2$ and $\varphi^4_3$ measures, with estimates for the standard gradient energy uniform in the lattice regularisation under a bounded-susceptibility condition \cite{BauerschmidtDagallier2024}.

For a fixed interacting field measure, the choice of mobility and noise covariance determines the dissipation energy of the dynamics. The corresponding question for renormalised fields is how these choices affect entropy relaxation, particularly for preconditioned dynamics of the type considered here. Estimates for the associated energy forms that remain stable as the regularisation is removed would yield entropy-dissipation bounds for the limiting dynamics. Such bounds would relate equilibrium field fluctuations to the time needed for the dynamics to lose memory of their initial distribution.

A parallel extension concerns nonlinear driving, for which the stationary law need not remain Gaussian. The central question is how to separate dissipation and circulation relative to the actual stationary law and relate the circulating component to time-reversal breaking. Active Model B is a representative of the broader class of stochastic field theories for active matter \cite{NardiniEtAl2017}. It describes a conserved scalar order parameter $u$ for phase separation and has the formal dimensionless form
\begin{equation*}
    \partial_t u
    =\Delta\bigl(u^3-u-\Delta u+\alpha|\nabla u|^2\bigr)
    +\sqrt{2\varepsilon}\,\nabla\cdot\xi
\end{equation*}
on a periodic domain with fixed spatial mean. Here $\xi$ is vector-valued space--time white noise, $\varepsilon>0$ is the noise strength, and $\alpha$ measures activity. The passive part pairs the mobility $-\Delta$ with noise covariance $2\varepsilon(-\Delta)$. Entropy production and the relation between stationary probability currents and propagating modes near interfaces have been studied in this model \cite{NardiniEtAl2017,OByrne2023}.

In such models, separating the passive and active terms of the equation need not give the decomposition into dissipation and circulation relative to the driven stationary law. O'Byrne and Cates discuss a functional Hodge construction whose mathematical justification in their setting remains to be established \cite[Section~3.2.2]{OByrneCates2025}. For the stationary-law decomposition considered here, the task is to describe circulation by an antisymmetric operator in the $L^2$ space of that law, together with a symmetric closed form compatible with the noise. This would provide a basis for asking when the noise-weighted current energy gives a finite relative entropy between forward and reversed path laws. Spatial regularisation is part of this question, since discretisation affects both detailed balance and entropy production \cite{CatesEtAl2022}. If a continuum limit with space--time white noise is sought, the derivative nonlinearity $\Delta(\alpha|\nabla u|^2)$ also requires a singular-SPDE interpretation, and the path-law comparison must be controlled as the cutoff is removed. Where a finite-energy drift representation is unavailable, one can instead ask how an operator or bilinear-form description of circulation relates to irreversibility.

Both extensions lead to a common question of statistical physics: whether entropy dissipation and stationary irreversibility pass from microscopic dynamics to a field limit. A microscopic derivation also fixes the physical interpretation of time, temperature and mobility in the resulting SPDE. For equilibrium systems, the near-critical two-dimensional Ising--Kac model gives a precise example of this passage, as its rescaled Glauber dynamics converges to the dynamical $\Phi^4_2$ equation \cite{MourratWeber2017} and its equilibrium measures converge to the $\Phi^4_2$ measure \cite{HairerIberti2018}. A thermodynamic extension would follow relative entropy and its dissipation through this limit. For driven microscopic systems, the corresponding question concerns stationary currents and forward--reverse path relative entropy.

Under a fixed coarse-graining map compatible with time reversal, the data-processing inequality shows that forward--reverse path relative entropy cannot increase. The loss of path information is related to entropy production carried by hidden degrees of freedom \cite{Esposito2012,KawaguchiNakayama2013}. To follow the thermodynamic quantities through
\begin{equation*}
    \text{microscopic Markov dynamics}
    \longrightarrow \text{coarse-grained field}
    \longrightarrow \text{Langevin SPDE}
\end{equation*}
one would therefore need convergence estimates for entropy and dissipation, or control of forward--reverse path relative entropy, in addition to convergence of the processes themselves. Such results would connect hydrodynamic and fluctuation limits with field-level stochastic thermodynamics and macroscopic fluctuation theory \cite{KipnisLandim1999,BertiniEtAl2015}.

This comparison across scales also depends on the initial law and boundary mechanism. The Gaussian and Gibbs initial laws used here, including those associated with mass and interaction quenches, suggest asking which microscopic preparation protocols yield field laws of finite relative entropy, while reservoirs would add boundary entropy fluxes to bulk dissipation and stationary circulation. The resulting comparison would clarify which thermodynamic relations persist in the field description.

\section*{Acknowledgements}

The authors thank Q.~Zhang for discussions on the finite-dimensional entropy-production decomposition, and A.~Zhang and J.~Zimmer for helpful discussions on stationary currents, reversibility and path-space entropy production for stochastic evolution equations. The work of Shuyuan Fan, Yuanke Chen and Jinqiao Duan was partly supported by the Guangdong Provincial Key Laboratory of Mathematical and Neural Dynamical Systems (No.\ 2024B1212010004), the Dongguan Key Laboratory for Data Science and Intelligent Medicine, the Innovation Team Project for Higher Education Institutions in Guangdong Province, ``Cross-disciplinary Research Team on Data Science and Intelligent Medicine'' (No.\ 2023KCXTD054), and the NSFC International Collaboration Fund for Creative Research Teams (Grant No.\ W2541005). Shuyuan Fan was additionally supported by the Natural Science Foundation of Guangdong Province (No.\ 2025A1515011188). Lifei Wang acknowledges financial support from the Hebei International Joint Research Center for Mathematics and Interdisciplinary Science and the Central Guidance on Local Science and Technology Development Fund of Hebei Province (No.~264Z0106G).

\appendix
\section{Properties of Symmetric Closed Dirichlet Forms}
\label{app:dirichlet-details}

This appendix proves the finite-entropy extension and the properties of closed gradients used in the entropy results. It also records a Gibbs closability criterion and the projection limit for relative entropy.

\begin{lemma}[Approximation from regular densities to finite-entropy data]
\label{lem:finite-entropy-extension}
Assume \cref{ass:equilibrium}. Suppose that every probability density $g\in\operatorname{Dom}(L)$ satisfying $0<c_-\le g\le c_+<\infty$ obeys
\begin{equation}\label{eq:regular-dissipation-hypothesis}
    \Ent_\kappa(P_tg) + 4\int_0^t\cE(\sqrt{P_sg},\sqrt{P_sg})\,ds = \Ent_\kappa(g),
    \qquad t\ge0.
\end{equation}
Then every probability density $h_0$ with $\Ent_\kappa(h_0)<\infty$ satisfies, with $h_s=P_sh_0$,
\begin{equation*}
    \Ent_\kappa(h_t) + 4\int_0^t\cE(\sqrt{h_s},\sqrt{h_s})\,ds \le \Ent_\kappa(h_0),
    \qquad t\ge0.
\end{equation*}
\end{lemma}

\begin{proof}
\par\noindent\textbf{Step 1. Truncation and entropy convergence.}\par
We use a truncation followed by resolvent regularisation. Set
\begin{equation*}
    \tilde h_0^{(n)} = \min(h_0, n) + n^{-1},\qquad
    Z_n = \int\tilde h_0^{(n)}\,d\kappa\;(\ge n^{-1}),\qquad
    \bar h_0^{(n)} = \tilde h_0^{(n)}/Z_n.
\end{equation*}
Then $(2n)^{-1} \le \bar h_0^{(n)} \le 3n$ for sufficiently large $n$, and $\bar h_0^{(n)} \to h_0$ in $L^1(\kappa)$. The truncations also satisfy
\begin{equation}\label{eq:entropy-truncation-convergence}
    \Ent_\kappa(\bar h_0^{(n)})\longrightarrow\Ent_\kappa(h_0).
\end{equation}
We give the truncation calculation explicitly. Write $\Phi(r)=r\log r$ with $\Phi(0)=0$. Since
\begin{equation*}
    0\le h_0 - \min(h_0,n) = (h_0-n)\mathbf1_{\{h_0>n\}}\le h_0,
\end{equation*}
dominated convergence gives
\begin{equation*}
    \|\tilde h_0^{(n)}-h_0\|_{L^1(\kappa)}\longrightarrow0,
    \qquad Z_n\longrightarrow1,
    \qquad \|\bar h_0^{(n)}-h_0\|_{L^1(\kappa)}\longrightarrow0.
\end{equation*}
The negative part $\Phi^-$ is bounded by $e^{-1}$. On $\{h_0\le n\}$, $\tilde h_0^{(n)}=h_0+n^{-1}$, and
\begin{equation*}
    \Phi^+(h_0+n^{-1}) \le C\bigl(1+h_0\log^+h_0\bigr),
\end{equation*}
with a constant independent of $n$. Hence dominated convergence applies there. On the complementary set, for all sufficiently large $n$,
\begin{equation*}
    \int_{\{h_0>n\}}\Phi^+(n+n^{-1})\,d\kappa
    \le 2n\log(2n)\,\kappa(h_0>n)
    \le 4\int_{\{h_0>n\}}h_0\log h_0\,d\kappa
    \longrightarrow 0.
\end{equation*}
It follows that
\begin{equation*}
    \int\tilde h_0^{(n)}\log\tilde h_0^{(n)}\,d\kappa
    \longrightarrow
    \int h_0\log h_0\,d\kappa.
\end{equation*}
Finally, the homogeneous entropy of an integrable $h\ge0$ with mass $Z$ satisfies
\begin{align*}
    \Ent_\kappa(h/Z)
    &= \int\frac hZ\log\!\left(\frac hZ\right)d\kappa
    = \frac1Z\int h\log h\,d\kappa-\log Z\\
    &= \frac1Z\left(\int h\log h\,d\kappa-Z\log Z\right)
    = \frac1Z\Ent_\kappa(h).
\end{align*}
Combining this formula with $Z_n\to1$ and the preceding convergence proves \eqref{eq:entropy-truncation-convergence}.

\par\noindent\textbf{Step 2. Resolvent regularisation.}\par
Let $h_0^{(n)}$ be the resolvent regularisation of the bounded density $\bar h_0^{(n)}$, defined by
\begin{equation}\label{eq:resolvent-regularisation}
    \mathsf R_n^{\rm res}=n(n-L)^{-1} = \int_0^\infty ne^{-nt}P_t\,dt,
    \qquad
    h_0^{(n)} = \mathsf R_n^{\rm res}\bar h_0^{(n)}.
\end{equation}
Since $P_t$ is positivity preserving and $P_t1=1$, the integral representation gives
\begin{equation*}
    \mathsf R_n^{\rm res}1=1,
    \qquad
    f\ge0\ \Longrightarrow\ \mathsf R_n^{\rm res}f\ge0.
\end{equation*}
Consequently, if $a\le f\le b$, then
\begin{equation*}
    a = \mathsf R_n^{\rm res}a
    \le \mathsf R_n^{\rm res}f
    \le \mathsf R_n^{\rm res}b=b.
\end{equation*}
Invariance of $\kappa$ under $P_t$ and Fubini's theorem also give
\begin{equation*}
    \int_{\mathbb X}\mathsf R_n^{\rm res}f\,d\kappa
    = \int_0^\infty ne^{-nt} \int_{\mathbb X}P_tf\,d\kappa\,dt
    = \int_{\mathbb X}f\,d\kappa.
\end{equation*}
Thus $h_0^{(n)}$ is a probability density and retains the positive upper and lower bounds of $\bar h_0^{(n)}$. Moreover,
\begin{equation*}
    (n-L)\mathsf R_n^{\rm res}f=nf, \qquad f\in L^2(\kappa),
\end{equation*}
so $\mathsf R_n^{\rm res}f\in\operatorname{Dom}(L)$.

The Markov semigroup is contractive on $L^1(\kappa)$. If $f\in L^1(\kappa)\cap L^2(\kappa)$, then, since $\kappa$ is a probability measure,
\begin{equation*}
    \|P_tf-f\|_{L^1(\kappa)}
    \le \|P_tf-f\|_{L^2(\kappa)}
    \longrightarrow 0
    \qquad\text{as }t \downarrow 0.
\end{equation*}
For general $f\in L^1(\kappa)$, choose $f_m\in L^1(\kappa)\cap L^2(\kappa)$ with $f_m\to f$ in $L^1(\kappa)$. Contractivity gives
\begin{equation*}
    \|P_tf-f\|_1 \le 2\|f-f_m\|_1 + \|P_tf_m-f_m\|_1.
\end{equation*}
First letting $t\downarrow0$ and then $m\to\infty$ proves strong continuity on $L^1(\kappa)$. Hence, for $f\in L^1(\kappa)$,
\begin{equation*}
    \|\mathsf R_n^{\rm res}f-f\|_1
    \le \int_0^\infty ne^{-nt}\|P_tf-f\|_1\,dt
    \longrightarrow 0.
\end{equation*}
Indeed, after splitting the integral at $\delta>0$, the first part is bounded by $\sup_{0\le t\le\delta}\|P_tf-f\|_1$ and the second by $2e^{-n\delta}\|f\|_1$. It follows that
\begin{align*}
    \|h_0^{(n)}-h_0\|_1
    &\le \|\mathsf R_n^{\rm res}(\bar h_0^{(n)}-h_0)\|_1 + \|\mathsf R_n^{\rm res}h_0-h_0\|_1\\
    &\le \|\bar h_0^{(n)}-h_0\|_1 + \|\mathsf R_n^{\rm res}h_0-h_0\|_1
    \longrightarrow 0.
\end{align*}

\par\noindent\textbf{Step 3. Convergence of the regularised entropies.}\par
Let $\Phi(r)=r\log r$. Jensen's inequality for the positive unital operator $\mathsf R_n^{\rm res}$ gives
\begin{equation*}
    \Phi(\mathsf R_n^{\rm res}\bar h_0^{(n)})
    \le\mathsf R_n^{\rm res}\Phi(\bar h_0^{(n)}).
\end{equation*}
Integrating this inequality and using $\int\mathsf R_n^{\rm res}f\,d\kappa=\int f\,d\kappa$, which follows from the invariance of $\kappa$ under $P_t$, yields
\begin{align*}
    \Ent_\kappa(h_0^{(n)})
    &= \int_{\mathbb X} \Phi(\mathsf R_n^{\rm res}\bar h_0^{(n)})\,d\kappa\\
    &\le \int_{\mathbb X} \mathsf R_n^{\rm res}\Phi(\bar h_0^{(n)})\,d\kappa 
    = \int_{\mathbb X}\Phi(\bar h_0^{(n)})\,d\kappa
    = \Ent_\kappa(\bar h_0^{(n)}).
\end{align*}
For probability densities, relative entropy has the variational representation
\begin{equation*}
    \Ent_\kappa(h)
    = \sup_{g\in L^\infty(\kappa)} \left\{\int_{\mathbb X}hg\,d\kappa - \log\int_{\mathbb X}e^g\,d\kappa\right\}.
\end{equation*}
It is therefore lower semicontinuous for $L^1(\kappa)$ convergence. Since $h_0^{(n)}\to h_0$ in $L^1(\kappa)$,
\begin{equation*}
    \Ent_\kappa(h_0) \le \liminf_{n\to\infty}\Ent_\kappa(h_0^{(n)}).
\end{equation*}
Together with \eqref{eq:entropy-truncation-convergence}, these bounds prove
\begin{equation}\label{eq:regularised-entropy-convergence}
    \Ent_\kappa(h_0^{(n)})\longrightarrow\Ent_\kappa(h_0).
\end{equation}
\par\noindent\textbf{Step 4. Passage to the dissipation inequality.}\par\nopagebreak
Set $h_s^{(n)}=P_sh_0^{(n)}$ and $h_s=P_sh_0$. The $L^1$-contractivity of the semigroup gives
\begin{equation*}
    \|h_s^{(n)}-h_s\|_{L^1(\kappa)}
    \le\|h_0^{(n)}-h_0\|_{L^1(\kappa)}\longrightarrow0.
\end{equation*}
The $L^1(\kappa)$ lower semicontinuity established above from the variational representation gives
\begin{equation*}
    \Ent_\kappa(h_t)\le\liminf_{n\to\infty}\Ent_\kappa(h_t^{(n)}).
\end{equation*}
Closedness in \cref{ass:equilibrium}(i) makes $\cE$ lower semicontinuous on $L^2(\kappa)$. Moreover, $\|\sqrt{h_s^{(n)}}-\sqrt{h_s}\|_{L^2}^2 \le\|h_s^{(n)}-h_s\|_{L^1}\to0$ for each $s\in[0,t]$. Fatou's lemma therefore gives
\begin{equation*}
    \int_0^t \cE(\sqrt{h_s},\sqrt{h_s})\,ds
    \le \liminf_{n\to\infty}
    \int_0^t \cE(\sqrt{h_s^{(n)}},\sqrt{h_s^{(n)}})\,ds.
\end{equation*}
Each $h_0^{(n)}$ belongs to $\operatorname{Dom}(L)$ and is bounded above and away from zero. Hence~\eqref{eq:regular-dissipation-hypothesis} gives
\begin{equation}\label{eq:ii-trunc}
    \Ent_\kappa(h_t^{(n)}) + 4\int_0^t\cE(\sqrt{h_s^{(n)}},\sqrt{h_s^{(n)}})\,ds = \Ent_\kappa(h_0^{(n)}).
\end{equation}
Using~\eqref{eq:regularised-entropy-convergence} and the two lower semicontinuity bounds above, we obtain
\begin{align*}
    \Ent_\kappa(h_0)
    &= \lim_{n\to\infty}\left[ \Ent_\kappa(h_t^{(n)}) + 4\int_0^t\cE(\sqrt{h_s^{(n)}},\sqrt{h_s^{(n)}})\,ds \right]\\
    &\ge\liminf_{n\to\infty}\Ent_\kappa(h_t^{(n)}) + 4\liminf_{n\to\infty} \int_0^t\cE(\sqrt{h_s^{(n)}},\sqrt{h_s^{(n)}})\,ds\\
    &\ge\Ent_\kappa(h_t) + 4\int_0^t\cE(\sqrt{h_s},\sqrt{h_s})\,ds.
\end{align*}
\end{proof}

For the closed Gibbs gradient form, it remains to verify the normal-contraction property and conservativity.

\begin{lemma}[Normal contractions for a closed gradient form]
\label{lem:closed-gradient-markov}
Let $D_{\mathbb K}$ be the closure in $L^2(\kappa)\to L^2(\kappa;\mathbb K)$ of the cylinder gradient $\nabla_{\mathbb K}$, and define
\begin{equation*}
    \cE(u,v)=\int_{\mathbb X} \langle D_{\mathbb K}u,D_{\mathbb K}v\rangle_{\mathbb K}\,d\kappa.
\end{equation*}
For every normal contraction $\theta:\mathbb R\to\mathbb R$ and every $u\in\operatorname{Dom}(D_{\mathbb K})$,
\begin{equation*}
    \theta\circ u\in\operatorname{Dom}(D_{\mathbb K}),
    \qquad
    \|D_{\mathbb K}(\theta\circ u)\|_{\mathbb K}
    \le\|D_{\mathbb K}u\|_{\mathbb K}
    \quad\kappa\text{-a.e.}
\end{equation*}
If the normal contraction $\theta$ belongs to $C^\infty(\mathbb R)$, then
\begin{equation*}
    D_{\mathbb K}(\theta\circ u)=\theta'(u)D_{\mathbb K}u
    \qquad\kappa\text{-a.e.}
\end{equation*}
Consequently, $\cE$ has the Markov property. If the cylinder core contains $1$, then $\cE$ is conservative.
\end{lemma}

\begin{proof}
\par\noindent\textbf{Step 1. Smooth contractions on the closed domain.}\par
Let $\eta\in C^\infty(\mathbb R)$ be a normal contraction. For a cylinder function $F$, the ordinary chain rule gives
\begin{equation*}
    \nabla_{\mathbb K}(\eta\circ F) = \eta'(F)\nabla_{\mathbb K}F,
\end{equation*}
and hence
\begin{equation*}
    \cE_0(\eta\circ F,\eta\circ F)
    = \int_{\mathbb X}|\eta'(F)|^2\|\nabla_{\mathbb K}F\|_{\mathbb K}^2\,d\kappa
    \le \cE_0(F,F).
\end{equation*}
We next pass this estimate to the closed domain. Let $u\in\operatorname{Dom}(D_{\mathbb K})$ and choose cylinder functions $F_j$ such that
\begin{equation*}
    F_j\longrightarrow u\quad\text{in }L^2(\kappa),
    \qquad
    \nabla_{\mathbb K}F_j\longrightarrow D_{\mathbb K}u
    \quad\text{in }L^2(\kappa;\mathbb K).
\end{equation*}
For the smooth normal contraction $\eta$ fixed above, its Lipschitz property gives $\eta(F_j)\to\eta(u)$ in $L^2(\kappa)$. After passing to a subsequence, both $F_j\to u$ and $\nabla_{\mathbb K}F_j\to D_{\mathbb K}u$ hold $\kappa$-a.e. Since $|\eta'|\le1$ and $\eta'$ is continuous,
\begin{equation*}
    \bigl\|\eta'(F_j)\nabla_{\mathbb K}F_j - \eta'(u)D_{\mathbb K}u\bigr\|_{L^2(\kappa;\mathbb K)}
    \le \bigl\|\nabla_{\mathbb K}F_j-D_{\mathbb K}u \bigr\|_{L^2(\kappa;\mathbb K)} + \bigl\|[\eta'(F_j)-\eta'(u)]D_{\mathbb K}u \bigr\|_{L^2(\kappa;\mathbb K)} \longrightarrow 0.
\end{equation*}
For the second term, pointwise convergence follows from continuity of $\eta'$, while its squared integrand is bounded by $4\|D_{\mathbb K}u\|_{\mathbb K}^2\in L^1(\kappa)$. Dominated convergence therefore makes this term tend to zero. Closedness of $D_{\mathbb K}$ therefore gives
\begin{equation*}
    \eta\circ u\in\operatorname{Dom}(D_{\mathbb K}),
    \qquad
    D_{\mathbb K}(\eta\circ u) = \eta'(u)D_{\mathbb K}u.
\end{equation*}

\par\noindent\textbf{Step 2. Arbitrary normal contractions.}\par
Now let $\theta$ be an arbitrary Lipschitz normal contraction. If $\rho_m$ is a standard mollifier, set
\begin{equation*}
    \theta_m(r)=(\rho_m*\theta)(r)-(\rho_m*\theta)(0).
\end{equation*}
Then $\theta_m\in C^\infty(\mathbb R)$, $\theta_m(0)=0$, $|\theta_m'|\le1$, and $\theta_m\to\theta$ uniformly on $\mathbb R$. The smooth chain rule just proved gives
\begin{equation*}
    \|D_{\mathbb K}(\theta_m\circ u)(x)\|_{\mathbb K}
    \le \|D_{\mathbb K}u(x)\|_{\mathbb K}
    \qquad\text{for }\kappa\text{-a.e. }x.
\end{equation*}
Hence $(D_{\mathbb K}(\theta_m\circ u))_m$ is bounded in $L^2(\kappa;\mathbb K)$. Since this is a Hilbert space, there are a subsequence, not relabelled, and $G\in L^2(\kappa;\mathbb K)$ such that
\begin{equation*}
    D_{\mathbb K}(\theta_m\circ u)\rightharpoonup G
    \quad\text{in }L^2(\kappa;\mathbb K).
\end{equation*}
The uniform convergence $\theta_m\to\theta$ and the fact that $\kappa$ is a probability measure give
\begin{equation*}
    \theta_m\circ u\longrightarrow\theta\circ u
    \quad\text{strongly, hence weakly, in }L^2(\kappa).
\end{equation*}
Therefore the pairs
\begin{equation*}
    \bigl(\theta_m\circ u,D_{\mathbb K}(\theta_m\circ u)\bigr)
    \rightharpoonup(\theta\circ u,G)
\end{equation*}
in the product Hilbert space $L^2(\kappa)\times L^2(\kappa;\mathbb K)$. The graph of the closed linear operator $D_{\mathbb K}$ is a norm-closed linear subspace of this product space. Every norm-closed linear subspace of a Hilbert space is weakly closed, so the limiting pair also belongs to the graph. Consequently,
\begin{equation*}
    \theta\circ u\in\operatorname{Dom}(D_{\mathbb K}),
    \qquad
    D_{\mathbb K}(\theta\circ u)=G.
\end{equation*}
Finally, the set
\begin{equation*}
    \left\{H\in L^2(\kappa;\mathbb K):
    \|H(x)\|_{\mathbb K}\le\|D_{\mathbb K}u(x)\|_{\mathbb K}
    \ \kappa\text{-a.e.}\right\}
\end{equation*}
is convex by the triangle inequality. It is also norm closed: if $H_n\to H$ in $L^2(\kappa;\mathbb K)$, a subsequence converges $\kappa$-a.e., and the pointwise norm bound passes to its limit. Every norm-closed convex subset of a Hilbert space is weakly closed. It contains every $D_{\mathbb K}(\theta_m\circ u)$ and therefore contains $G$. We have proved
\begin{equation*}
    \|D_{\mathbb K}(\theta\circ u)\|_{\mathbb K}
    \le\|D_{\mathbb K}u\|_{\mathbb K}
    \quad\kappa\text{-a.e.}
\end{equation*}
for every Lipschitz normal contraction $\theta$; compare \cite{AlbeverioMaRockner2015}. Taking $\theta(r)=(0\vee r)\wedge1$ proves the Markov property.

\par\noindent\textbf{Step 3. Conservativity.}\par
The constant function $1$ belongs to the cylinder core and $D_{\mathbb K}1=0$. Therefore
\begin{equation*}
    \cE(1,v)=\int_{\mathbb X}
    \langle D_{\mathbb K}1,D_{\mathbb K}v\rangle_{\mathbb K}\,d\kappa=0,
    \qquad v\in\operatorname{Dom}(\cE).
\end{equation*}
It follows from the generator--form correspondence that $1\in\operatorname{Dom}(L)$ and $L1=0$. Hence $P_t1=1$, and the closed form is conservative.

\end{proof}

The contraction properties in \cref{lem:closed-gradient-markov} allow us to verify the form hypotheses required by \cref{cor:langevin-entropy-equilibrium} for the closure of the gradient pre-form $\cE_0$.

\begin{proposition}[Properties of the closed gradient form]
\label{prop:closed-gradient-properties}
Let $\kappa$ be a probability measure on $\mathbb X=C(\T)$ and suppose that the Fourier gradient pre-form \eqref{eq:pre-form} is closable on $L^2(\kappa)$. Its closure $(\cE,\operatorname{Dom}(\cE))$ is a conservative, strongly local, symmetric Dirichlet form. In particular, it satisfies \cref{ass:equilibrium,ass:diffusion-calculus}.
\end{proposition}
\begin{proof}
\par\noindent\textbf{Step 1. Closed-gradient representation.}\par
If $F_n\to0$ in $L^2(\kappa)$ and $\nabla_{\mathbb K}F_n\to G$ in $L^2(\kappa;\mathbb K)$, then $(F_n)$ is $\cE_0$-Cauchy. Closability of $\cE_0$ gives
\begin{equation*}
    \|G\|_{L^2(\kappa;\mathbb K)}^2
    =\lim_n\cE_0(F_n,F_n)=0.
\end{equation*}
Thus the operator
\begin{equation*}
    \nabla_{\mathbb K}:\cF C_b^\infty \subset L^2(\kappa)\longrightarrow L^2(\kappa;\mathbb K)
\end{equation*}
is closable. Denote its closure by $D_{\mathbb K}$. Explicitly,
\begin{equation*}
    \operatorname{Dom}(D_{\mathbb K})
    =\left\{u\in L^2(\kappa):
    \begin{array}{l}
        \text{there exist $F_n\in\cF C_b^\infty$ and $G\in L^2(\kappa;\mathbb K)$ such that}\\
        F_n\to u\text{ in }L^2(\kappa)\text{ and }
        \nabla_{\mathbb K}F_n\to G\text{ in }L^2(\kappa;\mathbb K)
    \end{array}\right\},
    \qquad D_{\mathbb K}u:=G.
\end{equation*}
Closability makes the limit $G$ independent of the approximating sequence. The definition of the closure gives
\begin{equation*}
    \operatorname{Dom}(\cE)=\operatorname{Dom}(D_{\mathbb K}),
    \qquad
    \cE(u,v)=\int_{\mathbb X}
    \langle D_{\mathbb K}u,D_{\mathbb K}v\rangle_{\mathbb K}\,d\kappa.
\end{equation*}
Thus symmetry, non-negativity and closedness follow directly from the closed operator $D_{\mathbb K}$.

The operator $D_{\mathbb K}$ is, by construction, the closure of the cylinder gradient, and the cylinder core contains the constant function $1$. Thus the two hypotheses of \cref{lem:closed-gradient-markov} hold. By \cref{lem:closed-gradient-markov}, the closed form has the Markov property and is conservative.

\pagebreak[3]
\par\noindent\textbf{Step 2. Locality and diffusion calculus.}\par
We first verify locality directly. Choose a bounded smooth $\psi$ with $\psi'$ compactly supported, $0\le\psi'\le1$ and $\psi'(0)=1$. For $c\in\mathbb R$ and $\varepsilon>0$, set
\begin{equation*}
    \theta_\varepsilon(r) = \varepsilon\bigl[\psi((r-c)/\varepsilon)-\psi(-c/\varepsilon)\bigr].
\end{equation*}
This is a smooth normal contraction. The smooth chain rule proved in \cref{lem:closed-gradient-markov} therefore gives
\begin{equation*}
    D_{\mathbb K}\theta_\varepsilon(u) = \psi'((u-c)/\varepsilon)D_{\mathbb K}u
    \longrightarrow
    \mathbf1_{\{u=c\}}D_{\mathbb K}u
    \quad\text{in }L^2(\kappa;\mathbb K).
\end{equation*}
The convergence follows by dominated convergence, whereas $\|\theta_\varepsilon(u)\|_{L^2(\kappa)}\le2\varepsilon\|\psi\|_\infty\to0$. Closedness implies $\mathbf1_{\{u=c\}}D_{\mathbb K}u=0$. If $u$ is constant on an open neighbourhood $U$ of the support of $v$, then $D_{\mathbb K}u=0$ on $U$, while $D_{\mathbb K}v=0$ on the open set $\mathbb X\setminus\operatorname{supp}v$, where $v$ is identically zero. The closed-gradient representation therefore gives $\cE(u,v)=0$, and the form is strongly local.

For bounded $u,v\in\operatorname{Dom}(\cE)$, define
\begin{equation*}
    \Gamma(u,v) = \langle D_{\mathbb K}u,D_{\mathbb K}v\rangle_{\mathbb K}.
\end{equation*}
Pointwise Cauchy--Schwarz in $\mathbb K$, followed by Cauchy--Schwarz in $L^2(\kappa)$, gives
\begin{align*}
    \int_{\mathbb X}|\Gamma(u,v)|\,d\kappa
    &\le \left( \int_{\mathbb X}\|D_{\mathbb K}u\|_{\mathbb K}^2\,d\kappa \right)^{1/2}\left(\int_{\mathbb X}\|D_{\mathbb K}v\|_{\mathbb K}^2\,d\kappa\right)^{1/2}\\
    &= \cE(u,u)^{1/2}\cE(v,v)^{1/2}<\infty.
\end{align*}
Thus $\Gamma(u,v)\in L^1(\kappa)$, and
\begin{equation*}
    \cE(u,v)=\int_{\mathbb X}\Gamma(u,v)\,d\kappa.
\end{equation*}
If $\Phi\in C^1(\mathbb R)$ and $\Phi'$ is bounded on the range of $u$, choose a compact interval $I$ containing both $0$ and the essential range of $u$. Multiplying $\Phi'$ by a smooth cutoff equal to one on $I$, integrating from $0$ and adding $\Phi(0)$ produces a function $\widehat\Phi\in C_b^1(\mathbb R)$ such that $\widehat\Phi=\Phi$ on $I$ and $\widehat\Phi'$ has compact support. Let $\rho$ be a standard mollifier and set
\begin{equation*}
    \Phi_n=\rho_{1/n}*\widehat\Phi.
\end{equation*}
Then $\Phi_n\in C_b^\infty(\mathbb R)$ and
\begin{equation*}
    \|\Phi_n-\widehat\Phi\|_\infty
    +\|\Phi_n'-\widehat\Phi'\|_\infty\longrightarrow0,
    \qquad
    M:=\sup_n\|\Phi_n'\|_\infty<\infty.
\end{equation*}

For each $n$, the function
\begin{equation*}
    \theta_n(r)=\frac{\Phi_n(r)-\Phi_n(0)}{\max\{1,M\}}
\end{equation*}
is a smooth normal contraction. The smooth chain rule in \cref{lem:closed-gradient-markov} therefore applies to $\theta_n(u)$. Since $1$ belongs to the cylinder core and $D_{\mathbb K}1=0$, scaling and adding back the constant $\Phi_n(0)$ give
\begin{equation*}
    \Phi_n(u)\in\operatorname{Dom}(D_{\mathbb K}),
    \qquad
    D_{\mathbb K}\Phi_n(u)=\Phi_n'(u)D_{\mathbb K}u.
\end{equation*}
Finally, the preceding uniform convergence and the fact that $u\in I$ almost everywhere imply
\begin{align*}
    \Phi_n(u)&\longrightarrow\Phi(u)
    &&\text{in }L^2(\kappa),\\
    \Phi_n'(u)D_{\mathbb K}u&\longrightarrow
    \Phi'(u)D_{\mathbb K}u
    &&\text{in }L^2(\kappa;\mathbb K).
\end{align*}
Closedness then gives
\begin{equation*}
    \Phi(u)\in\operatorname{Dom}(D_{\mathbb K}),
    \qquad
    D_{\mathbb K}(\Phi(u))=\Phi'(u)D_{\mathbb K}u.
\end{equation*}
Consequently,
\begin{equation*}
    \Gamma(\Phi(u),v) = \Phi'(u)\Gamma(u,v)
    \qquad\kappa\text{-a.e.},
\end{equation*}
which is precisely the diffusion calculus in \cref{ass:diffusion-calculus}.
\end{proof}

We also record a sufficient condition for closability when starting from a Gibbs measure. This criterion supplies the first step of the construction discussed in \cref{rmk:form-construction}. Write $h_j=\sqrt{q_j}e_j$ for the Fourier Cameron--Martin basis; it is orthonormal in $\mathbb K$. We write $\partial_hV$ for differentiation along $h\in\mathbb K$. When these derivatives have an $\mathbb H$-Riesz representative, it is denoted $\nabla_{\mathbb H}V$, so $\partial_hV=\langle\nabla_{\mathbb H}V,h\rangle_{\mathbb H}$.

\begin{definition}[Logarithmic derivative]
\label{def:logarithmic-derivative}
Let $\kappa$ be a probability measure on $\mathbb X$ and let $h\in\mathbb K$. A function $\beta_h\in L^1(\kappa)$ is called the \emph{logarithmic derivative} of $\kappa$ in the direction $h$ if
\begin{equation*}
    \int_{\mathbb X}\partial_hF\,d\kappa = \int_{\mathbb X}F\beta_h\,d\kappa
\end{equation*}
for every $F\in\cF C_b^\infty$. We use this positive-sign convention for $\beta_h$.
\end{definition}

\begin{assumption}[Gibbs integration-by-parts hypotheses]\label{ass:gibbs-ibp}
Let $\mathbb X=C(\T;\mathbb R)$, let $\mu_Q$ be the Gaussian measure fixed in \S\ref{sec:conventions}, and let $\kappa=Z^{-1}e^{-V}\mu_Q$ with $0<Z<\infty$ and $V<\infty$ $\mu_Q$-almost everywhere. For every $j$:
\begin{enumerate}[label=(\roman*)]
    \item $e^{-V}$ belongs to $W^{1,1}_{h_j}(\mu_Q)$, the closure of the class $\cF C_b^\infty$ of smooth cylinder functions under the norm
    \begin{equation*}
        \|u\|_{W^{1,1}_{h_j}(\mu_Q)} := \|u\|_{L^1(\mu_Q)} + \|\partial_{h_j}u\|_{L^1(\mu_Q)},
    \end{equation*}
    with weak derivative $\partial_{h_j}(e^{-V}) = -(\partial_{h_j}V)e^{-V}$;
    \item $\langle\phi,e_j\rangle_{\mathbb H}\in L^2(\kappa)$ and $\partial_{h_j}V\in L^2(\kappa)$.
\end{enumerate}
\end{assumption}

\begin{proposition}[Gibbs logarithmic derivatives and closability]
\label{prop:closability-criterion}
Under \cref{ass:gibbs-ibp},
\begin{equation}\label{eq:beta-form}
    \beta_j(\phi) = \frac{\langle\phi,e_j\rangle_{\mathbb H}}{\sqrt{q_j}} + \partial_{h_j}V(\phi)\in L^2(\kappa)
\end{equation}
and, for every smooth bounded Fourier cylinder $F$,
\begin{equation}\label{eq:log-derivative}
    \int_{\mathbb X}\partial_{h_j}F\,d\kappa=\int_{\mathbb X}F\beta_j\,d\kappa.
\end{equation}
Consequently $\nabla_{\mathbb K}$ is closable from $L^2(\kappa)$ to $L^2(\kappa;\mathbb K)$; equivalently, the pre-form \eqref{eq:pre-form} is closable.
\end{proposition}

\begin{proof}
Put $u=Fe^{-V}$ and $\gamma_j(\phi)=\langle\phi,e_j\rangle_{\mathbb H}/\sqrt{q_j}$. The bounded cylinder product rule and \cref{ass:gibbs-ibp}~(i) give $u\in W^{1,1}_{h_j}(\mu_Q)$. Moreover, $\gamma_ju\in L^1(\mu_Q)$ by part~(ii). To justify integration by parts for this $W^{1,1}_{h_j}(\mu_Q)$ function, choose $\chi\in C_c^\infty(\mathbb R)$ equal to one near zero, with $0\le\chi\le1$, and put $\chi_R(r)=\chi(r/R)$. For smooth cylinder approximations $u_n\to u$ in $W^{1,1}_{h_j}(\mu_Q)$, Gaussian integration by parts gives
\begin{equation*}
    \int\chi_R(\gamma_j)\partial_{h_j}u_n\,d\mu_Q = \int [\gamma_j\chi_R(\gamma_j) - \chi_R'(\gamma_j)]u_n \, d\mu_Q.
\end{equation*}
Here $\partial_{h_j}\gamma_j=1$. For fixed $R$, all multipliers are bounded, so passage to $n\to\infty$ is justified by $W^{1,1}$ convergence. Since $\|\chi_R'\|_\infty\le\|\chi'\|_\infty/R$, dominated convergence as $R\to\infty$ yields
\begin{equation*}
    \int\partial_{h_j}(Fe^{-V})\,d\mu_Q = \int\gamma_jFe^{-V}\,d\mu_Q.
\end{equation*}
Expanding the weak derivative and dividing by $Z$ proves \eqref{eq:log-derivative}; \cref{ass:gibbs-ibp}(ii) gives $\beta_j\in L^2(\kappa)$. If $F_n\to0$ in $L^2(\kappa)$ and $\nabla_{\mathbb K}F_n\to G$ in $L^2(\kappa;\mathbb K)$, apply \eqref{eq:log-derivative} to $F_n\psi$:
\begin{equation*}
    \int\psi\,\partial_{h_j}F_n\,d\kappa = \int F_n(\psi\beta_j-\partial_{h_j}\psi)\,d\kappa\longrightarrow0.
\end{equation*}
Smooth bounded Fourier cylinders are dense in $L^2(\kappa)$, so $\langle G,h_j\rangle_{\mathbb K}=0$ almost everywhere for every $j$. A common full-measure set and completeness of $(h_j)$ give $G=0$, proving closability.
\end{proof}

The next elementary measure-theoretic fact identifies the full relative entropy from an increasing family of observations.

\begin{lemma}[Relative entropy under increasing projections]\label{lem:entropy-projections}
Let $\nu=h\mu$ be probability measures, and let $(\mathcal F_N)$ be increasing sub-sigma-fields generating the ambient sigma-field modulo $\mu$-null sets. Put $h_N=\mathbb E_\mu[h\mid\mathcal F_N]$. Then
\begin{equation*}
    \Ent_\mu(h_N)\uparrow\Ent_\mu(h),
\end{equation*}
with the value $+\infty$ allowed. In particular, if $\mathcal F_N=\sigma(P_N)$, then

\begin{equation*}
    \RelEnt{\nu\circ P_N^{-1}}{\mu\circ P_N^{-1}}
    \uparrow\RelEnt{\nu}{\mu}.
\end{equation*}
\end{lemma}

\begin{proof}
The conditional-expectation martingale converges to $h$ almost surely and in $L^1(\mu)$. Let $\Phi(r)=r\log r-r+1$, with $\Phi(0)=1$. Since all the densities have mass one, $\Ent_\mu(h_N)=\int\Phi(h_N)\,d\mu$. Conditional Jensen gives
\begin{equation*}
    \int\Phi(h_N)\,d\mu
    \le \int\Phi(h_{N+1})\,d\mu
    \le \int\Phi(h)\,d\mu.
\end{equation*}
Since $\Phi\ge0$, Fatou's lemma gives the reverse limiting inequality:
\begin{equation*}
    \int\Phi(h)\,d\mu \le \liminf_{N\to\infty}\int\Phi(h_N)\,d\mu.
\end{equation*}
For $\mathcal F_N=\sigma(P_N)$, the pullback of the projected Radon--Nikodym density is $h_N$, by the defining conditional-expectation identity. This proves the last assertion.
\end{proof}

\section{Dirichlet-Form Verification for the Nonlinear Models}
\label{app:spde-details}

This appendix verifies the Dirichlet-form hypotheses for the quartic Gibbs dynamics and its negative quadratic perturbation. We apply the well-posedness and symmetry results of Hairer, Stuart and Voss, prove that the semigroup energy is the closed Fourier-gradient form, and check the initial laws used in the entropy estimates.

\subsection{Well-Posedness and Gibbs Measures}
\label{app:nonlinear-gibbs-verification}

We first record the Gaussian integrability and support facts used in the verification, together with the Cameron--Martin embedding for the field state space.

\begin{lemma}[Gaussian integrability, moments and support \cite{Bogachev1998}]
\label{lem:gaussian-measure-facts}
Let $\gamma$ be a centred Gaussian Radon probability measure on a separable Banach space $\mathbb E$, and let $\mathbb H_\gamma$ be its Cameron--Martin space. Then there exists $a>0$ such that
\begin{equation*}
    \int_{\mathbb E}\exp\!\bigl(a\|x\|_{\mathbb E}^2\bigr)\,d\gamma(x)<\infty,
\end{equation*}
and
\begin{equation*}
    \operatorname{supp}\gamma = \overline{\mathbb H_\gamma}^{\,\mathbb E}.
\end{equation*}
If $Z$ is a centred real Gaussian random variable, then, for every $p>0$,
\begin{equation*}
    \E|Z|^p=c_p\bigl(\E|Z|^2\bigr)^{p/2}, \qquad c_p=\E|\xi|^p,\quad \xi\sim\cN(0,1).
\end{equation*}
\end{lemma}

The space $\mathbb K$ in \eqref{eq:field-spaces} equals the periodic Sobolev space $H^1(\T)$ as a set. Its Cameron--Martin inner product and norm are
\begin{align*}
    \langle h,k\rangle_{\mathbb K}
    &= \langle Q^{-1/2}h,Q^{-1/2}k\rangle_{\mathbb H},\\
    \|h\|_{\mathbb K}^2
    &= \int_{\T}\bigl(|\partial_xh|^2+m|h|^2\bigr)\,dx.
\end{align*}
Thus $\|\cdot\|_{\mathbb K}$ is equivalent to the usual periodic $H^1$ norm.

\begin{lemma}[Cameron--Martin embedding]
\label{lem:cameron-martin-embedding}
The inclusion
\begin{equation*}
    i:\mathbb K\longrightarrow\mathbb X
\end{equation*}
is continuous and has dense range.
\end{lemma}

\begin{proof}
The one-dimensional periodic Sobolev inequality and the equivalence of the preceding norms give
\begin{equation*}
    \|h\|_{\mathbb X}=\|h\|_\infty
    \le C\|h\|_{H^1(\T)}
    \le C_m\|h\|_{\mathbb K},
    \qquad h\in\mathbb K,
\end{equation*}
which proves continuity. Every trigonometric polynomial belongs to $H^1(\T)=\mathbb K$, and the Fej\'er means of a continuous periodic function converge uniformly to that function. Hence the range of $i$ is dense in $C(\T)=\mathbb X$.
\end{proof}

The two one-dimensional nonlinear models use the same Gaussian sample regularity and moment bounds.

\begin{lemma}[Gaussian regularity]\label{lem:Gaussian-regularity}
Let $\mu_Q = \cN(0,Q)$ with $Q = A^{-1}$, $A = -\partial_x^2 + m$, $m>0$. For every $s < 1/2$, $\mu_Q(H^s(\T)) = 1$, and $\|\phi\|_{H^s}$ possesses finite moments of all orders under $\mu_Q$. Choosing $s \in (1/3, 1/2)$, the one-dimensional Sobolev embedding $H^s(\T) \hookrightarrow L^6(\T)$ yields $\E_{\mu_Q}\|\phi\|_{L^6}^r < \infty$ for all $0\le r<\infty$. The same conclusion holds for any mass-quenched initial law $\mu_0 = \cN(0, A_0^{-1})$ with $A_0 = -\partial_x^2 + m_0$, $m_0>0$.
\end{lemma}

\begin{proof}
Expand $\phi$ in the Fourier basis: $\phi = \sum_j q_j^{1/2}\xi_j e_j$ with $\xi_j \sim \cN(0,1)$ i.i.d.\ and $q_j$ the corresponding eigenvalues of $Q$ fixed after \eqref{eq:basis}. For $s<1/2$,
\begin{equation*}
    \E_{\mu_Q}\|\phi\|_{H^s}^2 = \sum_{k\in\Z}\frac{(1+k^2)^s}{k^2+m}<\infty.
\end{equation*}
Thus the Fourier series defines a centred Gaussian Radon random variable in the separable Hilbert space $H^s(\T)$. Applying the Fernique estimate in \cref{lem:gaussian-measure-facts} with $\mathbb E=H^s(\T)$ gives an $a_s>0$ such that
\begin{equation*}
    \E_{\mu_Q}\exp(a_s\|\phi\|_{H^s}^2)<\infty,
\end{equation*}
and hence all polynomial moments are finite. The embedding $H^s(\T)\subset L^6(\T)$ for $s>1/3$ completes the proof. For $\mu_0=\mathcal N(0,A_0^{-1})$, the same calculation replaces the denominator $k^2+m$ by $k^2+m_0$; the summability threshold and the application of \cref{lem:gaussian-measure-facts} are unchanged.
\end{proof}

The next lemma applies the existing SPDE theory and verifies its hypotheses for both potentials.

\begin{lemma}[Well-posedness and symmetry of the quartic dynamics]
\label{lem:quartic-analytic-inputs}
Let $V$ be either $V_4$ or $V_{\rm AC}$ in \cref{thm:Phi4-model,rmk:quartic-perturbation}, let $m,\lambda>0$, and let $\kappa=Z^{-1}e^{-V}\mu_Q$ on $\mathbb X=C(\T)$. Then:
\begin{enumerate}[label=(\roman*)]
    \item The constant $Z$ belongs to $(0,\infty)$. The Gaussian measure $\mu_Q$ is a full-support Radon measure on $\mathbb X$. The Gibbs measure $\kappa$ is equivalent to $\mu_Q$, has full support, has a bounded density with respect to $\mu_Q$, and satisfies
    \begin{equation*}
        \int_{\mathbb X}\|\phi\|_\infty^r\,d\kappa(\phi)<\infty,
        \qquad 0\le r<\infty.
    \end{equation*}
    \item For every $x\in\mathbb X$, the corresponding equation \eqref{eq:Phi4-spde} or \eqref{eq:AC-spde} has a unique global probabilistically strong $\mathbb X$-valued solution with continuous paths. Writing $b(\phi)=-\phi-Q\nabla_{\mathbb H}V(\phi)$, it satisfies
    \begin{equation}\label{eq:X-valued-quartic-SDE}
        \phi_t^x=x+\int_0^t b(\phi_s^x)\,ds+\sqrt2\,W_t^Q
    \end{equation}
    in $\mathbb X$. Its Markov transition semigroup $P_tF(x)=\mathbb E[F(\phi_t^x)]$ preserves $\kappa$ and is $\kappa$-symmetric. There exists $N>0$ and, for every $p>0$, there are $C_p,K_p,\gamma_p>0$ such that
    \begin{equation*}
        \mathbb E\|\phi_t^x\|_\infty^p
        \le C_p(1+\|x\|_\infty)^{Np}e^{-\gamma_pt}+K_p,
        \qquad t\ge0.
    \end{equation*}
    \item The drift $b:\mathbb X\to\mathbb X$ is continuous and locally Lipschitz, and belongs to $L^2(\kappa;\mathbb X)$. The operators $(P_t)$ extend to a conservative, strongly continuous, symmetric Markov contraction semigroup on $L^2(\kappa)$.
\end{enumerate}
\end{lemma}

\begin{proof}
\par\noindent\textbf{Step 1. State space and Gibbs measure.}\par
For $x,y\in\T$, Fourier summation gives the increment estimate
\begin{equation*}
    \E_{\mu_Q}|\phi(x)-\phi(y)|^2 = \sum_jq_j|e_j(x)-e_j(y)|^2 \le C|x-y|.
\end{equation*}
Since each increment is a centred real Gaussian variable, the scalar moment formula in \cref{lem:gaussian-measure-facts} gives, for every $p\ge2$,
\begin{equation*}
    \E_{\mu_Q}|\phi(x)-\phi(y)|^p \le C_p|x-y|^{p/2}.
\end{equation*}
Given $\gamma<\frac12$, choose $p$ with $p(\frac12-\gamma)>1$. Kolmogorov's continuity theorem then gives a version in $C^\gamma(\T)$, and hence $\mu_Q$ is a Gaussian Radon measure on $\mathbb X=C(\T)$. Its Cameron--Martin space is continuously and densely embedded in $\mathbb X$ by \cref{lem:cameron-martin-embedding}; the support formula in \cref{lem:gaussian-measure-facts} therefore gives $\operatorname{supp}\mu_Q=\mathbb X$. Both Gibbs densities are strictly positive. Also $e^{-V_4}\le1$, while
\begin{equation*}
    \frac{r^4}{4} - \frac{r^2}{2} = \frac{(r^2-1)^2}{4} - \frac14\ge - \frac14
\end{equation*}
gives $e^{-V_{\rm AC}}\le e^{\lambda\pi/2}$. Thus $\kappa$ is Radon, equivalent to $\mu_Q$, and has full support. More precisely, for a finite constant $C_V$,
\begin{equation}\label{eq:quartic-density-domination}
    0<\frac{d\kappa}{d\mu_Q}\le C_V
    \qquad\mu_Q\text{-a.e.}.
\end{equation} The Fernique estimate in \cref{lem:gaussian-measure-facts}, applied on $C(\T)$, gives exponential integrability of a positive multiple of $\|\phi\|_\infty^2$; consequently,
\begin{equation}\label{eq:gibbs-sup-moments}
    \int_{\mathbb X}\|\phi\|_\infty^r\,d\kappa(\phi)<\infty,
    \qquad 0\le r<\infty.
\end{equation}

\pagebreak[3]
\par\noindent\textbf{Step 2. Well-posedness and moment bounds.}\par
Hairer, Stuart and Voss established global strong well-posedness and moment bounds for nonlinear preconditioned SPDEs in \cite[Theorem~2.10]{HairerStuartVoss2007}. Under the corresponding gradient hypotheses, \cite[Theorem~3.6]{HairerStuartVoss2007} also proves symmetry of the Gibbs semigroup and invariance of the Gibbs measure. For the two equations in \cref{thm:Phi4-model,rmk:quartic-perturbation}, the operators in their notation are
\begin{equation*}
    L_{\rm HSV}=-A,\qquad
    G_{\rm HSV}=Q=A^{-1},\qquad
    F_{\rm HSV}=-\nabla_{\mathbb H}V.
\end{equation*}
We verify the hypotheses for both potentials; the calculation also records the required Gaussian regularity and the dissipativity estimates for the two polynomial drifts. The operator $L_{\rm HSV}$ is self-adjoint and strictly dissipative on $\mathbb H=L^2(\T)$, the Fourier basis in \eqref{eq:basis} is a complete eigenbasis, and its analytic semigroup leaves $\mathbb X=C(\T)$ invariant with
\begin{equation*}
    \|e^{tL_{\rm HSV}}x\|_\infty\le e^{-mt}\|x\|_\infty.
\end{equation*}
Choose $\alpha\in(\frac14,\frac12)$. Then $\operatorname{Dom}(A^\alpha)=H^{2\alpha}(\T)$ embeds continuously and densely into $C(\T)$, while
\begin{equation*}
    \Tr(A^{-2\alpha}) = \sum_{k\in\mathbb Z}(k^2+m)^{-2\alpha} < \infty.
\end{equation*}
If $Y\sim\mathcal N(0,A^{-2\alpha})$, Fourier summation gives
\begin{equation}\label{eq:Aalpha-increments}
    \mathbb E|Y(x)-Y(y)|^2\le C_\alpha|x-y|^{4\alpha-1}.
\end{equation}
Together with the scalar Gaussian moment formula in \cref{lem:gaussian-measure-facts}, Kolmogorov's continuity theorem shows that $Y$ has a $C(\T)$-valued version. These facts verify Assumptions A1--A2.

The maps
\begin{equation*}
    F_4(z)=-\lambda z^3,
    \qquad F_{\rm AC}(z)=-\lambda(z^3-z)
\end{equation*}
belong to $C^1(\mathbb X;\mathbb X)$ and are locally Lipschitz. Their derivatives are the multiplication operators
\begin{equation*}
    \mathrm DF_4(z)[h]=-3\lambda z^2h,\qquad
    \mathrm DF_{\mathrm{AC}}(z)[h]=\lambda(1-3z^2)h,
\end{equation*}
so $\|\mathrm DF(z)\|_{\mathcal L(\mathbb X)}\le C(1+\|z\|_\infty^2)$ and $\|F(z)\|_\infty\le C(1+\|z\|_\infty^3)$ in both cases. These are the polynomial bounds required in A3. For the dissipativity condition A5, write $z=x+y$. Young's inequality gives
\begin{align*}
    \langle x,F_4(x+y)\rangle_{\mathbb H}
    &= -\lambda\int_\T z^4\,dx+\lambda\int_\T yz^3\,dx\\
    &\le -\frac\lambda2\|z\|_{L^4}^4 + C_\lambda(1+\|y\|_\infty^4),\\
    \langle x,F_{\rm AC}(x+y)\rangle_{\mathbb H}
    &= -\lambda\int_\T z^4\,dx+\lambda\int_\T z^2\,dx + \lambda\int_\T yz^3\,dx-\lambda\int_\T yz\,dx\\
    &\le -\frac\lambda4\|z\|_{L^4}^4 + C_\lambda(1+\|y\|_\infty^4).
\end{align*}
Finally, $G_{\rm HSV}=Q$ is positive, self-adjoint and trace class with dense range, while Step~1 gives $\mathcal N(0,Q)(\mathbb X)=1$. Moreover,
\begin{equation*}
    G_{\rm HSV}L_{\rm HSV}=-I,
\end{equation*}
which verifies A6--A7. Thus Assumptions A1--A3 and A5--A7 all hold. \cite[Theorem~2.10]{HairerStuartVoss2007} then gives a unique global $\mathbb X$-valued strong solution for every initial point $x\in\mathbb X$, together with constants $N>0$ and, for each $p>0$, constants $C_p,K_p,\gamma_p>0$ such that
\begin{equation*}
    \mathbb E\|\phi_t^x\|_\infty^p
    \le C_p(1+\|x\|_\infty)^{Np}e^{-\gamma_pt}+K_p,
    \qquad t\ge0.
\end{equation*}
Since $U=-V$ is bounded above in both cases and $F_{\rm HSV}=U'$, \cite[Theorem~3.6]{HairerStuartVoss2007} also applies. Thus the SPDE transition semigroup preserves $\kappa$ and is $\kappa$-symmetric.

\par\noindent\textbf{Step 3. Drift integrability and strong continuity.}\par
The polynomial maps $\nabla_{\mathbb H}V:C(\T)\to C(\T)$ are locally Lipschitz, and $Q=A^{-1}$ acts boundedly on $C(\T)$. Consequently $b$ is continuous and locally Lipschitz, with
\begin{equation*}
    \|b(\phi)\|_\infty\le C(1+\|\phi\|_\infty^3).
\end{equation*}
The moment bound \eqref{eq:gibbs-sup-moments} therefore gives $b\in L^2(\kappa;\mathbb X)$.

Global existence gives $P_t1=1$. Invariance and Jensen's inequality give
\begin{equation*}
    \|P_tF\|_{L^2(\kappa)}^2
    \le\int_{\mathbb X}P_t(F^2)\,d\kappa
    =\|F\|_{L^2(\kappa)}^2,
\end{equation*}
so the transition operators extend to $L^2(\kappa)$ as Markov contractions; their symmetry follows from Step~2. For bounded continuous $F$, path continuity and dominated convergence yield
\begin{equation*}
    \|P_tF-F\|_{L^2(\kappa)}^2
    \le\int_{\mathbb X}\mathbb E\bigl[|F(\phi_t^x)-F(x)|^2\bigr]\,d\kappa(x)
    \longrightarrow0
    \qquad(t\downarrow0).
\end{equation*}
Since $\kappa$ is Radon on the separable Banach space $\mathbb X$, bounded continuous functions are dense in $L^2(\kappa)$. The contraction bound extends strong continuity to all of $L^2(\kappa)$.
\end{proof}

\subsection{Verification of the Gradient-Form Hypotheses}
\label{app:spde-realisation}

The next lemma verifies the closed-form hypotheses for the given SPDE semigroup. Agreement of the energy on cylinder functions is only the first step: the last two steps prove that those functions form a core for the entire canonical form.

\begin{lemma}[Gradient form of the quartic dynamics]
\label{lem:quartic-direct-energy}
Let $(P_t)$ be the semigroup in \cref{lem:quartic-analytic-inputs}, and let $L_P$ be its non-positive self-adjoint generator on $L^2(\kappa)$. Its canonical Dirichlet form is
\begin{equation*}
    \operatorname{Dom}(\cE^P)
    =\operatorname{Dom}((-L_P)^{1/2}),
    \qquad
    \cE^P(u,v)=
    \langle(-L_P)^{1/2}u,(-L_P)^{1/2}v\rangle_{L^2(\kappa)}.
\end{equation*}
The Fourier-gradient pre-form \eqref{eq:pre-form} is closable, and its closure is exactly $(\cE^P,\operatorname{Dom}(\cE^P))$. Thus, if $D_{\mathbb K}$ denotes the closure of the Fourier cylinder gradient,
\begin{equation}\label{eq:quartic-canonical-gradient}
    \operatorname{Dom}(\cE^P)=\operatorname{Dom}(D_{\mathbb K}),
    \qquad
    \cE^P(u,v)=\int_{\mathbb X}
    \langle D_{\mathbb K}u,D_{\mathbb K}v\rangle_{\mathbb K}\,d\kappa.
\end{equation}
In particular, the Fourier cylinders are a form core, the form satisfies \cref{ass:equilibrium,ass:diffusion-calculus}, and its semigroup is the given SPDE semigroup. These conclusions hold for every $m,\lambda>0$ in both models.
\end{lemma}

\begin{proof}\leavevmode
\par\noindent\textbf{Step 1. Cylinder generators and their energy.}\par\nopagebreak
For a smooth bounded Fourier cylinder $F$, set
\begin{equation*}
    L_0F(\phi)
    =\operatorname{Tr}\!\bigl(Q\nabla_{\mathbb H}^2F(\phi)\bigr)
    +\langle b(\phi),\nabla_{\mathbb H}F(\phi)\rangle_{\mathbb H},
    \qquad
    \Gamma_0(F,G)=
    \langle Q^{1/2}\nabla_{\mathbb H}F,
    Q^{1/2}\nabla_{\mathbb H}G\rangle_{\mathbb H}.
\end{equation*}
The derivatives of $F$ are bounded and depend on finitely many coordinates. Hence the trace term is bounded, and $b\in L^2(\kappa;\mathbb X)$ implies $L_0F\in L^2(\kappa)$.

Start the known solution with law $\kappa$. Applying the finite-dimensional It\^o formula to the coordinates on which $F$ depends gives
\begin{equation*}
    F(\phi_t)-F(\phi_0)
    =\int_0^tL_0F(\phi_s)\,ds+M_t^F,
    \qquad
    \langle M^F\rangle_t
    =2\int_0^t\Gamma_0(F,F)(\phi_s)\,ds.
\end{equation*}
Here the factor $2$ is the square of the noise amplitude $\sqrt2$ in \eqref{eq:X-valued-quartic-SDE}. Since $\Gamma_0(F,F)$ is bounded, $M^F$ is a square-integrable martingale. Stationarity also gives
\begin{equation*}
    \mathbb E_\kappa\left|\int_0^tL_0F(\phi_s)\,ds\right|^2
    \le t^2\|L_0F\|_{L^2(\kappa)}^2.
\end{equation*}
Conditioning on $\phi_0$ and using the Markov property therefore yields
\begin{equation*}
    P_tF-F=\int_0^tP_sL_0F\,ds
    \qquad\text{in }L^2(\kappa),
\end{equation*}
where the integral on the right is a Bochner integral. Strong continuity implies
\begin{equation*}
    \frac{P_tF-F}{t}\longrightarrow L_0F
    \quad\text{in }L^2(\kappa),
\end{equation*}
so $F\in\operatorname{Dom}(L_P)$ and $L_PF=L_0F$.

For two such cylinders, the product rule gives
\begin{equation*}
    L_0(FG)=F L_0G+G L_0F+2\Gamma_0(F,G).
\end{equation*}
The product $FG$ is again a cylinder. Invariance gives $\int L_0(FG)\,d\kappa=0$, while symmetry gives $\langle F,L_PG\rangle_{L^2(\kappa)}=\langle G,L_PF\rangle_{L^2(\kappa)}$. Consequently,
\begin{equation}\label{eq:quartic-cylinder-energy}
    \cE^P(F,G)
    =-\langle F,L_PG\rangle_{L^2(\kappa)}
    =\int_{\mathbb X}\Gamma_0(F,G)\,d\kappa
    =\cE_0(F,G).
\end{equation}
Thus the energy on Fourier cylinders is determined by the SPDE generator and its invariant measure.

\par\noindent\textbf{Step 2. Closability and the closed gradient form.}\par\nopagebreak
The Fourier cylinder algebra is dense in $L^2(\kappa)$: its countable coordinate family separates points of the Polish space $C(\T)$ and generates its Borel sigma-field, and bounded smooth functions are dense in each finite-dimensional marginal. Equation \eqref{eq:quartic-cylinder-energy} identifies $\cE_0$ with the restriction of the closed form $\cE^P$. This restriction is closable. Indeed, if $F_n\to0$ in $L^2(\kappa)$ and $\cE_0(F_n-F_m,F_n-F_m)\to0$, then $(F_n)$ is Cauchy for the $\cE^P_1$ norm. Closedness of $\cE^P$ identifies its limit with $0$, so $\cE_0(F_n,F_n)\to0$.

The cylinder gradient is closable as an operator as well. If $F_n\to0$ in $L^2(\kappa)$ and $\nabla_{\mathbb K}F_n\to G$ in $L^2(\kappa;\mathbb K)$, the preceding argument gives
\begin{equation*}
    \|G\|_{L^2(\kappa;\mathbb K)}^2
    =\lim_n\cE_0(F_n,F_n)=0.
\end{equation*}
Let $\cE_{\min}$ be the resulting closed gradient form. Its domain is a closed subspace of $\operatorname{Dom}(\cE^P)$ for the $\cE^P_1$ norm, and
\begin{equation*}
    \operatorname{Dom}(\cE_{\min})\subseteq\operatorname{Dom}(\cE^P),
    \qquad
    \cE_{\min}(u,v)=\cE^P(u,v)
    \quad\text{on }\operatorname{Dom}(\cE_{\min}).
\end{equation*}
The remaining steps prove that this inclusion is an equality.

\par\noindent\textbf{Step 3. Synchronous coupling in Cameron--Martin directions.}\par\nopagebreak
Set
\begin{equation*}
    a=\begin{cases}0,&V=V_4,\\ \lambda,&V=V_{\rm AC},\end{cases}
    \qquad \rho=1-\frac{a}{m}.
\end{equation*}
Pointwise monotonicity of $r\mapsto r^3$ gives, for $x,y\in C(\T)$,
\begin{equation*}
    \langle\nabla_{\mathbb H}V(x)-\nabla_{\mathbb H}V(y),x-y\rangle_{\mathbb H}
    \ge-a\|x-y\|_{\mathbb H}^2.
\end{equation*}
Couple the solutions starting from $x+h$ and $x$, where $h\in\mathbb K$, by using the same $Q$-Wiener process, and write $\delta_t=\phi_t^{x+h}-\phi_t^x$. The noise cancels. The operator $Q$ maps $C(\T)$ continuously into $\mathbb K$, since
\begin{equation*}
    \|Qf\|_{\mathbb K}^2
    =\langle Qf,f\rangle_{\mathbb H}
    \le m^{-1}\|f\|_{\mathbb H}^2.
\end{equation*}
The difference has the variation-of-constants representation
\begin{equation*}
    \delta_t=e^{-t}h-\int_0^t e^{-(t-s)}Q
    \bigl[\nabla_{\mathbb H}V(\phi_s^{x+h})
    -\nabla_{\mathbb H}V(\phi_s^x)\bigr]\,ds.
\end{equation*}
The integrand is continuous with values in $\mathbb K$, so this identity shows that $\delta$ is a continuously differentiable $\mathbb K$-valued path. Using $\|h\|_{\mathbb K}^2\ge m\|h\|_{\mathbb H}^2$, we obtain
\begin{align*}
    \frac12\frac{d}{dt}\|\delta_t\|_{\mathbb K}^2
    &=-\|\delta_t\|_{\mathbb K}^2
    -\langle\nabla_{\mathbb H}V(\phi_t^{x+h})
    -\nabla_{\mathbb H}V(\phi_t^x),\delta_t\rangle_{\mathbb H}\\
    &\le-\|\delta_t\|_{\mathbb K}^2+a\|\delta_t\|_{\mathbb H}^2
    \le-\rho\|\delta_t\|_{\mathbb K}^2.
\end{align*}
Gronwall's lemma gives
\begin{equation}\label{eq:quartic-K-coupling}
    \|\phi_t^{x+h}-\phi_t^x\|_{\mathbb K}
    \le e^{-\rho t}\|h\|_{\mathbb K}.
\end{equation}
For a cylinder $g$, let $L_g=\sup_x\|\nabla_{\mathbb K}g(x)\|_{\mathbb K}<\infty$. The ordinary chain rule along Cameron--Martin line segments and \eqref{eq:quartic-K-coupling} imply
\begin{equation}\label{eq:quartic-semigroup-K-Lipschitz}
    |P_tg(x+h)-P_tg(x)|
    \le L_g e^{-\rho t}\|h\|_{\mathbb K},
    \qquad x\in\mathbb X,\ h\in\mathbb K.
\end{equation}
The right-hand side is finite for each fixed $t$ for all $m,\lambda>0$, including the Allen--Cahn regime $m\le\lambda$.

\par\noindent\textbf{Step 4. Approximation in the form domain.}\par\nopagebreak
We prove the approximation fact needed to use \eqref{eq:quartic-semigroup-K-Lipschitz}. Let $f:\mathbb X\to\mathbb R$ be bounded and Borel, and assume
\begin{equation*}
    |f(x+h)-f(x)|\le L_f\|h\|_{\mathbb K},
    \qquad x\in\mathbb X,\ h\in\mathbb K.
\end{equation*}
Write $h_j=\sqrt{q_j}e_j$ and $\xi_j(\phi)=\langle\phi,e_j\rangle_{\mathbb H}/\sqrt{q_j}$, so $(h_j)$ is an orthonormal basis of $\mathbb K$ and the $\xi_j$ are independent standard Gaussians under $\mu_Q$. Define
\begin{equation*}
    \Pi_n\phi=\sum_{j=0}^n\xi_j(\phi)h_j,
    \qquad
    Y_n=(I-\Pi_n)\phi,
    \qquad
    \nu_n=\operatorname{Law}_{\mu_Q}(Y_n).
\end{equation*}
The finite coordinate vector $(\xi_0,\ldots,\xi_n)$ and the $C(\T)$-valued Gaussian variable $Y_n$ are independent: every covariance between a coordinate and a continuous linear functional of $Y_n$ vanishes. Hence
\begin{equation*}
    f_n(z)=\int_{\mathbb X}f\!\left(\sum_{j=0}^nz_jh_j+y\right)\,d\nu_n(y)
\end{equation*}
is bounded by $\|f\|_\infty$, is Euclidean Lipschitz with constant at most $L_f$, and satisfies
\begin{equation*}
    f_n(\xi_0,\ldots,\xi_n)
    =\mathbb E_{\mu_Q}[f\mid\sigma(\xi_0,\ldots,\xi_n)].
\end{equation*}
The Fourier coordinates generate the Borel sigma-field of $\mathbb X$, so the conditional-expectation martingale converges to $f$ in $L^2(\mu_Q)$. This argument does not require convergence of $\Pi_n\phi$ in the supremum norm.

Convolve $f_n$ with a nonnegative smooth compactly supported probability mollifier on $\mathbb R^{n+1}$. Choose its scale small enough that the resulting $\widetilde f_n$ satisfies
\begin{equation*}
    \|\widetilde f_n-f_n\|_\infty\le\frac1n,
    \qquad
    \|\widetilde f_n\|_\infty\le\|f\|_\infty,
    \qquad
    \|\nabla\widetilde f_n\|_\infty\le L_f,
    \qquad n\ge1.
\end{equation*}
Indeed, convolution preserves the Lipschitz constant, and its uniform approximation error is bounded by $L_f$ times the first moment of the mollifier. All derivatives of $\widetilde f_n$ are bounded. Thus
\begin{equation*}
    F_n(\phi)=\widetilde f_n(\xi_0(\phi),\ldots,\xi_n(\phi))
\end{equation*}
is a smooth bounded Fourier cylinder, with
\begin{equation*}
    F_n\longrightarrow f\quad\text{in }L^2(\mu_Q),
    \qquad
    \nabla_{\mathbb K}F_n
    =\sum_{j=0}^n(\partial_j\widetilde f_n)(\xi_0,\ldots,\xi_n)h_j,
    \qquad
    \|\nabla_{\mathbb K}F_n\|_{\mathbb K}\le L_f.
\end{equation*}
The bounded-density estimate \eqref{eq:quartic-density-domination} gives $F_n\to f$ in $L^2(\kappa)$, while
\begin{equation*}
    \cE_{\min}(F_n,F_n)
    =\int\|\nabla_{\mathbb K}F_n\|_{\mathbb K}^2\,d\kappa
    \le L_f^2.
\end{equation*}
Consequently $(F_n)$ is bounded in the Hilbert space $\operatorname{Dom}(\cE_{\min})$ with its form norm. A weakly convergent subsequence in that space has $L^2(\kappa)$ limit $f$, because the whole sequence converges strongly to $f$ in $L^2(\kappa)$. Hence
\begin{equation*}
    f\in\operatorname{Dom}(\cE_{\min}),
    \qquad \cE_{\min}(f,f)\le L_f^2.
\end{equation*}
Apply this conclusion to $f=P_tg$, using boundedness of $P_tg$ and \eqref{eq:quartic-semigroup-K-Lipschitz}. We have proved
\begin{equation*}
    P_tg\in\operatorname{Dom}(\cE_{\min})
    \qquad(t>0,\ g\in\cF C_b^\infty).
\end{equation*}

\par\noindent\textbf{Step 5. Equality of the form domains.}\par\nopagebreak
Let $u\in\operatorname{Dom}(\cE^P)$ and choose Fourier cylinders $g_n\to u$ in $L^2(\kappa)$. For fixed $t>0$, the spectral theorem for $-L_P$ gives
\begin{equation*}
    \cE^P(P_tv,P_tv)
    \le \sup_{r\ge0}(r e^{-2tr})\,\|v\|_{L^2(\kappa)}^2
    =\frac{1}{2et}\|v\|_{L^2(\kappa)}^2,
    \qquad v\in L^2(\kappa).
\end{equation*}
Together with $L^2$-contractivity, this implies $P_tg_n\to P_tu$ in the $\cE^P_1$ norm. Each $P_tg_n$ lies in $\operatorname{Dom}(\cE_{\min})$ by Step~4, and this is a closed subspace for that norm by Step~2. Thus $P_tu\in\operatorname{Dom}(\cE_{\min})$.

Finally, spectral dominated convergence gives $P_tu\to u$ in the $\cE^P_1$ norm as $t\downarrow0$. Closedness of the same subspace now gives $u\in\operatorname{Dom}(\cE_{\min})$. This proves equality of the two domains and their forms, hence \eqref{eq:quartic-canonical-gradient}. The Fourier cylinders are a form core by the definition of $\cE_{\min}$. Applying \cref{prop:closed-gradient-properties} yields conservativity, strong locality and the diffusion chain rule. The closed form has the original semigroup $(P_t)$ because it is its canonical form.
\end{proof}

\subsection{Verification for the Quartic Potentials}

\begin{proposition}[Verification for the quartic Gibbs dynamics]
\label{prop:quartic-model-verification}
Let $V$ be either $V_4$ or $V_{\mathrm{AC}}$, let $\kappa=Z^{-1}e^{-V}\mu_Q$ on $\mathbb X=C(\T)$, and let $m,\lambda>0$. Then:
\begin{enumerate}[label=(\roman*)]
    \item The constant $Z$ belongs to $(0,\infty)$, and $\kappa$ is a full-support Radon probability measure equivalent to $\mu_Q$. It has finite moments of every order in $\|\cdot\|_\infty$.
    \item For every $x\in\mathbb X$, the corresponding equation \eqref{eq:Phi4-spde} or \eqref{eq:AC-spde} has a unique global probabilistically strong $\mathbb X$-valued solution with continuous paths. Its transition semigroup preserves $\kappa$ and is $\kappa$-symmetric. The moment bound in \cref{lem:quartic-analytic-inputs}(ii) holds.
    \item The Fourier-gradient pre-form \eqref{eq:pre-form} is closable. Its closure $(\cE,\operatorname{Dom}(\cE))$ is the canonical form of the SPDE semigroup, satisfies \cref{ass:equilibrium,ass:diffusion-calculus}, and has the closed-gradient representation \eqref{eq:quartic-canonical-gradient}. In particular, Fourier cylinders form a core and the closed-form semigroup is the SPDE transition semigroup on $L^2(\kappa)$.
    \item If $A_0=-\partial_x^2+m_0$ with $m_0>0$, then $\mu_0=\cN(0,A_0^{-1})$ has finite relative entropy with respect to both $\kappa_4$ and $\kappa_{\mathrm{AC}}$. Moreover, for $\lambda_{\mathrm{init}}\ge0$, the quartic Gibbs law with coupling $\lambda_{\mathrm{init}}$ has finite relative entropy with respect to $\kappa_4$.
\end{enumerate}
\end{proposition}

\begin{proof}
\par\noindent\textbf{Step 1. Established dynamics and verification of the form hypotheses.}\par
Parts~(i)--(ii) are \cref{lem:quartic-analytic-inputs}(i)--(ii). Part~(iii) is \cref{lem:quartic-direct-energy}. The SPDE is therefore the given evolution to which the abstract entropy and reversibility results apply.

\par\noindent\textbf{Step 2. Finite-entropy quenches.}\par
Let $Q_0=A_0^{-1}$, $A_0=-\partial_x^2+m_0$ and $\lambda_k^0=k^2+m_0$. Since $m,m_0>0$,
\begin{equation*}
    \operatorname{Ran}Q_0^{1/2} = H^1(\T) = \operatorname{Ran}Q^{1/2}.
\end{equation*}
Moreover, $Q^{-1/2}Q_0Q^{-1/2}-I$ is diagonal in the Fourier basis, with eigenvalues
\begin{equation*}
    r_k-1
    =\frac{\lambda_k}{\lambda_k^0}-1
    =\frac{m-m_0}{k^2+m_0}=O(k^{-2}).
\end{equation*}
The relative covariance perturbation is therefore Hilbert--Schmidt. The centred Feldman--H\'ajek theorem gives $\mu_0\sim\mu_Q$ on $\mathbb H$. The inclusion $\mathbb X=C(\T)\hookrightarrow\mathbb H$ is a Borel isomorphism onto its image, which has full measure under both Gaussian laws, so the equivalence also holds on $\mathbb X$. Since the Fourier coordinates generate $\mathcal B(\mathbb X)$, \cref{lem:entropy-projections} applies to $\mu_0\ll\mu_Q$; the finite-dimensional Gaussian relative entropy formula then gives
\begin{equation}\label{eq:quenched-Gaussian-entropy}
    \RelEnt{\mu_0}{\mu_Q} = \frac12\sum_{k\ge0}d_k \bigl[ r_k-1 - \log r_k \bigr] < \infty,
\end{equation}
because the summand is $O((r_k-1)^2)=O(k^{-4})$. For either quartic potential,
\begin{equation}\label{eq:quenched-Gibbs-entropy}
    \RelEnt{\mu_0}{\kappa} = \RelEnt{\mu_0}{\mu_Q} + \E_{\mu_0}[V]+\log Z.
\end{equation}
By \cref{lem:Gaussian-regularity}, both $\E_{\mu_0}\|\phi\|_{L^4}^4$ and $\E_{\mu_0}\|\phi\|_{L^2}^2$ are finite. Thus the middle term in \eqref{eq:quenched-Gibbs-entropy} is finite for $V_4$ and for $V_{\mathrm{AC}}$, while part~(i) gives $|\log Z|<\infty$. This proves the finite-entropy assertion for the free Gaussian initial laws.

For the interaction quench, define, for $\alpha\ge0$,
\begin{equation*}
    Z_\alpha
    = \int_{\mathbb X}\exp\!\left( -\frac\alpha4\int_\T\phi^4\,dx \right) d\mu_Q,
    \qquad
    \kappa_\alpha
    = Z_\alpha^{-1}\exp\!\left( -\frac\alpha4\int_\T\phi^4\,dx \right) \mu_Q.
\end{equation*}
If the initial law is $\kappa_{\lambda_{\mathrm{init}}}$ with $\lambda_{\mathrm{init}}\ge0$, then
\begin{equation*}
    \RelEnt{\kappa_{\lambda_{\mathrm{init}}}}{\kappa_\lambda}
    = \frac{\lambda-\lambda_{\mathrm{init}}}{4} \E_{\kappa_{\lambda_{\mathrm{init}}}} \int_\T\phi^4\,dx
    + \log\frac{Z_\lambda}{Z_{\lambda_{\mathrm{init}}}}.
\end{equation*}
The expectation is finite because the density of $\kappa_{\lambda_{\mathrm{init}}}$ with respect to $\mu_Q$ is bounded and \cref{lem:Gaussian-regularity} supplies the fourth moment. Both normalising constants lie in $(0,\infty)$, so the displayed relative entropy is finite. This completes part~(iv).
\end{proof}

\section{Gaussian Path-Space Comparison}
\label{app:path-space-details}

This appendix first identifies the closed-gradient energy of Gaussian density ratios and then establishes the path-space change of measure used in the stationary forward--reverse comparison.

\begin{lemma}[Closed-gradient energy of a Gaussian density]\label{lem:gaussian-density-energy}
Let $Q$ be positive, injective and trace class on $\mathbb H$, with an orthonormal eigenbasis $Qe_j=q_je_j$, and let $\mu_Q=\cN(0,Q)$. Suppose that $\nu\sim\mu_Q$ is centred Gaussian with independent coordinates of variances $r_jq_j$, where $r_j>0$, and set $h=d\nu/d\mu_Q$. If
\begin{equation*}
    \mathcal I :=\sum_{j\ge0}\frac{(r_j-1)^2}{r_j} < \infty,
\end{equation*}
then the closed Gaussian gradient form $\cE_Q(u,u)=\int\|D_{\mathbb K}u\|_{\mathbb K}^2\,d\mu_Q$ satisfies
\begin{equation*}
    \sqrt h\in\operatorname{Dom}(\cE_Q),
    \qquad 4\cE_Q(\sqrt h,\sqrt h)=\mathcal I.
\end{equation*}
Here $D_{\mathbb K}$ is the closure of the cylinder Cameron--Martin gradient and $\mathbb K=Q^{1/2}\mathbb H$. In the $\mathbb H$-coordinates $D_Q=Q^{-1/2}D_{\mathbb K}$ used in \cref{prop:circulation-semigroup}, the gradient is
\begin{equation}\label{eq:gaussian-density-gradient}
    D_Q\sqrt h=\frac12\sqrt h\,G,\qquad
    G(\phi)=\sum_{j\ge0}\frac{r_j-1}{\sqrt{q_j}r_j}
    \langle\phi,e_j\rangle_{\mathbb H}e_j,
\end{equation}
where the series converges in $L^2(\nu;\mathbb H)$.
\end{lemma}

\begin{proof}
\par\noindent\textbf{Step 1. Finite-coordinate densities and cutoffs.}\par
Write $x_j=\langle\phi,e_j\rangle_{\mathbb H}$, $f_j=\sqrt{q_j}e_j$ and $\mathcal F_N=\sigma(x_0,\ldots,x_N)$. The $f_j$ form an orthonormal basis of $\mathbb K$. By the product Gaussian formula,
\begin{equation*}
    h_N
    := \mathbb E_{\mu_Q}[h\mid\mathcal F_N]
    = \prod_{j=0}^Nr_j^{-1/2}\exp\left( \frac{r_j-1}{2r_jq_j}x_j^2 \right).
\end{equation*}
Set $g_N=\sqrt{h_N}$ and $a_j=(r_j-1)/(\sqrt{q_j}r_j)$. Finite-dimensional differentiation gives
\begin{equation*}
    \nabla_{\mathbb K}g_N
    = \frac12g_N\sum_{j=0}^Na_jx_jf_j,\qquad
    \int\|\nabla_{\mathbb K}g_N\|_{\mathbb K}^2\,d\mu_Q
    = \frac14 \mathcal I_N,\quad
    \mathcal I_N:=\sum_{j=0}^N\frac{(r_j-1)^2}{r_j}.
\end{equation*}
To put this possibly unbounded cylinder in the closed domain, choose smooth cutoffs $\chi_R$ on $\operatorname{span}\{e_0,\ldots,e_N\}$, equal to one on the ball of radius $R$, zero outside the ball of radius $2R$, and with $\|\nabla\chi_R\|\le C/R$. Then $\chi_Rg_N$ belongs to the bounded smooth cylinder core,
\begin{align*}
    \chi_Rg_N&\longrightarrow g_N &&\text{in }L^2(\mu_Q),\\
    \nabla_{\mathbb K}(\chi_Rg_N)
    &=\chi_R\nabla_{\mathbb K}g_N+g_N\nabla_{\mathbb K}\chi_R
    \longrightarrow\nabla_{\mathbb K}g_N
    &&\text{in }L^2(\mu_Q;\mathbb K).
\end{align*}
Indeed, dominated convergence handles the first gradient term, and the second has norm at most $C\|Q\|^{1/2}\|g_N\|_{L^2(\mu_Q)}/R=C\|Q\|^{1/2}/R$. Closedness proves that $g_N$ has the displayed closed gradient and energy.

\par\noindent\textbf{Step 2. Convergence in the closed-gradient norm.}\par
The conditional expectations $h_N$ converge to $h$ in $L^1(\mu_Q)$, and
\begin{equation*}
    \|g_N-\sqrt h\|_{L^2(\mu_Q)}^2 \le \|h_N-h\|_{L^1(\mu_Q)}\longrightarrow0.
\end{equation*}
For $M>N$, write $g_M=g_N\rho_{N,M}$, where $\rho_{N,M}$ depends only on coordinates $N+1,\ldots,M$. Independence gives $\int\rho_{N,M}^2\,d\mu_Q=1$ and $4\int \| \nabla_{\mathbb K}\rho_{N,M} \|_{\mathbb K}^2\,d\mu_Q=\mathcal I_M-\mathcal I_N$. Moreover,
\begin{equation*}
    D_{\mathbb K}g_M-D_{\mathbb K}g_N
    = (\rho_{N,M}-1)D_{\mathbb K}g_N + g_N\nabla_{\mathbb K}\rho_{N,M}.
\end{equation*}
The two summands take values in orthogonal coordinate subspaces of $\mathbb K$. Factoring their squared integrals over the independent coordinates gives
\begin{align*}
    \|D_{\mathbb K}g_M-D_{\mathbb K}g_N\|_{L^2(\mu_Q;\mathbb K)}^2
    &=\frac{\mathcal I_N}{4}\int(\rho_{N,M}-1)^2\,d\mu_Q
    +\frac{\mathcal I_M-\mathcal I_N}{4}\\
    &=\frac{\mathcal I_N}{4}\|g_M-g_N\|_{L^2(\mu_Q)}^2
    +\frac{\mathcal I_M-\mathcal I_N}{4}
    \longrightarrow 0.
\end{align*}
Thus $(g_N)$ is Cauchy in the graph norm. Closedness gives $\sqrt h\in\operatorname{Dom}(\cE_Q)$, and strong convergence of its gradients yields $4\cE_Q(\sqrt h,\sqrt h)=\lim_N \mathcal I_N=\mathcal I$. The series defining $G$ is Cauchy in $L^2(\nu;\mathbb H)$ because the sum of its coordinate second moments is $\mathcal I$. To identify the gradient, take a subsequence along which both $g_N\to\sqrt h$ and $D_Qg_N\to D_Q\sqrt h$ hold $\mu_Q$-almost everywhere. For each fixed $j$ and all $N\ge j$,
\begin{equation*}
    \langle D_Qg_N,e_j\rangle_{\mathbb H}
    =\frac12a_jx_jg_N.
\end{equation*}
Passing to the limit for the countable coordinate family gives $\langle D_Q\sqrt h,e_j\rangle_{\mathbb H}=\frac12a_jx_j\sqrt h$ for every $j$ on a common set of full measure. Since $\nu\sim\mu_Q$ and $\sqrt h\,G\in L^2(\mu_Q;\mathbb H)$, these coordinate identities prove \eqref{eq:gaussian-density-gradient}.
\end{proof}

\begin{lemma}[Stationary Gaussian path-space change of measure]
\label{lem:stationary-gaussian-girsanov}
Assume~\eqref{eq:B-assumptions}, and let $\Pbb_{B,T}$ and $\Pbb_{-B,T}$ be the stationary path laws on $C([0,T];\mathbb H)$ of \eqref{eq:linear} with drift operators $-(I+QB)$ and $-(I-QB)$, respectively. Then, for every $T>0$,
\begin{equation*}
    \Pbb_{B,T}\sim\Pbb_{-B,T},
\end{equation*}
and, under $\Pbb_{-B,T}$,
\begin{equation*}
    \frac{d\Pbb_{B,T}}{d\Pbb_{-B,T}}
    = \exp\!\left\{
    -\sqrt{2}\int_0^T\!\langle Q^{1/2}B\phi_t,dW_t^{-B}\rangle_{\mathbb H}
    -\int_0^T\!\|Q^{1/2}B\phi_t\|_{\mathbb H}^2\,dt
    \right\}.
\end{equation*}
The stochastic integral has a Borel version on the canonical path space.
\end{lemma}

\begin{proof}
\par\noindent\textbf{Step 1. Cameron--Martin drift and energy.}\par
The drift difference is
\begin{equation*}
    b_B(\phi)-b_{-B}(\phi)=-2QB\phi = \sqrt{2}Q^{1/2}u(\phi),\qquad
    u(\phi)=-\sqrt{2}Q^{1/2}B\phi.
\end{equation*}
Thus it lies in the noise Cameron--Martin directions. Furthermore,
\begin{equation}\label{eq:Girsanov-energy}
    \int_{\mathbb H}\|Q^{1/2}B\phi\|_{\mathbb H}^2\,d\mu_Q = \Tr(B^*QBQ) < \infty.
\end{equation}
\par\noindent\textbf{Step 2. Canonical stochastic integral.}\par
On $C([0,T];\mathbb H)$, let $\phi_t(\omega)=\omega_t$ and let $(\mathcal F_t)$ be the usual augmentation under $\Pbb_{-B,T}$ of the filtration generated by $(\phi_t)$. We first construct the stochastic integral as a functional of this canonical trajectory. Enumerate an orthonormal eigenbasis of $Q$ as $(e_j)_{j\ge0}$, with $Qe_j=q_je_j$. Under $\Pbb_{-B,T}$ define
\begin{equation}\label{eq:canonical-brownian-coordinate}
    \mathcal W_j^\omega(t)=\frac1{\sqrt{2q_j}}
    \left[
    \langle\omega_t-\omega_0,e_j\rangle_{\mathbb H} + \int_0^t\langle \mathsf K_{-B}\omega_s,e_j\rangle_{\mathbb H}\,ds
    \right].
\end{equation}
Taking the $e_j$-coordinate of the integral equation under $\Pbb_{-B,T}$ gives
\begin{equation*}
    \langle\phi_t-\phi_0,e_j\rangle_{\mathbb H} + \int_0^t\langle\mathsf K_{-B}\phi_s,e_j\rangle_{\mathbb H}\,ds = \sqrt{2q_j}\,\mathcal W_j(t).
\end{equation*}
Thus each $\mathcal W_j$ is a continuous local martingale starting from zero. The coordinate martingale brackets of the $Q$-Wiener noise give
\begin{equation*}
    \langle\mathcal W_i,\mathcal W_j\rangle_t = \delta_{ij}t.
\end{equation*}
For every finite set of indices, the multidimensional L\'evy characterisation \cite{DaPratoZabczyk2014} therefore shows that the corresponding $\mathcal W_j$ are independent standard Brownian motions in $(\mathcal F_t)$. Let $\Pi_M$ be the orthogonal projection onto $\operatorname{span}\{e_0,\ldots,e_M\}$ and set
\begin{equation*}
    I_M=\sum_{j=0}^M\int_0^T \langle Q^{1/2}B\phi_t,e_j\rangle_{\mathbb H}\,d\mathcal W_j(t).
\end{equation*}
To see explicitly that $I_M$ is a canonical-path functional, let $t_k^n=kT2^{-n}$ and define the adapted dyadic sums
\begin{equation*}
    S_{M,n}(\omega) = \sum_{j=0}^M\sum_{k=0}^{2^n-1}\langle Q^{1/2}B\omega_{t_k^n},e_j\rangle_{\mathbb H}
    \times\bigl[\mathcal W_j^\omega(t_{k+1}^n) - \mathcal W_j^\omega(t_k^n)\bigr].
\end{equation*}
The integrand is adapted and square-integrable by \eqref{eq:Girsanov-energy}. For $0\le s\le t\le T$, the integral equation and stationarity give
\begin{align*}
    \E_{-B}\|\phi_t-\phi_s\|_{\mathbb H}^2
    &\le2(t-s)\int_s^t\E_{-B}\|\mathsf K_{-B}\phi_r\|_{\mathbb H}^2\,dr + 4(t-s)\Tr Q\\
    &\le2\|\mathsf K_{-B}\|^2\Tr Q\,(t-s)^2+4\Tr Q\,(t-s).
\end{align*}
Applying the bounded operator $Q^{1/2}B$ and It\^o's isometry shows that the mean-square error of the left-endpoint sums tends to zero with the mesh. Hence $S_{M,n}\to I_M$ in $L^2(\Pbb_{-B,T})$. Each $S_{M,n}$ is Borel measurable on $C([0,T];\mathbb H)$ by \eqref{eq:canonical-brownian-coordinate}; after taking an almost surely convergent subsequence, $I_M$ therefore has a Borel representative. For $M_2>M_1$, It\^o's isometry and stationarity give
\begin{align*}
    \E_{-B}|I_{M_2}-I_{M_1}|^2
    &=\E_{-B}\int_0^T \|(\Pi_{M_2}-\Pi_{M_1})Q^{1/2}B\phi_t\|_{\mathbb H}^2\,dt\\
    &=T\int_{\mathbb H} \|(\Pi_{M_2}-\Pi_{M_1})Q^{1/2}B\phi\|_{\mathbb H}^2\,d\mu_Q(\phi)
    \longrightarrow 0.
\end{align*}
by \eqref{eq:Girsanov-energy} and monotone convergence in the orthonormal basis. Thus $(I_M)$ is Cauchy in $L^2(\Pbb_{-B,T})$. Denote its limit by
\begin{equation*}
    I = \int_0^T\langle Q^{1/2}B\phi_t,dW_t^{-B}\rangle_{\mathbb H}.
\end{equation*}
An $L^2$ limit of Borel random variables on the standard Borel space $C([0,T];\mathbb H)$ has a Borel representative, which we fix for $I$. It\^o's isometry identifies this limit with the Hilbert-space It\^o integral of the predictable integrand $Q^{1/2}B\phi$ in $L^2(dt\otimes\Pbb_{-B,T};\mathbb H)$.

Using another eigenbasis of $Q$, or another increasing family $(R_M)$ of finite-rank orthogonal projections converging strongly to the identity, gives the same limit $I$, $\Pbb_{-B,T}$-almost surely. Indeed, strong convergence and contractivity give $R_MQ^{1/2}B\phi\to Q^{1/2}B\phi$ in $L^2(dt\otimes\Pbb_{-B,T};\mathbb H)$ by dominated convergence and \eqref{eq:Girsanov-energy}. The corresponding integrals, taken against the same canonical cylindrical Brownian motion $W^{-B}$, therefore converge to $I$ in $L^2(\Pbb_{-B,T})$ by It\^o's isometry. Consequently the exponential in \eqref{eq:path-RN-B} is a well-defined Borel functional of the canonical path.

\par\noindent\textbf{Step 3. Short-time change of measure.}\par
We next prove that the stochastic exponential is a true martingale on an arbitrary finite interval. Under $\mu_Q$, the random variable $Q^{1/2}B\phi$ is a centred Gaussian variable in $\mathbb H$ with covariance operator
\begin{equation*}
    Q^{1/2}BQB^*Q^{1/2}.
\end{equation*}
Its trace equals $\|Q^{1/2}BQ^{1/2}\|_{\mathrm{HS}}^2<\infty$, so this is a Gaussian Radon random variable on the separable Hilbert space $\mathbb H$. Applying the Fernique estimate in \cref{lem:gaussian-measure-facts} with $\mathbb E=\mathbb H$ gives an $\varepsilon>0$ such that
\begin{equation*}
    \int_{\mathbb H}\exp\!\bigl(\varepsilon
    \|Q^{1/2}B\phi\|_{\mathbb H}^2\bigr)\,d\mu_Q(\phi)<\infty.
\end{equation*}
On an interval of length $\delta<\varepsilon$, Jensen's inequality and stationarity imply
\begin{align*}
    \E_{-B}\exp\!\left(\frac12\int_0^\delta\|u(\phi_t)\|_{\mathbb H}^2dt\right)
    &=\E_{-B}\exp\!\left(\int_0^\delta \|Q^{1/2}B\phi_t\|_{\mathbb H}^2dt\right)\\
    &\le\frac1\delta\int_0^\delta \E_{-B}\exp\!\left( \delta\|Q^{1/2}B\phi_t\|_{\mathbb H}^2 \right)dt<\infty.
\end{align*}
The process $u(\phi_t)$ is predictable, and the displayed estimate is exactly Novikov's condition on every stationary interval of length at most $\delta$. The corresponding stochastic exponential is therefore a uniformly integrable martingale on each such interval, by Novikov's criterion; we use the cylindrical Girsanov theorem in \cite[Chapter~10]{DaPratoZabczyk2014}.

\par\noindent\textbf{Step 4. Extension to an arbitrary time interval.}\par
Choose a partition $0=t_0<t_1<\cdots<t_n=T$ with $t_i-t_{i-1}\le\delta$. We construct the change of measure inductively. Set $\Pbb^{(0)}=\Pbb_{-B,T}$. Suppose that $\Pbb^{(i-1)}$ has already been defined so that the process follows the $B$ dynamics on $[0,t_{i-1}]$ and the $(-B)$ dynamics after $t_{i-1}$, and the law of $\phi_{t_{i-1}}$ is $\mu_Q$. Since the density used up to this stage is $\mathcal F_{t_{i-1}}$-measurable, the conditional law of the future given $\mathcal F_{t_{i-1}}$ is still the $(-B)$ dynamics started from $\phi_{t_{i-1}}$. Consequently,
\begin{equation*}
    \mathbb E_{\Pbb^{(i-1)}} \exp\!\left(\frac12\int_{t_{i-1}}^{t_i} \|u(\phi_s)\|_{\mathbb H}^2\,ds\right)
    = \int_{\mathbb H}\mu_Q(dx)\,\mathbb E_{-B,x}\exp\!\left(\frac12\int_0^{t_i-t_{i-1}}\|u(\phi_s)\|_{\mathbb H}^2\,ds\right)
    < \infty.
\end{equation*}
Here $\E_{-B,x}$ denotes expectation for the $(-B)$ dynamics started from $x$. The equality uses the $\mu_Q$ marginal at $t_{i-1}$; its right-hand side is the stationary expectation estimated above.

Let $W^{(i-1)}$ denote the cylindrical Brownian motion under $\Pbb^{(i-1)}$ on the interval $[t_{i-1},t_i]$. Since all preceding density increments are $\mathcal F_{t_{i-1}}$-measurable, its differentials on this interval agree with those of $W^{-B}$. Define, for $t\in[t_{i-1},t_i]$,
\begin{equation*}
    Z_i(t)=\exp\!\bigg\{\int_{t_{i-1}}^t \langle u(\phi_s),dW_s^{(i-1)}\rangle_{\mathbb H} - \frac12\int_{t_{i-1}}^t \|u(\phi_s)\|_{\mathbb H}^2\,ds\bigg\}.
\end{equation*}
The preceding estimate and Novikov's criterion show that $(Z_i(t))$ is a uniformly integrable martingale. In particular,
\begin{equation*}
    \mathbb E_{\Pbb^{(i-1)}} [Z_i(t_i)\mid\mathcal F_{t_{i-1}}] = 1.
\end{equation*}
Define
\begin{equation*}
    \frac{d\Pbb^{(i)}}{d\Pbb^{(i-1)}}=Z_i(t_i).
\end{equation*}
The integrand $u(\phi_s)$ is predictable and square integrable, and $Z_i$ is a true martingale by the interval-wise Novikov estimate. Girsanov's theorem therefore changes the drift from $-(I-QB)\phi$ to $-(I+QB)\phi$ on $[t_{i-1},t_i]$. The conditional expectation above shows that the law on $\mathcal F_{t_{i-1}}$ is unchanged. Since both the $B$ and the $(-B)$ dynamics preserve $\mu_Q$, the endpoint marginal under $\Pbb^{(i)}$ is again $\mu_Q$. This proves the induction hypothesis at the next step.

Writing $Z_i=Z_i(t_i)$, the tower property now gives
\begin{equation*}
    \mathbb E_{-B}[Z_1\cdots Z_n] = \mathbb E_{\Pbb^{(n-1)}}[Z_n]=1,
\end{equation*}
and the same identity at every earlier induction step shows that each $\Pbb^{(i)}$ is a probability measure. Because the intervals are disjoint, multiplication of the interval exponentials gives the single Dol\'eans exponential
\begin{equation*}
    \prod_{i=1}^nZ_i
    = \exp\!\left\{\int_0^T\langle u(\phi_t),dW_t^{-B}\rangle_{\mathbb H}
    - \frac12\int_0^T\|u(\phi_t)\|_{\mathbb H}^2\,dt\right\}.
\end{equation*}
Substituting $u=-\sqrt2Q^{1/2}B\phi$ yields exactly \eqref{eq:path-RN-B}. Under $\Pbb^{(n)}$, the canonical process starts from $\mu_Q$ and the successive Girsanov changes show that it satisfies \eqref{eq:linear} with drift $-(I+QB)\phi$ on all of $[0,T]$ and with noise coefficient $\sqrt2Q^{1/2}$. By \cref{prop:skew-invariance}, this linear equation has a global strong solution and is pathwise unique; hence it is unique in law. The transformed canonical law is therefore $\Pbb^{(n)}=\Pbb_{B,T}$.

If $B$ and $-B$ are interchanged, the Cameron--Martin integrand changes sign, while its square norm, the Fernique estimate and every interval-wise Novikov bound are unchanged. The same argument therefore gives $\Pbb_{-B,T}\ll\Pbb_{B,T}$ and proves mutual absolute continuity.

\end{proof}

\section{Numerical Methods}
\label{app:numerics}

This appendix describes the finite-mode calculations in \S\ref{sec:examples}. The quartic experiment estimates entropy relaxation from stationary correlations, while the Gaussian experiments use exact transition sampling and explicit quadratic observables.

\subsection{Quartic Gibbs Dynamics}
\label{app:numerics-phi4}

Using the Fourier basis in \eqref{eq:basis}, write $q_k=(k^2+m)^{-1}$ and introduce the whitened coordinates
\begin{equation*}
    \phi_N(z)=\sqrt{q_0}z_0e_0
    +\sum_{k=1}^N\sqrt{q_k}\bigl(z_k^ce_k^c+z_k^se_k^s\bigr),
    \qquad z\in\R^{2N+1}.
\end{equation*}
The finite-dimensional Gibbs law and its potential are
\begin{equation*}
    \nu_N(dz)=\widetilde Z_N^{-1}e^{-U_N(z)}\,dz,
    \qquad U_N(z)=\frac12|z|^2+W_N(z),
    \qquad W_N(z)=\frac\lambda4\int_\T\phi_N(z)^4\,dx.
\end{equation*}
Thus $\nu_N$ is the whitened low-mode marginal of the measure $\kappa_{4,N}$ used in the proof of \cref{thm:Phi4-model}. The corresponding continuous dynamics is $dZ_t=-[Z_t+\nabla W_N(Z_t)]\,dt+\sqrt2\,dW_t$, with standard $(2N+1)$-dimensional Brownian motion $W_t$. We use $m=\lambda=1$, $N=32$ and $256$ equally spaced spatial points. Since $256>4N$, the trapezoidal rule integrates the quartic Fourier polynomial and its coordinate derivatives exactly, up to roundoff.

To estimate the relaxation from $h_0=1+\varepsilon f$, where $f(z)=\tanh z_0$ and $\varepsilon=0.2$, we express the entropy through stationary correlations. Let $P_t^{(N)}$ denote the continuous semigroup and define
\begin{equation*}
    H_N(t)=\Ent_{\nu_N}\bigl(P_t^{(N)}h_0\bigr),
    \qquad \mathcal C_N(s)=\E_{\nu_N}[f(Z_0)f(Z_s)],
    \qquad R_N(t)=\frac{\mathcal C_N(2t)}{\mathcal C_N(0)}.
\end{equation*}
The even potential preserves oddness, and reversibility gives
\begin{equation*}
    P_t^{(N)}h_0=1+\varepsilon P_t^{(N)}f,
    \qquad \mathcal C_N(2t)=\|P_t^{(N)}f\|_{L^2(\nu_N)}^2.
\end{equation*}
Define the even function
\begin{equation*}
    S(u)=\frac12\bigl[(1+u)\log(1+u)+(1-u)\log(1-u)\bigr].
\end{equation*}
Pairing $z$ with $-z$ yields $H_N(t)=\int S(\varepsilon P_t^{(N)}f)\,d\nu_N$. Since
\begin{equation*}
    S(u)=\sum_{r=1}^\infty\frac{u^{2r}}{(2r)(2r-1)},
    \qquad
    \frac{\varepsilon^2v^2}{2}\le S(\varepsilon v)\le S(\varepsilon)v^2
    \quad\text{for }|v|\le1,
\end{equation*}
the Markov bound $|P_t^{(N)}f|\le1$ implies
\begin{equation}\label{eq:numerical-entropy-comparison}
    K_\varepsilon^{-1}R_N(t)
    \le\frac{H_N(t)}{H_N(0)}\le K_\varepsilon R_N(t),
    \qquad K_\varepsilon=\frac{2S(\varepsilon)}{\varepsilon^2}
    =1.00677568\ldots.
\end{equation}
The finite-dimensional log-Sobolev bound in the proof of \cref{thm:Phi4-model} gives $H_N(t)/H_N(0)\le e^{-2t}$, while its Poincar\'e consequence gives $R_N(t)\le e^{-2t}$ because $\int f\,d\nu_N=0$.

To evaluate the stationary correlations, we generate independent Gibbs initial states by proposing $z\sim\cN(0,I)$ and accepting with probability $e^{-W_N(z)}$. We then evolve these states with a Gaussian proposal retaining the exact free Ornstein--Uhlenbeck step, within the Gaussian-reference sampling setting discussed in \cite{CotterRobertsStuartWhite2013}. Specifically, with $a=e^{-\delta}$, we propose
\begin{equation}\label{eq:numerical-phi4-proposal}
    y=az-(1-a)\nabla W_N(z)+\sqrt{1-a^2}\,\xi,
    \qquad \xi\sim\cN(0,I),
\end{equation}
and accept it with the Metropolis--Hastings probability \cite{Hastings1970}
\begin{equation*}
    \alpha_\delta(z,y)=1\wedge
    \frac{e^{-U_N(y)}q_\delta(y,z)}{e^{-U_N(z)}q_\delta(z,y)},
\end{equation*}
where $q_\delta$ is the Gaussian proposal density in \eqref{eq:numerical-phi4-proposal}. Otherwise the state remains at $z$. The resulting kernel $\mathsf P_\delta$ is reversible with respect to $\nu_N$, and every step, including a rejection, advances the numerical time by $\delta$. Since this kernel also preserves reflection symmetry, \eqref{eq:numerical-entropy-comparison} holds at integer times for its own entropy and correlation functions.

We estimate $R_N(t)$ by the ratio of sample means of $f(Z_0)f(Z_{2t})$ and $f(Z_0)^2$ over $40{,}000$ independent Gibbs-started trajectories. Figure~\ref{fig:phi4-entropy-relaxation} uses $\delta=0.0025$ and $0\le t\le0.45$. A joint delete-one-batch jackknife for the numerator and denominator, using $160$ independent batches of $250$ trajectories, gives standard errors of at most $2.67\%$ of the plotted ratio. Comparisons use $N=16,32$ and $\delta=0.01,0.005,0.0025$; the finest two steps are coupled through their Gaussian increments. Over the displayed interval, their absolute difference plus twice its paired standard error is at most $0.58\%$ of the finer-step estimate. These comparisons assess time-discretisation effects separately from the entropy approximation in \eqref{eq:numerical-entropy-comparison}.

\subsection{Two-Mode Gaussian Dynamics}
\label{app:numerics-gaussian}

The modal process $X_t$ in \cref{sec:two-mode-circulation} satisfies the exact projected equation
\begin{equation*}
    dX_t=-(I+\mathsf Q^{(2)}B_c)X_t\,dt
    +\sqrt{2\mathsf Q^{(2)}}\,dW_t,
    \qquad X_0\sim\cN(0,\mathsf C_0^{(2)}),
\end{equation*}
where $\mathsf Q^{(2)}=\operatorname{diag}(1,1/2)$, $\mathsf C_0^{(2)}=\operatorname{diag}(1/2,1/3)$, $B_c=c\left(\begin{smallmatrix}0&-1\\1&0\end{smallmatrix}\right)$ and $W$ is standard two-dimensional Brownian motion. The invariance of $\operatorname{span}\{e_j,e_k\}$ under $Q$ and $B$ makes the projected dynamics closed. The stationary density $\pi_2$ and current $-\mathsf Q^{(2)}B_1x\,\pi_2(x)$ are evaluated analytically for Figure~\ref{fig:gaussian-stationary-current}.

For the transient calculations, the whitened coordinates $Y=(\mathsf Q^{(2)})^{-1/2}X$ and the Ornstein--Uhlenbeck transition formula \cite{DaPratoZabczyk2014} give the exact update
\begin{equation}\label{eq:numerical-gaussian-transition}
    Y_{n+1}=e^{-\Delta}R(-c\Delta/\sqrt2)Y_n
    +\sqrt{1-e^{-2\Delta}}\,\xi_n,
\end{equation}
where $R(\theta)$ is the counterclockwise rotation matrix and the $\xi_n$ are independent standard Gaussian vectors. We use $\Delta=0.05$, $0\le t\le1.5$ and $400{,}000$ independent trajectories for each $c\in\{0,1\}$, grouped into $80$ batches. Here $\Delta$ is an observation interval, so \eqref{eq:numerical-gaussian-transition} introduces no time-discretisation bias.

With $M=400{,}000$, we estimate the covariance using the known zero mean and then apply the Gaussian relative-entropy formula:
\begin{equation*}
    \begin{aligned}
        \widehat{\mathsf C}_t&=\frac1M\sum_{r=1}^M X_t^{(r)}(X_t^{(r)})^T,\\
        \widehat H_2(t)&=\frac12\left\{
        \Tr\bigl((\mathsf Q^{(2)})^{-1}\widehat{\mathsf C}_t\bigr)-2
        -\log\det\bigl((\mathsf Q^{(2)})^{-1}\widehat{\mathsf C}_t\bigr)\right\}.
    \end{aligned}
\end{equation*}
The analytic curve is obtained from the exact covariance $\mathsf C_t^{(2)}$. For $c=1$, put $D_t=(\mathsf Q^{(2)})^{-1}-(\mathsf C_t^{(2)})^{-1}$, so that $D_tx$ is the exact logarithmic gradient of the relative density. Using the same samples, we estimate the marginal rates $\sigma_{\rm na}^{(2)}$, $\sigma_{\rm hk}^{(2)}$ and $\sigma_{\rm tot}^{(2)}$ by the respective sample means of
\begin{equation*}
    |(\mathsf Q^{(2)})^{1/2}D_tX_t|^2,
    \qquad |(\mathsf Q^{(2)})^{1/2}B_1X_t|^2,
    \qquad |(\mathsf Q^{(2)})^{1/2}(D_t+B_1)X_t|^2.
\end{equation*}
The total rate is thus estimated from its own quadratic form rather than by adding the other two estimates. Independent batches give standard errors for these quadratic averages, while a delete-one-batch jackknife estimates the uncertainty of $\widehat H_2(t)$.


\end{document}